 \documentclass[12pt]{amsart}
 \usepackage{amsmath,amsfonts,mathrsfs,amsthm,amssymb}
 \usepackage{amscd}
 \usepackage{amssymb}
 \usepackage{multirow}
 \usepackage{framed}
 \usepackage{enumerate}
 \usepackage{array}
 \usepackage{color}
 \usepackage[usenames,dvipsnames,svgnames,table]{xcolor}
 \usepackage{mathtools}
 \usepackage{enumitem}
 \usepackage{mathabx}
 \usepackage{subdepth}
 \usepackage[normalem]{ulem}
 \DeclarePairedDelimiter{\floor}{\lfloor}{\rfloor}
 \DeclarePairedDelimiter{\ceil}{\lceil}{\rceil}

 \usepackage{hyperref}
 \usepackage{tikz}
 \usepackage{ifthen}
 \usepackage{xspace}

\def\xstringversion     {1.6}
\def\xstringdate        {2012/10/24}

\edef\CurrentAtCatcode  {\the\catcode`\@}
\catcode`\@=11
\newwrite\@xs@message% canal pour les messages
\newcount\integerpart\newcount\decimalpart% compteurs utilis\'es par xstring
\newif\if@xs@empty

\expandafter\ifx\csname @latexerr\endcsname\relax% on n'utilise pas LaTeX ?
	\immediate\write\m@ne{Package: xstring \xstringdate\space\space v\xstringversion\space\space String manipulations (C Tellechea)}%
	\long\def\@firstoftwo#1#2{#1}
	\long\def\@secondoftwo#1#2{#2}
	\long\def\@gobble#1{}
	\long\def\@ifnextchar#1#2#3{%
		\let\reserved@d=#1%
		\def\reserved@a{#2}%
		\def\reserved@b{#3}%
		\futurelet\@let@arg\@ifnch}
	\def\@ifnch{%
		\ifx\@let@arg\@sptoken
			\let\reserved@c\@xifnch
		\else
			\ifx\@let@arg\reserved@d
				\let\reserved@c\reserved@a
			\else
				\let\reserved@c\reserved@b
			\fi
		\fi
		\reserved@c}
	\def\:{\let\@sptoken= } \:
	\def\:{\@xifnch} \expandafter\def\: {\futurelet\@let@arg\@ifnch}
	\def\@ifstar#1{\@ifnextchar *{\@firstoftwo{#1}}}
	\long\def\@testopt#1#2{\@ifnextchar[{#1}{#1[{#2}]}}
	\def\@empty{}
\fi% fin des d\'efinitions LaTeX

\def\@xs@testempty#1{%
	\expandafter\ifx\expandafter\@empty\detokenize{#1}\@empty\@xs@emptytrue\else\@xs@emptyfalse\fi}

\def\@xs@MakeVerb{% lit 1 argument et le transforme en verb
	\begingroup% groupe o\`u les catcodes sont \`a 12 pour la lecture suivante
		\def\do##1{\catcode`##112\relax}%
		\dospecials% on entre dans le mode verb
		\obeyspaces% et on tient compte des espaces
		\@xs@ReadVerb% et on va lire l'argument
}

\def\setverbdelim#1{% d\'efinit quel est le d\'elimiteur de verb
	\expandafter\@xs@testempty\expandafter{\@gobble#1}%
	\if@xs@empty
	\else
		\begingroup
			\newlinechar`\^^J%
			\immediate\write\@xs@message
			{Package xstring Warning: verb delimiter is not a single token on input line \the\inputlineno^^J}%
		\endgroup
	\fi
	\def\@xs@ReadVerb##1#1##2#1{% lit ##2 qui est entre les d\'elimiteurs de verb
		\endgroup% on ferme le groupe
		\@xs@afterreadverb{##2}}% on appelle l'ex\'ecution de fin
}

\def\verbtocs#1{%
	\def\@xs@afterreadverb##1{\def#1{##1}}%
	\@xs@MakeVerb
}

\begingroup% on ouvre un groupe o\`u...
	\catcode\z@12 % ...le caract\`ere 0 a 12 pour catcode
	\gdef\tokenize#1#2{%
		\begingroup
			\@xs@def\@xs@reserved@A{#2}% on d\'eveloppe en accord avec \fullexpandarg ou \noexpandarg
			\def\@xs@AssignResult^^00##1^^00\@xs@nil{\gdef#1{##1}}% on assigne en tenant compte du \@xs@nil qui vient de la fin du fichier virtuel
			\everyeof{\@xs@nil}% met un \@xs@nil \`a la fin du fichier virtuel
			\endlinechar\m@ne
			\catcode\z@12\relax
			\expandafter\@xs@AssignResult\scantokens\expandafter{\expandafter^^00\@xs@reserved@A^^00}% on fait l'assignation
		\endgroup
	}%
\endgroup

\begingroup
	\catcode\z@3 \def\@xs@twochars{^^00}%
	\catcode\z@7 \xdef\@xs@twochars{\@xs@twochars^^00}%
\endgroup

\edef\@xs@reserved@A{\def\noexpand\@xs@AssignResult##1\@xs@twochars}
\@xs@reserved@A#2{\endgroup\expandafter\def\expandafter#2\expandafter{\@gobble#1}}
\def\tokenize#1#2{%
	\begingroup
		\@xs@def\@xs@reserved@A{#2}% on d\'eveloppe en accord avec \fullexpandarg ou \noexpandarg
		\everyeof\expandafter{\@xs@twochars#1}% met "^^@^^@#1" \`a la fin du fichier virtuel
		\endlinechar\m@ne
		\expandafter\@xs@AssignResult\scantokens\expandafter{\expandafter\relax\@xs@reserved@A}% on fait l'assignation
}%

\def\@xs@ReturnResult#1#2{%
	\def\@xs@argument@A{#1}%
	\@xs@testempty{#2}%
	\if@xs@empty
		\@xs@argument@A
	\else
		\let#2\@xs@argument@A
	\fi
}

\def\normalexpandarg{%
	\let\@xs@def\def% on d\'efinit \@xs@call avec \def
	\def\@xs@expand##1{\unexpanded\expandafter{##1}}}
\def\expandarg{%
	\let\@xs@def\def% on d\'efinit \@xs@call avec \def
	\def\@xs@expand##1{\unexpanded\expandafter\expandafter\expandafter{##1}}%
}

\def\fullexpandarg{%
	\let\@xs@def\edef% on d\'efinit\@xs@call avec \edef
	\def\@xs@expand##1{##1}% et on neutralise \@xs@expand
}

\def\saveexpandmode{\let\@xs@saved@def\@xs@defarg\let\@xs@saved@expand\@xs@expand}
\def\restoreexpandmode{\let\@xs@defarg\@xs@saved@def\let\@xs@expand\@xs@saved@expand}

\def\@xs@ifbeginwithbrace#1{%
	\csname @%
		\expandafter\@gobble\string{%
		\expandafter\@gobble\expandafter{\expandafter{\string#1}%
		\expandafter\expandafter\expandafter\expandafter\expandafter\expandafter\expandafter\expandafter\expandafter\expandafter\expandafter\expandafter\expandafter\expandafter\expandafter\@firstoftwo
		\expandafter\expandafter\expandafter\expandafter\expandafter\expandafter\expandafter\@gobble
		\expandafter\expandafter\expandafter\@gobble
		\expandafter\expandafter\expandafter{\expandafter\string\expandafter}\string}%
		\expandafter\@gobble\string}%
		\@secondoftwo{first}{second}oftwo%
	\endcsname
}

\def\@xs@returnfirstsyntaxunit#1#2{%
	\def\@xs@groupfound{\expandafter\def\expandafter#2\expandafter{\expandafter{#2}}\@xs@gobbleall}% on met #2 dans des accolades et on finit
	\def\@xs@assignfirsttok##1##2\@xs@nil{\let\@xs@toks0\def#2{##1}}%
	\def\@xs@testfirsttok{%
		\let\@xs@next\@xs@assignfirsttok
		\ifx\@xs@toks\bgroup
			\expandafter\@xs@ifbeginwithbrace\expandafter{\@xs@argument}{\def\@xs@next{\afterassignment\@xs@groupfound\def#2}}{}%
		\fi
		\@xs@next}%
	\def\@xs@argument{#1}%
	\edef\@xs@next{\expandafter\@xs@beforespace\detokenize{#1} \@xs@nil}% #1 commence par un espace ?
	\ifx\@xs@next\@empty
		\def\@xs@next{\expandafter\ifx\expandafter\@empty\detokenize\expandafter{\@xs@argument}\@empty\let#2\@empty\else\def#2{ }\let\@xs@toks0\fi}%
	\else
		\def\@xs@next{\expandafter\futurelet\expandafter\@xs@toks\expandafter\@xs@testfirsttok\@xs@argument\@xs@nil}%
	\fi
	\@xs@next
}

\def\@xs@testsecondtoken#1\@xs@nil{\@xs@ifbeginwithbrace{#1}}
\def\@xs@gobblespacebeforebrace#1#{}% supprime tout ce qui est avant la 1ere accolade ouvrante
\def\@xs@removefirstsyntaxunit#1#2{%
	\def\@xs@argument{#1}%
	\expandafter\expandafter\expandafter\ifx\expandafter\expandafter\expandafter\@empty\expandafter\@xs@beforespace\detokenize\expandafter{\@xs@argument} \@xs@nil\@empty% #1 commence par un espace ?
		\expandafter\@xs@testempty\expandafter{\@xs@argument}%
		\if@xs@empty
			\let#2\@empty
		\else
			\afterassignment\@xs@testsecondtoken% après avoir mangé le 1er token, on va tester si la suite commence par «{»
			\expandafter\let\expandafter\@xs@secontoken\expandafter=\expandafter\@sptoken\@xs@argument\@xs@@nil\@xs@nil% on mange le 1er token et on rajoute \@xs@@nil à la fin pour éviter de perdre les accolades du groupe
				{\expandafter\expandafter\expandafter\def\expandafter\expandafter\expandafter#2%
				 \expandafter\expandafter\expandafter{\expandafter\@xs@gobblespacebeforebrace\@xs@argument}}%
				{\expandafter\expandafter\expandafter\def\expandafter\expandafter\expandafter#2%
				 \expandafter\expandafter\expandafter{\expandafter\@xs@behindspace\@xs@argument\@xs@nil}}%
		\fi
	\else
		\expandafter\expandafter\expandafter\def\expandafter\expandafter\expandafter#2%
		\expandafter\expandafter\expandafter{\expandafter\@gobble\@xs@argument}%
	\fi
}

\def\@xs@beforespace#1 #2\@xs@nil{#1}
\def\@xs@behindspace#1 #2\@xs@nil{#2}
\def\@xs@returnfirstsyntaxunit@ii#1#2\@xs@nil{#1}
\def\@xs@gobbleall#1\@xs@nil{}

\def\@xs@expand@and@detokenize#1#2{%
	\def#1{#2}%
	\expandafter\edef\expandafter#1\expandafter{\@xs@expand#1}% on d\'eveloppe #2 selon le mode de d\'eveloppement
	\edef#1{\detokenize\expandafter{#1}}% puis on d\'etokenize et on assigne \`a #1
}

\def\@xs@expand@and@assign#1#2{\@xs@def#1{#2}}% on d\'eveloppe #2 selon \fullexpandarg ou \normalexpandarg

\def\@xs@edefaddtomacro#1#2{\edef#1{\unexpanded\expandafter{#1}#2}}
\def\@xs@addtomacro#1#2{\expandafter\def\expandafter#1\expandafter{#1#2}}

\def\@xs@argstring{0########1########2########3########4########5########6########7########8########9}
\def\@xs@DefArg#1{\def\@xs@defarg0##1#1##2\@xs@nil{\def\@xs@myarg{##1#1}}\expandafter\@xs@defarg\@xs@argstring\@xs@nil}
\def\@xs@DefArg@#1{\expandafter\@xs@defarg@\expandafter{\number\numexpr#1+1}}
\def\@xs@defarg@#1{\def\@xs@defarg0##11##2#1##3\@xs@nil{\def\@xs@myarg{[##11]##2#1}}\expandafter\@xs@defarg\@xs@argstring\@xs@nil}
\def\@xs@OneArg#1{\expandafter\@xs@onearg\expandafter{\number\numexpr#1-1}{#1}}
\def\@xs@onearg#1#2{\def\@xs@defarg##1#1##2#2##3\@xs@nil{\def\@xs@myarg{##2#2}}\expandafter\@xs@defarg\@xs@argstring\@xs@nil}

\def\@xs@BuildLines#1#2#3#4{%
	\let\@xs@newlines\@empty
	\let\@xs@newargs\@empty
	\def\@xs@buildlines##1{%
		\expandafter\@xs@OneArg\expandafter{\number\numexpr##1+#1-1}%
		\edef\@xs@reserved@B{\noexpand\@xs@expand\csname @xs@arg@\romannumeral\numexpr##1\endcsname}%
		\ifnum##1=\@ne% si c'est le premier argument
			\@xs@testempty{#3}%
			\if@xs@empty
				\expandafter\@xs@addtomacro\expandafter\@xs@newargs\expandafter{\expandafter{\@xs@reserved@B}}%
				\edef\@xs@reserved@B{\ifnum##1>#4 @xs@def\else @xs@assign\fi}%
			\else% et s'il y a un argument optionnel alors, on met des crochets
				\expandafter\@xs@addtomacro\expandafter\@xs@newargs\expandafter{\expandafter[\@xs@reserved@B]}%
				\def\@xs@reserved@B{@xs@def}% ne pas d\'etok\'eniser l'argument optionnel grace au \@xs@def
			\fi
		\else
			\expandafter\@xs@addtomacro\expandafter\@xs@newargs\expandafter{\expandafter{\@xs@reserved@B}}%
			\edef\@xs@reserved@B{\ifnum##1>#4 @xs@def\else @xs@assign\fi}%
		\fi
		\edef\@xs@newlines{\unexpanded\expandafter{\@xs@newlines}\expandafter\noexpand\csname\@xs@reserved@B\endcsname\expandafter\noexpand\csname @xs@arg@\romannumeral\numexpr##1\endcsname{\@xs@myarg}}%
		\ifnum##1<#2\relax
			\def\@xs@next{\expandafter\@xs@buildlines\expandafter{\number\numexpr##1+1}}%
			\expandafter\@xs@next
		\fi}%
	\@xs@buildlines\@ne
}

\def\@xs@newmacro{%
	\@ifstar
		{\let\@xs@reserved@D\@empty\@xs@newmacro@}
		{\let\@xs@reserved@D\relax\@xs@newmacro@0}%
}
\def\@xs@newmacro@#1#2#3#4#5{%
	\edef\@xs@reserved@A{@xs@\expandafter\@gobble\string#2}%
	\edef\@xs@reserved@C{\expandafter\noexpand\csname\@xs@reserved@A @\ifx\@empty#3\@empty @\fi\endcsname}%
	\edef\@xs@reserved@B{%
		\ifx\@empty\@xs@reserved@D
			\def\noexpand#2{\noexpand\@ifstar
				{\let\noexpand\@xs@assign\noexpand\@xs@expand@and@detokenize\expandafter\noexpand\@xs@reserved@C}%
				{\let\noexpand\@xs@assign\noexpand\@xs@expand@and@assign\expandafter\noexpand\@xs@reserved@C}%
			}%
		\else
			\def\noexpand#2{\let\noexpand\@xs@assign\noexpand\@xs@expand@and@assign\expandafter\noexpand\@xs@reserved@C}%
		\fi
		\ifx\@empty#3\@empty
		\else
			\def\expandafter\noexpand\@xs@reserved@C{%
				\noexpand\@testopt{\expandafter\noexpand\csname\@xs@reserved@A @@\endcsname}{\ifx\@xs@def\edef#3\else\unexpanded{#3}\fi}}%
		\fi
	}%
	\@xs@reserved@B
	\ifx\@empty#3\@empty
		\@xs@BuildLines1{#4}{#3}{#1}%
		\@xs@DefArg{#4}%
	\else
		\expandafter\@xs@BuildLines\expandafter1\expandafter{\number\numexpr#4+1}{#3}{#1}%
		\@xs@DefArg@{#4}%
	\fi
	\edef\@xs@reserved@B{\def\expandafter\noexpand\csname\@xs@reserved@A @@\endcsname\@xs@myarg}%
	\edef\@xs@reserved@C{\unexpanded\expandafter{\@xs@newlines}\edef\noexpand\@xs@call}%
	\edef\@xs@reserved@D{%
		\noexpand\noexpand\expandafter\noexpand\csname\@xs@reserved@A\endcsname\unexpanded\expandafter{\@xs@newargs}%
	}%
	\ifnum#5=\@ne\edef\@xs@reserved@D{\noexpand\noexpand\noexpand\@testopt{\unexpanded\expandafter{\@xs@reserved@D}}{}}\fi
	\@xs@edefaddtomacro\@xs@reserved@C{{\unexpanded\expandafter{\@xs@reserved@D}}\noexpand\@xs@call}%
	\@xs@edefaddtomacro\@xs@reserved@B{{\unexpanded\expandafter{\@xs@reserved@C}}}%
	\@xs@reserved@B
	\edef\@xs@reserved@B{%
		\def\expandafter\noexpand\csname\@xs@reserved@A\endcsname
			\@xs@myarg\ifnum#5=\@ne[\unexpanded{##}\number\numexpr\ifx\@empty#3\@empty#4+1\else#4+2\fi]\fi
	}%
	\@xs@reserved@B
}

\def\@xs@read@reserved@C{%
	\expandafter\@xs@testempty\expandafter{\@xs@reserved@C}%
	\if@xs@empty
		\ifnum\@xs@nestlevel=\z@
			\let\@xs@next\relax
		\else
			\let\@xs@next\@xs@atendofgroup
		\fi
	\else
		\expandafter\@xs@returnfirstsyntaxunit\expandafter{\@xs@reserved@C}\@xs@reserved@A
		\expandafter\@xs@removefirstsyntaxunit\expandafter{\@xs@reserved@C}\@xs@reserved@C
		\let\@xs@next\@xs@read@reserved@C
		\@xs@exploregroups
		\ifx\bgroup\@xs@toks
			\advance\integerpart\@ne
			\begingroup
				\expandafter\def\expandafter\@xs@reserved@C\@xs@reserved@A
				\@xs@manage@groupID
				\let\@xs@nestlevel\@ne
				\integerpart\z@
				\@xs@atbegingroup
		\else
			\global\advance\decimalpart\@ne
			\@xs@atnextsyntaxunit
		\fi
	\fi
	\@xs@next
}

\def\@xs@read@reserved@D{%
	\expandafter\@xs@testempty\expandafter{\@xs@reserved@D}%
	\if@xs@empty
		\ifnum\@xs@nestlevel=\z@
			\let\@xs@next\relax
		\else
			\let\@xs@next\@xs@atendofgroup
		\fi
	\else
		\expandafter\expandafter\expandafter\@xs@IfBeginWith@i\expandafter\expandafter\expandafter{\expandafter\@xs@reserved@D\expandafter}\expandafter{\@xs@reserved@E}%
			{\global\advance\decimalpart\@ne
			\let\@xs@reserved@D\@xs@reserved@A
			\@xs@atoccurfound
			}%
			{\expandafter\@xs@returnfirstsyntaxunit\expandafter{\@xs@reserved@D}\@xs@reserved@A
			\expandafter\@xs@removefirstsyntaxunit\expandafter{\@xs@reserved@D}\@xs@reserved@D
			\let\@xs@next\@xs@read@reserved@D
			\@xs@exploregroups
			\ifx\bgroup\@xs@toks
				\advance\integerpart\@ne
				\begingroup
					\expandafter\def\expandafter\@xs@reserved@D\@xs@reserved@A
					\@xs@manage@groupID
					\let\@xs@reserved@C\@empty
					\let\@xs@nestlevel\@ne
					\integerpart\z@
			\else
				\expandafter\@xs@addtomacro\expandafter\@xs@reserved@C\expandafter{\@xs@reserved@A}%
			\fi
			}%
	\fi
	\@xs@next
}

\@xs@newmacro\StrRemoveBraces{}{1}{1}{%
	\def\@xs@reserved@C{#1}%
	\let\@xs@reserved@B\@empty
	\let\@xs@nestlevel\z@
	\@xs@StrRemoveBraces@i
	\expandafter\@xs@ReturnResult\expandafter{\@xs@reserved@B}{#2}%
}

\def\@xs@StrRemoveBraces@i{%
	\expandafter\@xs@testempty\expandafter{\@xs@reserved@C}%
	\if@xs@empty
		\ifnum\@xs@nestlevel=\z@
			\let\@xs@next\relax
		\else
			\expandafter\endgroup
			\expandafter\@xs@addtomacro\expandafter\@xs@reserved@B\expandafter{\@xs@reserved@B}%
			\let\@xs@next\@xs@StrRemoveBraces@i
		\fi
	\else
		\expandafter\@xs@returnfirstsyntaxunit\expandafter{\@xs@reserved@C}\@xs@reserved@A
		\expandafter\@xs@removefirstsyntaxunit\expandafter{\@xs@reserved@C}\@xs@reserved@C
		\let\@xs@next\@xs@StrRemoveBraces@i
		\ifx\bgroup\@xs@toks
			\ifx\@xs@exploregroups\relax% on explore les groupes ?
				\begingroup
					\expandafter\def\expandafter\@xs@reserved@C\@xs@reserved@A
					\let\@xs@nestlevel\@ne
					\integerpart\z@
					\let\@xs@reserved@B\@empty
			\else
				\expandafter\@xs@addtomacro\expandafter\@xs@reserved@B\@xs@reserved@A
			\fi
		\else
			\expandafter\@xs@addtomacro\expandafter\@xs@reserved@B\expandafter{\@xs@reserved@A}%
		\fi
	\fi
	\@xs@next
}

\def\@xs@cutafteroccur#1#2#3{%
	\ifnum#3<\@ne\expandafter\@firstoftwo\else\expandafter\@secondoftwo\fi
		{\let\@xs@reserved@C\@empty\let\@xs@reserved@E\@empty\let\groupID\@empty}
		{\@xs@cutafteroccur@i{#1}{#2}{#3}}%
}
	
\def\@xs@cutafteroccur@i#1#2#3{%
	\def\@xs@reserved@D{#1}\let\@xs@reserved@C\@empty\def\@xs@reserved@E{#2}%
	\decimalpart\z@\integerpart\z@\def\groupID{0}\let\@xs@nestlevel\z@
	\def\@xs@atendofgroup{%
		\expandafter\endgroup
		\expandafter\@xs@addtomacro\expandafter\@xs@reserved@C\expandafter{\expandafter{\@xs@reserved@C}}%
		\@xs@read@reserved@D}%
	\def\@xs@atoccurfound{%
		\ifnum\decimalpart=\numexpr(#3)\relax
			\global\let\@xs@reserved@D\@xs@reserved@D
			\global\let\@xs@reserved@C\@xs@reserved@C
			\global\let\groupID\groupID
			\@xs@exitallgroups
			\let\@xs@next\relax
		\else
			\expandafter\@xs@addtomacro\expandafter\@xs@reserved@C\expandafter{\@xs@reserved@E}%
			\let\@xs@next\@xs@read@reserved@D
		\fi}%
	\@xs@read@reserved@D
	\def\@xs@argument@A{#2}%
	\ifnum\decimalpart=\numexpr(#3)\relax % occurrence trouv\'ee ?
		\let\@xs@reserved@E\@xs@reserved@D
		\expandafter\expandafter\expandafter\def\expandafter\expandafter\expandafter\@xs@reserved@D\expandafter\expandafter\expandafter{\expandafter\@xs@reserved@C\@xs@argument@A}%
	\else
		\let\@xs@reserved@C\@empty\let\@xs@reserved@E\@empty\let\groupID\@empty
	\fi
}

\@xs@newmacro*3\IfSubStr{1}{2}{0}{%
	\def\@xs@argument@A{#2}\def\@xs@argument@B{#3}%
	\expandafter\expandafter\expandafter\@xs@cutafteroccur
	\expandafter\expandafter\expandafter{\expandafter\@xs@argument@A\expandafter}\expandafter{\@xs@argument@B}{#1}%
	\expandafter\@xs@testempty\expandafter{\@xs@reserved@D}%
	\if@xs@empty
		\expandafter\@secondoftwo
	\else
		\expandafter\@firstoftwo
	\fi
}

\@xs@newmacro*2\IfBeginWith{}{2}{0}{%
	\def\@xs@argument@A{#1}\def\@xs@argument@B{#2}%
	\expandafter\@xs@testempty\expandafter{\@xs@argument@B}%
	\if@xs@empty
		\let\@xs@next\@secondoftwo
	\else
		\def\@xs@next{\expandafter\expandafter\expandafter\@xs@IfBeginWith@i
			\expandafter\expandafter\expandafter{\expandafter\@xs@argument@A\expandafter}\expandafter{\@xs@argument@B}}%
	\fi
	\@xs@next
}

\def\@xs@IfBeginWith@i#1#2{%
	\def\@xs@argument@A{#1}\def\@xs@argument@B{#2}%
	\expandafter\@xs@testempty\expandafter{\@xs@argument@B}%
	\if@xs@empty% #2 est vide, tous les tests sont pass\'es avec succ\`es : on renvoie #3
		\let\@xs@next\@firstoftwo
	\else
		\expandafter\@xs@testempty\expandafter{\@xs@argument@A}%\@xs@testempty{#1}%
		\if@xs@empty
			\let\@xs@next\@secondoftwo% #1 est vide, c'est que #2 est + long que #1 : on renvoie #4
		\else
			\expandafter\@xs@returnfirstsyntaxunit\expandafter{\@xs@argument@B}\@xs@reserved@B
			\expandafter\@xs@returnfirstsyntaxunit\expandafter{\@xs@argument@A}\@xs@reserved@A
			\ifx\@xs@reserved@A\@xs@reserved@B% il y a \'egalit\'e...
				\expandafter\@xs@removefirstsyntaxunit\expandafter{\@xs@argument@B}\@xs@reserved@B
				\expandafter\@xs@removefirstsyntaxunit\expandafter{\@xs@argument@A}\@xs@reserved@A% on enl\`eve les 1ere unit\'es syntaxiques
				\def\@xs@next{% et on recommence avec ces arguments racourcis d'1 unit\'e syntaxique
					\expandafter\expandafter\expandafter\@xs@IfBeginWith@i
					\expandafter\expandafter\expandafter{\expandafter\@xs@reserved@A\expandafter}\expandafter{\@xs@reserved@B}}%
			\else
				\let\@xs@next\@secondoftwo
			\fi
		\fi
	\fi
	\@xs@next
}

\@xs@newmacro*2\IfEndWith{}{2}{0}{%
	\def\@xs@argument@A{#1}\def\@xs@argument@B{#2}%
	\@xs@testempty{#2}%
	\if@xs@empty
		\let\@xs@reserved@A\@secondoftwo
	\else
		\expandafter\expandafter\expandafter\@xs@StrCount
			\expandafter\expandafter\expandafter{\expandafter\@xs@argument@A\expandafter}\expandafter
			{\@xs@argument@B}[\@xs@reserved@A]%
		\ifnum\@xs@reserved@A=\z@
			\let\@xs@reserved@A\@secondoftwo
		\else
			\expandafter\@xs@testempty\expandafter{\@xs@reserved@C}%
			\if@xs@empty
				\let\@xs@reserved@A\@firstoftwo
			\else
				\let\@xs@reserved@A\@secondoftwo
			\fi
		\fi
	\fi
	\@xs@reserved@A
}

\@xs@newmacro*4\IfSubStrBefore{1,1}{3}{0}{\@xs@IfSubStrBefore@i[#1]{#2}{#3}{#4}}
\def\@xs@IfSubStrBefore@i[#1,#2]#3#4#5{%
	\def\@xs@reserved@C{#3}%
	\ifx\@xs@exploregroups\relax% si on explore les groupes
		\let\@xs@reserved@B\@empty
		\let\@xs@nestlevel\z@
		\@xs@StrRemoveBraces@i% on retire les accolades
		\let\@xs@reserved@C\@xs@reserved@B
	\fi
	\def\@xs@reserved@A{#5}%
	\expandafter\expandafter\expandafter\@xs@cutafteroccur\expandafter\expandafter\expandafter{\expandafter\@xs@reserved@C\expandafter}\expandafter{\@xs@reserved@A}{#2}%
	\def\@xs@reserved@A{#4}%
	\expandafter\expandafter\expandafter\@xs@cutafteroccur\expandafter\expandafter\expandafter{\expandafter\@xs@reserved@C\expandafter}\expandafter{\@xs@reserved@A}{#1}%
	\let\groupID\@empty
	\expandafter\@xs@testempty\expandafter{\@xs@reserved@C}%
	\if@xs@empty
		\expandafter\@secondoftwo
	\else
		\expandafter\@firstoftwo
	\fi
}

\@xs@newmacro*4\IfSubStrBehind{1,1}{3}{0}{\@xs@IfSubStrBehind@i[#1]{#2}{#3}{#4}}
\def\@xs@IfSubStrBehind@i[#1,#2]#3#4#5{\@xs@IfSubStrBefore@i[#2,#1]{#3}{#5}{#4}}

\def\@xs@formatnumber#1#2{%
	\def\@xs@argument@A{#1}%
	\@xs@testempty{#1}%
	\if@xs@empty
		\def#2{0X}% si vide, renvoie 0X
	\else
		\@xs@returnfirstsyntaxunit{#1}\@xs@reserved@A
		\def\@xs@reserved@B{+}%
		\ifx\@xs@reserved@A\@xs@reserved@B
			\expandafter\@xs@removefirstsyntaxunit\expandafter{\@xs@argument@A}\@xs@reserved@C
			\expandafter\@xs@testempty\expandafter{\@xs@reserved@C}%
			\if@xs@empty
			 	\def#2{+0X}%
			 \else
			 	\expandafter\def\expandafter#2\expandafter{\expandafter+\expandafter0\@xs@reserved@C}%
			 \fi
		\else
			\def\@xs@reserved@B{-}%
			\ifx\@xs@reserved@A\@xs@reserved@B
				\expandafter\@xs@removefirstsyntaxunit\expandafter{\@xs@argument@A}\@xs@reserved@A
				\expandafter\@xs@testempty\expandafter{\@xs@reserved@A}%
				\if@xs@empty
				 	\def#2{-0X}%
				 \else
				 	\expandafter\def\expandafter#2\expandafter{\expandafter-\expandafter0\@xs@reserved@A}%
				 \fi
			\else
				\expandafter\def\expandafter#2\expandafter{\expandafter0\@xs@argument@A}%
			\fi
		\fi
	\fi
}

\@xs@newmacro\IfInteger{}{1}{0}{%
	\@xs@formatnumber{#1}\@xs@reserved@A
	\decimalpart\z@
	\afterassignment\@xs@defafterinteger\integerpart\@xs@reserved@A\relax\@xs@nil
	\let\@xs@after@intpart\@xs@afterinteger
	\expandafter\@xs@testdot\@xs@afterinteger\@xs@nil
	\ifx\@empty\@xs@afterdecimal
		\ifnum\decimalpart=\z@
			\let\@xs@next\@firstoftwo% partie décimale constituée de 0 --> seul cas où on renvoie vrai
		\else
			\let\@xs@afterinteger\@xs@after@intpart
			\let\@xs@next\@secondoftwo
		\fi
	\else
		\let\@xs@afterinteger\@xs@after@intpart
		\let\@xs@next\@secondoftwo
	\fi
	\@xs@next
}

\@xs@newmacro\IfDecimal{}{1}{0}{%
	\@xs@formatnumber{#1}\@xs@reserved@A
	\decimalpart\z@
	\afterassignment\@xs@defafterinteger\integerpart\@xs@reserved@A\relax\@xs@nil
	\expandafter\@xs@testdot\@xs@afterinteger\@xs@nil
	\ifx\@empty\@xs@afterdecimal
		\expandafter\@firstoftwo
	\else
		\expandafter\@secondoftwo
	\fi
}

\def\@xs@defafterinteger#1\relax\@xs@nil{\def\@xs@afterinteger{#1}}
\def\@xs@testdot{%
	\let\xs@decsep\@empty
	\@ifnextchar.%
		{\def\xs@decsep{.}\@xs@readdecimalpart}%
		{\@xs@testcomma}%
}
\def\@xs@testcomma{%
	\@ifnextchar,%
		{\def\xs@dessep{,}\@xs@readdecimalpart}%
		{\@xs@endnumber}%
}
\def\@xs@readdecimalpart#1#2\@xs@nil{%
	\ifx\@empty#2\@empty
		\def\@xs@reserved@A{0X}%
	\else
		\def\@xs@reserved@A{0#2}%
	\fi
	\afterassignment\@xs@defafterinteger\decimalpart\@xs@reserved@A\relax\@xs@nil
	\expandafter\@xs@endnumber\@xs@afterinteger\@xs@nil
}
\def\@xs@endnumber#1\@xs@nil{\def\@xs@afterdecimal{#1}}

\def\@xs@IfStrEqFalse@i#1#2{\let\@xs@reserved@A\@secondoftwo}
\def\@xs@IfStrEqFalse@ii#1#2{% renvoie vrai si les 2 arg sont d\'ecimaux et s'ils sont \'egaux, faux sinon
	\@xs@IfDecimal{#1}%
		{\@xs@IfDecimal{#2}%
			{\ifdim#1pt=#2pt
				\let\@xs@reserved@A\@firstoftwo
			\else
				\let\@xs@reserved@A\@secondoftwo
			\fi% les 2 sont d\'ecimaux : on fait le test
			}%
			{\let\@xs@reserved@A\@secondoftwo}% un seul est d\'ecimal
		}%
		{\let\@xs@reserved@A\@secondoftwo}% #1 n'est pas d\'ecimal
}

\def\@xs@TestEqual#1#2{% teste si les 2 arguments sont \'egaux
	\def\@xs@reserved@A{#1}\def\@xs@reserved@B{#2}%
	\ifx\@xs@reserved@A\@xs@reserved@B
		\let\@xs@reserved@A\@firstoftwo% \'egalit\'e parfaite des 2 chaines
	\else
		\expandafter\expandafter\expandafter\@xs@reserved@D\expandafter\expandafter\expandafter{\expandafter\@xs@reserved@A\expandafter}\expandafter{\@xs@reserved@B}%
	\fi
	\@xs@reserved@A
}

\@xs@newmacro*2\IfStrEq{}{2}{0}{% teste si les deux chaines sont \'egales
	\let\@xs@reserved@D\@xs@IfStrEqFalse@i
	\@xs@TestEqual{#1}{#2}%
}

\@xs@newmacro*2\IfEq{}{2}{0}{% teste si les 2 arguments (chaine ou nombre) sont \'egaux
	\let\@xs@reserved@D\@xs@IfStrEqFalse@ii
	\@xs@TestEqual{#1}{#2}%
}

\def\IfStrEqCase{%
	\@ifstar
		{\def\@xs@reserved@E{\IfStrEq*}\@xs@IfStrCase}%
		{\def\@xs@reserved@E{\IfStrEq}\@xs@IfStrCase}%
}
\def\@xs@IfStrCase#1#2{\@testopt{\@xs@IfStringCase{#1}{#2}}{}}

\def\IfEqCase{%
	\@ifstar
		{\def\@xs@reserved@E{\IfEq*}\@xs@IfEqCase}%
		{\def\@xs@reserved@E{\IfEq}\@xs@IfEqCase}%
}
\def\@xs@IfEqCase#1#2{\@testopt{\@xs@IfStringCase{#1}{#2}}{}}

\def\@xs@IfStringCase#1#2[#3]{%
	\def\@xs@testcase##1##2##3\@xs@nil{% lit les 2 premieres unit\'es syntaxiques dans ##1 et ##2. Les autres dans ##3
		\@xs@reserved@E{#1}{##1}%
			{##2}% le test est positif, on ex\'ecute le code correspondant
			{\@xs@testempty{##3}%
			 \if@xs@empty% le test est n\'egatif
			 	\def\@xs@next{#3}% s'il n'y a plus de cas, on ex\'ecute le code <autres cas>
			 \else
			 	\def\@xs@next{\@xs@testcase##3\@xs@nil}% sinon, on recommence avec ce qui reste
			 \fi
			 \@xs@next
			 }%
	}%
	\@xs@testcase#2\@xs@nil
}

\@xs@newmacro*3\StrBefore{1}{2}{1}{%
	\@xs@cutafteroccur{#2}{#3}{#1}%
	\expandafter\@xs@ReturnResult\expandafter{\@xs@reserved@C}{#4}%
}

\@xs@newmacro*3\StrBehind{1}{2}{1}{%
	\@xs@cutafteroccur{#2}{#3}{#1}%
	\expandafter\@xs@ReturnResult\expandafter{\@xs@reserved@E}{#4}%
}

\@xs@newmacro*4\StrBetween{1,1}{3}{1}{\@xs@StrBetween@i[#1]{#2}{#3}{#4}[#5]}
\def\@xs@StrBetween@i[#1,#2]#3#4#5[#6]{%
	\begingroup
		\noexploregroups
		\@xs@cutafteroccur{#3}{#5}{#2}%
		\expandafter\@xs@cutafteroccur\expandafter{\@xs@reserved@C}{#4}{#1}%
		\expandafter
	\endgroup
	\expandafter\@xs@ReturnResult\expandafter{\@xs@reserved@E}{#6}%
	\let\groupID\@empty
}

\def\exploregroups{\let\@xs@exploregroups\relax}
\def\noexploregroups{\def\@xs@exploregroups{\let\@xs@toks0\relax}}
\def\saveexploremode{\let\@xs@saveexploremode\@xs@exploregroups}
\def\restoreexploremode{\let\@xs@exploregroups\@xs@saveexploremode}

\@xs@newmacro\StrSubstitute{0}{3}{1}{%
	\def\@xs@reserved@D{#2}\let\@xs@reserved@C\@empty\def\@xs@reserved@E{#3}%
	\def\@xs@argument@C{#3}\def\@xs@argument@D{#4}%
	\decimalpart\z@\integerpart\z@\def\groupID{0}\let\@xs@nestlevel\z@
	\def\@xs@atendofgroup{%
		\expandafter\endgroup
		\expandafter\@xs@addtomacro\expandafter\@xs@reserved@C\expandafter{\expandafter{\@xs@reserved@C}}%
		\@xs@read@reserved@D
	}%
	\def\@xs@atoccurfound{%
		\ifnum#1<\@ne
			\let\@xs@reserved@A\@xs@argument@D
		\else
			\ifnum\decimalpart>#1
				\let\@xs@reserved@A\@xs@argument@C
			\else
				\let\@xs@reserved@A\@xs@argument@D
			\fi
		\fi
		\expandafter\@xs@addtomacro\expandafter\@xs@reserved@C\expandafter{\@xs@reserved@A}%
		\@xs@read@reserved@D
	}%
	\@xs@testempty{#3}%
	\if@xs@empty
		\expandafter\@xs@ReturnResult\expandafter{\@xs@reserved@D}{#5}%
	\else
		\@xs@read@reserved@D
		\expandafter\@xs@ReturnResult\expandafter{\@xs@reserved@C}{#5}%
	\fi
}

\@xs@newmacro\StrDel{0}{2}{1}{\@xs@StrSubstitute[#1]{#2}{#3}{}[#4]}

\def\@xs@exitallgroups{\ifnum\@xs@nestlevel>\z@\endgroup\expandafter\@xs@exitallgroups\fi}

\@xs@newmacro\StrLen{}{1}{1}{%
	\def\@xs@reserved@C{#1}%
	\decimalpart\z@
	\let\@xs@nestlevel\z@
	\def\groupID{0}%
	\let\@xs@atbegingroup\relax
	\def\@xs@atendofgroup{\endgroup\@xs@read@reserved@C}%
	\let\@xs@atnextsyntaxunit\relax
	\@xs@read@reserved@C
	\expandafter\@xs@ReturnResult\expandafter{\number\decimalpart}{#2}%
}

\def\@xs@continuetonext{%
	\expandafter\@xs@testempty\expandafter{\@xs@reserved@C}%
	\if@xs@empty
		\ifnum\@xs@nestlevel>\z@
			\expandafter\endgroup\expandafter\@xs@addtomacro\expandafter\@xs@reserved@B\expandafter{\expandafter{\@xs@reserved@B}}% on concat\`ene
			\expandafter\expandafter\expandafter\@xs@continuetonext% on recommence
		\fi
	\fi
}%

\def\@xs@manage@groupID{%
	\begingroup\def\@xs@reserved@A{0}%
	\ifx\@xs@reserved@A\groupID% si on arrive du groupe de niveau 0
		\endgroup\edef\groupID{\number\integerpart}% on met simplement le niveau
	\else
		\endgroup\expandafter\@xs@addtomacro\expandafter\groupID\expandafter{\expandafter,\number\integerpart}%
	\fi
}

\def\StrSplit{%
	\@ifstar
		{\let\@xs@reserved@E\@xs@continuetonext\StrSpl@t}%
		{\let\@xs@reserved@E\relax\StrSpl@t}%
}
\@xs@newmacro\StrSpl@t{}{2}{0}{\@xs@StrSplit@i{#2}{#1}\@xs@StrSplit@ii}
\def\@xs@StrSplit@i#1#2{%
	\def\@xs@reserved@D{#1}\def\@xs@reserved@C{#2}\let\@xs@reserved@B\@empty\let\groupID\@empty
	\ifnum#1>\z@
		\decimalpart\z@\integerpart\z@\def\groupID{0}\let\@xs@nestlevel\z@
		\def\@xs@atendofgroup{%
			\expandafter\endgroup
			\expandafter\@xs@addtomacro\expandafter\@xs@reserved@B\expandafter{\expandafter{\@xs@reserved@B}}%
			\@xs@read@reserved@C
		}%
		\def\@xs@atbegingroup{\let\@xs@reserved@B\@empty}%
		\def\@xs@atnextsyntaxunit{%
			\expandafter\@xs@addtomacro\expandafter\@xs@reserved@B\expandafter{\@xs@reserved@A}%
			\ifnum\decimalpart=\@xs@reserved@D\relax
				\ifx\@xs@reserved@C\@empty\@xs@reserved@E\fi
				\global\let\@xs@reserved@B\@xs@reserved@B
				\global\let\@xs@reserved@C\@xs@reserved@C
				\global\let\groupID\groupID
				\@xs@exitallgroups
				\let\@xs@next\relax
			\fi
		}%
		\@xs@read@reserved@C
	\fi
}
\def\@xs@StrSplit@ii#1#2{\let#1\@xs@reserved@B\let#2\@xs@reserved@C}

\@xs@newmacro*3\StrCut{1}{2}{0}{%
	\@xs@testempty{#3}%
	\if@xs@empty\expandafter\@firstoftwo\else\expandafter\@secondoftwo\fi
		{\let\groupID\@empty
		\let\@xs@reserved@C\@empty
		\let\@xs@reserved@E\@empty
		}
		{\ifnum#1<\@ne\expandafter\@firstoftwo\else\expandafter\@secondoftwo\fi
			{\@xs@StrCut@ii{#2}{#3}1}
			{\@xs@StrCut@ii{#2}{#3}{#1}}%
		}%
	\@xs@StrCut@iii
}

\def\@xs@StrCut@ii#1#2#3{%
	\def\@xs@reserved@D{#1}%
	\let\@xs@reserved@C\@empty
	\def\@xs@reserved@E{#2}%
	\decimalpart\z@\integerpart\z@
	\def\groupID{0}%
	\let\@xs@nestlevel\z@
	\def\@xs@atendofgroup{%
		\expandafter\endgroup
		\expandafter\@xs@addtomacro\expandafter\@xs@reserved@C\expandafter{\expandafter{\@xs@reserved@C}}%
		\@xs@read@reserved@D
	}%
	\def\@xs@atoccurfound{%
		\ifnum\decimalpart=\numexpr(#3)\relax
			\global\let\@xs@reserved@D\@xs@reserved@D
			\global\let\@xs@reserved@C\@xs@reserved@C
			\global\let\groupID\groupID
			\@xs@exitallgroups
			\let\@xs@next\relax
		\else
				\expandafter\@xs@addtomacro\expandafter\@xs@reserved@C\expandafter{\@xs@reserved@E}%
			\let\@xs@next\@xs@read@reserved@D
		\fi
	}%
	\@xs@read@reserved@D
	\def\@xs@argument@A{#2}%
	\let\@xs@reserved@E\@xs@reserved@D
		\expandafter\expandafter\expandafter
	\def
		\expandafter\expandafter\expandafter
	\@xs@reserved@D
		\expandafter\expandafter\expandafter
	{\expandafter\@xs@reserved@C\@xs@argument@A}%
}

\def\@xs@StrCut@iii#1#2{\let#1\@xs@reserved@C\let#2\@xs@reserved@E}

\@xs@newmacro\StrMid{}{3}{1}{%
	\begingroup
		\noexploregroups
		\let\@xs@reserved@E\relax
		\@xs@StrSplit@i{#3}{#1}%
		\edef\@xs@reserved@C{\number\numexpr#2-1}%
		\let\@xs@reserved@E\relax
		\expandafter\expandafter\expandafter\@xs@StrSplit@i\expandafter\expandafter\expandafter{\expandafter\@xs@reserved@C\expandafter}\expandafter{\@xs@reserved@B}%
	\expandafter\endgroup
	\expandafter\@xs@ReturnResult\expandafter{\@xs@reserved@C}{#4}%
	\let\groupID\@empty
}

\@xs@newmacro\StrGobbleLeft{}{2}{1}{%
	\let\@xs@reserved@E\relax
	\@xs@StrSplit@i{#2}{#1}%
	\expandafter\@xs@ReturnResult\expandafter{\@xs@reserved@C}{#3}%
}

\@xs@newmacro\StrLeft{}{2}{1}{%
	\let\@xs@reserved@E\relax
	\@xs@StrSplit@i{#2}{#1}%
	\expandafter\@xs@ReturnResult\expandafter{\@xs@reserved@B}{#3}%
}

\@xs@newmacro\StrGobbleRight{}{2}{1}{%
	\@xs@StrLen{#1}[\@xs@reserved@D]%
	\let\@xs@reserved@E\relax
	\expandafter\@xs@StrSplit@i\expandafter{\number\numexpr\@xs@reserved@D-#2}{#1}%
	\expandafter\@xs@ReturnResult\expandafter{\@xs@reserved@B}{#3}%
}

\@xs@newmacro\StrRight{}{2}{1}{%
	\@xs@StrLen{#1}[\@xs@reserved@D]%
	\let\@xs@reserved@E\relax
	\expandafter\@xs@StrSplit@i\expandafter{\number\numexpr\@xs@reserved@D-#2}{#1}%
	\expandafter\@xs@ReturnResult\expandafter{\@xs@reserved@C}{#3}%
}

\@xs@newmacro\StrChar{}{2}{1}{%
	\let\@xs@reserved@B\@empty
	\def\@xs@reserved@C{#1}\def\@xs@reserved@D{#2}%
	\ifnum#2>\z@
		\def\groupID{0}\let\@xs@nestlevel\z@\integerpart\z@\decimalpart\z@
		\let\@xs@atbegingroup\relax
		\def\@xs@atendofgroup{\endgroup\@xs@read@reserved@C}%
		\def\@xs@atnextsyntaxunit{%
			\ifnum\decimalpart=\@xs@reserved@D% la n i\`eme US est atteinte ?
				\global\let\@xs@reserved@B\@xs@reserved@A% on capture l'US en cours qui est celle cherch\'ee
				\global\let\groupID\groupID
				\@xs@exitallgroups
				\let\@xs@next\relax
			\fi
		}%
		\@xs@read@reserved@C
	\fi
	\expandafter\@xs@testempty\expandafter{\@xs@reserved@B}%
	\if@xs@empty\let\groupID\@empty\fi
	\expandafter\@xs@ReturnResult\expandafter{\@xs@reserved@B}{#3}%
}

\@xs@newmacro\StrCount{}{2}{1}{%
	\@xs@testempty{#2}%
	\def\@xs@reserved@D{#1}\def\@xs@reserved@E{#2}\let\@xs@reserved@C\@empty
	\if@xs@empty
		\@xs@ReturnResult{0}{#3}%
	\else
		\decimalpart\z@\integerpart\z@\def\groupID{0}\let\@xs@nestlevel\z@
		\def\@xs@atendofgroup{%
			\expandafter\endgroup
			\expandafter\@xs@addtomacro\expandafter\@xs@reserved@C\expandafter{\expandafter{\@xs@reserved@C}}%
			\@xs@read@reserved@D
		}%
		\def\@xs@atoccurfound{\let\@xs@reserved@C\@empty\@xs@read@reserved@D}%
		\@xs@read@reserved@D
		\expandafter\@xs@ReturnResult\expandafter{\number\decimalpart}{#3}%
	\fi
}

\@xs@newmacro\StrPosition{1}{2}{1}{%
	\@xs@cutafteroccur{#2}{#3}{#1}%
	\let\@xs@reserved@E\groupID
	\ifx\@xs@reserved@C\@xs@reserved@D
		\@xs@ReturnResult{0}{#4}%
		\let\@xs@reserved@E\@empty
	\else
		\expandafter\@xs@StrLen\expandafter{\@xs@reserved@C}[\@xs@reserved@C]%
		\expandafter\@xs@ReturnResult\expandafter{\number\numexpr\@xs@reserved@C+1}{#4}%
	\fi
	\let\groupID\@xs@reserved@E
}

\def\comparestrict{\let\@xs@comparecoeff\@ne}
\def\comparenormal{\let\@xs@comparecoeff\z@}
\def\savecomparemode{\let\@xs@saved@comparecoeff\@xs@comparecoeff}
\def\restorecomparemode{\let\@xs@comparecoeff\@xs@saved@comparecoeff}
\@xs@newmacro*2\StrCompare{}{2}{1}{%
	\def\@xs@reserved@A{#1}%
	\def\@xs@reserved@B{#2}%
	\ifx\@xs@reserved@B\@xs@reserved@A
		\@xs@ReturnResult{0}{#3}%
	\else
		\def\@xs@next{\@xs@StrCompare@i{#1}{#2}{#3}}%
		\expandafter\@xs@next
	\fi
}

\def\@xs@StrCompare@i#1#2#3{%
	\def\@xs@StrCompare@iii##1{%
		\let\@xs@reserved@A\@empty
		\expandafter\@xs@testempty\expandafter{\@xs@reserved@C}%
		\if@xs@empty
			\def\@xs@reserved@A{*\@xs@comparecoeff}%
		\else
			\expandafter\@xs@testempty\expandafter{\@xs@reserved@D}%
			\if@xs@empty
				\def\@xs@reserved@A{*\@xs@comparecoeff}%
			\fi
		\fi
		\def\@xs@next{%
			\expandafter\@xs@ReturnResult\expandafter
			{\number\numexpr##1\@xs@reserved@A\relax}{#3}%
		}%
	}%
	\def\@xs@StrCompare@ii##1{% ##1 est la position
		\expandafter\@xs@returnfirstsyntaxunit\expandafter{\@xs@reserved@C}\@xs@reserved@A
		\expandafter\@xs@returnfirstsyntaxunit\expandafter{\@xs@reserved@D}\@xs@reserved@B
		\ifx\@xs@reserved@B\@xs@reserved@A
			\expandafter\@xs@testempty\expandafter{\@xs@reserved@A}%
			\if@xs@empty
				\@xs@StrCompare@iii{##1}% les 2 unit\'es syntaxiques sont \'egales, on renvoie la position
			\else
				\def\@xs@next{\expandafter\@xs@StrCompare@ii\expandafter{\number\numexpr##1+1}}% les 2 unit\'es syntaxiques sont \'egales et non vides, on recommence
				\expandafter\@xs@removefirstsyntaxunit\expandafter{\@xs@reserved@C}\@xs@reserved@C
				\expandafter\@xs@removefirstsyntaxunit\expandafter{\@xs@reserved@D}\@xs@reserved@D
			\fi
		\else% les 2 unit\'es syntaxiques sont diff\'erentes : on renvoie la position
			\@xs@StrCompare@iii{##1}%
		\fi
		\@xs@next
	}%
	\def\@xs@reserved@C{#1}\def\@xs@reserved@D{#2}%
	\@xs@StrCompare@ii1%
}

\@xs@newmacro\StrFindGroup{}{2}{1}{%
	\def\@xs@reserved@A{#2}\def\@xs@reserved@B{0}%
	\ifx\@xs@reserved@A\@xs@reserved@B
		\def\@xs@next{\@xs@ReturnResult{{#1}}{#3}}%
	\else
		\def\@xs@next{\@xs@StrFindGroup@i{#1}{#2}[#3]}%
	\fi
	\@xs@next
}

\def\@xs@StrFindGroup@i#1#2[#3]{%
	\def\@xs@StrFindGroup@ii{%
		\expandafter\@xs@testempty\expandafter{\@xs@reserved@C}%
		\if@xs@empty
			\def\@xs@next{\@xs@ReturnResult{}{#3}}% s'il ne reste plus rien, on renvoie vide
		\else
			\expandafter\@xs@returnfirstsyntaxunit\expandafter{\@xs@reserved@C}\@xs@reserved@D
			\ifx\bgroup\@xs@toks% si la 1\`ere unit\'e syntaxique est un groupe explicite
				\advance\decimalpart\@ne% on augmente le compteur
				\ifnum\decimalpart=\@xs@reserved@A% on est au groupe cherch\'e lors de la profondeur courante ?
					\ifx\@empty\@xs@reserved@B% on est \`a la profondeur maximale ?
						\def\@xs@next{\expandafter\@xs@ReturnResult\expandafter{\@xs@reserved@D}{#3}}% on renvoie ce groupe
					\else% sinon
						\expandafter\def\expandafter\@xs@next\expandafter{\expandafter\@xs@StrFindGroup@i\@xs@reserved@D}% on recommence avec ce groupe
						\expandafter\@xs@addtomacro\expandafter\@xs@next\expandafter{\expandafter{\@xs@reserved@B}[#3]}% et les profondeurs de recherche restantes
					\fi
				\else
					\expandafter\@xs@removefirstsyntaxunit\expandafter{\@xs@reserved@C}\@xs@reserved@C
					\let\@xs@next\@xs@StrFindGroup@ii
				\fi
			\else
				\expandafter\@xs@removefirstsyntaxunit\expandafter{\@xs@reserved@C}\@xs@reserved@C
				\let\@xs@next\@xs@StrFindGroup@ii
			\fi
		\fi
		\@xs@next
	}%
	\@xs@extractgroupnumber{#2}\@xs@reserved@A\@xs@reserved@B
	\decimalpart\z@
	\ifnum\@xs@reserved@A>\z@\def\@xs@reserved@C{#1}\else\let\@xs@reserved@C\@empty\fi
	\@xs@StrFindGroup@ii
}

\def\@xs@extractgroupnumber#1#2#3{%
	\def\@xs@extractgroupnumber@i##1,##2\@xs@nil{\def#2{##1}\def#3{##2}}%
	\@xs@extractgroupnumber@i#1,\@xs@nil
	\ifx\@empty#3\else\@xs@extractgroupnumber@i#1\@xs@nil\fi
}

\def\expandingroups{\let\@xs@expandingroups\exploregroups}
\def\noexpandingroups{\let\@xs@expandingroups\noexploregroups}
\def\StrExpand{\@testopt{\@xs@StrExpand}{1}}
\def\@xs@StrExpand[#1]#2#3{%
	\begingroup
		\@xs@expandingroups
		\ifnum#1>\z@
			\integerpart#1\relax
			\decimalpart\z@\def\groupID{0}\let\@xs@nestlevel\z@
			\def\@xs@atendofgroup{%
				\expandafter\endgroup
				\expandafter\@xs@addtomacro\expandafter\@xs@reserved@B\expandafter{\expandafter{\@xs@reserved@B}}%
				\@xs@read@reserved@C
			}%
			\def\@xs@atbegingroup{\let\@xs@reserved@B\@empty}%
			\def\@xs@atnextsyntaxunit{%
				\expandafter\expandafter\expandafter\@xs@addtomacro
				\expandafter\expandafter\expandafter\@xs@reserved@B
				\expandafter\expandafter\expandafter{\@xs@reserved@A}%
			}%
			\def\@xs@reserved@C{#2}%
			\@xs@StrExpand@i{#1}% Appel de la macro r\'ecursive
		\else
			\def\@xs@reserved@B{#2}%
		\fi
		\global\let\@xs@reserved@B\@xs@reserved@B
	\endgroup
	\let#3\@xs@reserved@B
	\let\groupID\@empty
}

\def\@xs@StrExpand@i#1{%
	\ifnum#1>\z@
		\let\@xs@reserved@B\@empty
		\@xs@read@reserved@C
		\let\@xs@reserved@C\@xs@reserved@B
		\def\@xs@reserved@A{\expandafter\@xs@StrExpand@i\expandafter{\number\numexpr#1-1}}%
	\else
		\let\@xs@reserved@A\relax
	\fi
	\@xs@reserved@A
}

\def\scancs{\@testopt{\@xs@scancs}{1}}
\def\@xs@scancs[#1]#2#3{%
	\@xs@StrExpand[#1]{#3}{#2}%
	\edef#2{\detokenize\expandafter{#2}}%
}

\catcode`\@=\CurrentAtCatcode\relax
\setverbdelim{|}%
\fullexpandarg\saveexpandmode
\comparenormal\savecomparemode
\noexploregroups\saveexploremode
\expandingroups

 \newcommand\id{\mathrm{id}}

 \newcommand\tGr{\mathrm{Gr}}

 \newcommand\tad{\mathrm{ad}}

 \newcommand\qRa{\quad\Rightarrow\quad}

 \newcommand\bop{\bigoplus}
 \newcommand\op{\oplus}
 \newcommand\ot{\otimes}

 \newcommand\fann{\mathfrak{ann}}

 \newcommand\bx{{\mathbf x}}

 \newcommand\bS{{\mathbf S}}

 \newcommand\bX{{\mathbf X}} 
 \newcommand\bY{{\mathbf Y}}
 \newcommand\bZ{{\mathbf Z}}
 
 \newcommand\rnk{\mathrm{rank}}

 \newcommand\im{\mathrm{im}}

 \newcommand\fa{{\mathfrak a}} 
 \newcommand\fb{{\mathfrak b}}
 \newcommand\fc{{\mathfrak c}}

 \newcommand\ff{{\mathfrak f}}
 \newcommand\fg{{\mathfrak g}}
 
 \newcommand\fh{{\mathfrak h}}

 \newcommand\fk{{\mathfrak k}}
 \newcommand\fl{{\mathfrak l}}

 \newcommand\fp{{\mathfrak p}}
 \newcommand\fq{{\mathfrak q}}
 
 \newcommand\fs{{\mathfrak s}}
 \newcommand\fsl{\mathfrak{sl}}

 \newcommand\fz{{\mathfrak z}}

 \newcommand\fD{{\mathfrak D}}

 \newcommand\fT{{\mathfrak T}}
 \newcommand\fU{{\mathfrak U}}

 \newcommand\fX{{\mathfrak X}}

 \newcommand\cA{{\mathcal A}}

 \newcommand\cD{{\mathcal D}}

 \newcommand\cG{{\mathcal G}}
 
 \newcommand\cI{{\mathcal I}}

 \newcommand\cL{{\mathcal L}}

 \newcommand\cP{{\mathcal P}}

 \newcommand\cS{{\mathcal S}}

 \newcommand\cV{{\mathcal V}}

 \newcommand\cZ{{\mathcal Z}}

 \newcommand\sfc{\mathsf{c}}

 \newcommand\sfE{\mathsf{E}}
 \newcommand\sfF{\mathsf{F}}
 
 \newcommand\sfH{\mathsf{H}}

 \newcommand\sfZ{\mathsf{Z}}

 \newcommand\bbC{{\mathbb C}}
 
 \newcommand\bbE{{\mathbb E}}

 \newcommand\bbO{{\mathbb O}}
 \newcommand\bbP{{\mathbb P}}
 
 \newcommand\bbR{{\mathbb R}}
 \newcommand\bbS{{\mathbb S}}
 
 \newcommand\bbU{{\mathbb U}}
 \newcommand\bbV{{\mathbb V}}
 \newcommand\bbW{{\mathbb W}}

 \newcommand\bbZ{{\mathbb Z}}

 \newcommand\tspan{\mathrm{span}}

 \newcommand\Ben{\begin{enumerate}}
 \newcommand\Een{\end{enumerate}}
 \newcommand\Bex{\begin{example}}
 \newcommand\Eex{\end{example}}
 \def\inj{\hookrightarrow}

\def\tGL{\text{GL}}

 \newcommand\SL{\mathrm{SL}}
 \newcommand\ad{{\rm ad}}
 
 \newcommand\Ad{{\rm Ad}}
 \def\RP{\bbR\bbP}

 \def\assoc/{associative}
 \def\arb/{arbitrary}
 \def\btw/{between}
 \def\coeff/{coefficient}
 \def\cohom/{cohomology}
 \def\coord/{coordinate}
 \def\coordsys/{coordinate system}
 \def\cpt/{compact}
 \def\cred/{completely reducible}
 \def\cts/{continuous}
 \def\dga/{differential-graded algebra}
 \def\dR/{de Rham}
 \def\Euc/{Euclidean} 
 \def\grp/{group}
 \def\hom/{homomorphism}
 \def\inv/{invariant}
 \def\iso/{isomorphism}
 \def\La/{Lie algebra}
 \def\Lag/{Lagrangian Grassmannian}
 \def\LG/{Lie group}
 \def\MA/{Monge--Amp\`ere}
 \def\MC/{Maurer--Cartan}
 \def\lintr/{linear transformation} 
 \def\mfld/{manifold}
 \def\nb/{normal bundle}
 \def\nbd/{neighbourhood}
 \def\nondeg/{non-degenerate}
 \def\posdef/{positive definite}
 \def\pu/{partition of unity}
 \def\rep/{representation}
 \def\Riem/{Riemannian}
 \def\sg/{subgroup}
 \def\ss/{semi-simple}
 \def\inv/{invariant}
 \def\irr/{irreducible}
 \def\Jacid/{Jacobi identity}
 \def\li/{linearly independent}
 \def\nd/{nowhere dependent}
 \def\nz/{nowhere zero}
 \def\on/{orthonormal}
 \def\onb/{\on/ basis}
 \def\orc/{\orth/ complement}
 \def\orth/{orthogonal}
 \def\orp/{\orth/ projection}
 \def\pde/{partial differential equation}
 \def\resp/{respectively}
 \def\seq/{sequence}
 \def\std/{standard}
 \def\SW/{Stiefel-Whitney}
 \def\uc/{universal cover}
 \def\vb/{vector bundle}
 \def\vf/{vector field}
 \def\vs/{vector space}
 \def\wrt/{with respect to}
 
 \def\st{\,|\,}

 \newcommand\qbox[1]{\quad\mbox{#1}\quad}
 \renewcommand\dim{{\rm dim}}
 
\newcommand\myscale{0.8}

\newcommand\tcirc[3]{
	\ifthenelse{\equal{#1}{w}}{\filldraw[fill=white,draw=black] (#2) circle (0.075);}{}%
	\ifthenelse{\equal{#1}{b}}{\filldraw[black] (#2) circle (0.075);}{}%
	\draw (#2) ++(0,0.35) node {$#3$};
	}

\newcommand\tdots[1]{\draw (#1) ++(0.55,0) node {$\cdots$}}

\newcommand\bond[1]{\draw (#1) -- +(1,0)}

\newcommand\vbond[1]{\draw (#1) -- +(0,-1)}

\newcommand\diagbond[2]{
	\ifthenelse{\equal{#1}{u}}{
		\draw[semithick] (#2) -- +(0.5,0.865);
	}{}
	\ifthenelse{\equal{#1}{d}}{
		\draw[semithick] (#2) -- +(0.5,-0.865);
	}{}
	}

\newcommand\dbond[2]{
	\draw (#2) ++(0.03,0.03) -- +(0.94,0);
	\draw (#2) ++(0.03,-0.03) -- +(0.94,0);
	\ifthenelse{\equal{#1}{r}}{
		\draw[semithick] (#2) ++(0.6,0) ++(-0.15,0.2) -- ++(0.15,-0.2) -- +(-0.15,-0.2);
	}{}
	\ifthenelse{\equal{#1}{l}}{
		\draw[semithick] (#2) ++(0.45,0) ++(0.15,0.2) -- ++(-0.15,-0.2) -- +(0.15,-0.2);
	}{}
	}

\newcommand\tbond[2]{
	\draw (#2)  -- +(1,0);
	\draw (#2) ++(0.05,0.06) -- +(0.9,0);
	\draw (#2) ++(0.05,-0.06) -- +(0.9,0);
	\ifthenelse{\equal{#1}{r}}{
		\draw[semithick] (#2) ++(0.6,0) ++(-0.15,0.2) -- ++(0.15,-0.2) -- +(-0.15,-0.2);
	}{}
	\ifthenelse{\equal{#1}{l}}{
		\draw[semithick] (#2) ++(0.45,0) ++(0.15,0.2) -- ++(-0.15,-0.2) -- +(0.15,-0.2);
	}{}
	}

\newcommand\tcross[2]{
	\draw (#1) ++(0,0.35) node {$#2$};
	\draw[semithick] (#1) ++(-0.15,-0.15)-- +(0.3,0.3);
	\draw[semithick] (#1) ++(-0.15,0.15)-- +(0.3,-0.3);
	}

\newcommand\tsquare[2]{
		\draw[semithick,color=blue] (#1) ++(-0.15,-0.15) rectangle ++(0.3,0.3);
		\tcross{#1}{#2};
		}

\newcommand\tstar[2]{
	\draw[color=red] (#1) node {\Large$*$};
	\draw (#1) ++(0,0.35) node {$#2$};
	}

\newcommand\tdiamond[2]{
	\draw[color=ForestGreen] (#1) ++(0,-0.02) node {\scalebox{2.5}{$\blackdiamond$}};
	\draw (#1) ++(0,0.30) node {$#2$};
	}

\newcommand\DDnode[3]{
\ifthenelse{\equal{#1}{w}}{\tcirc{w}{#2}{#3}}{}		% white - non-compact root (Satake diagram)
\ifthenelse{\equal{#1}{b}}{\tcirc{b}{#2}{#3}}{}		% black - compact root (Satake diagram)
\ifthenelse{\equal{#1}{x}}{\tcross{#2}{#3}}{}		% crossed root (corresponding to parabolic)
\ifthenelse{\equal{#1}{s}}{\tstar{#2}{#3}}{}		% starred root (my notation for sub-parabolic)
\ifthenelse{\equal{#1}{q}}{\tsquare{#2}{#3}}{}		% crossed square (Iw root)
\ifthenelse{\equal{#1}{d}}{\tdiamond{#2}{#3}}{}		% diamond
}

\newcommand\Aone[2]{
 \begin{tiny}
 \begin{tikzpicture}[scale=\myscale,baseline=-3pt]
 
 \StrChar{#1}{1}[\nodetype];
 \DDnode{\nodetype}{0,0}{#2};
 \useasboundingbox (-.4,-.2) rectangle (0.4,0.55); % make bounding box bigger
 \end{tikzpicture}
 \end{tiny}
 }

\newcommand\Atwo[2]{
 \begin{tiny}
 \begin{tikzpicture}[scale=\myscale,baseline=-3pt]
 
 \bond{0,0};	% triple bond

 \StrBefore{#2}{,}[\labelone]
 \StrBehind{#2}{,}[\labeltwo]
 
 \StrChar{#1}{1}[\nodetype];
 \DDnode{\nodetype}{0,0}{\labelone};
 \StrChar{#1}{2}[\nodetype];
 \DDnode{\nodetype}{1,0}{\labeltwo};
 \useasboundingbox (-.4,-.2) rectangle (1.4,0.55); % make bounding box bigger
 \end{tikzpicture}
 \end{tiny}
 }

\newcommand\Dfive[2]{
 \begin{tiny}
 \begin{tikzpicture}[scale=\myscale,baseline=-3pt]
 \bond{0,0};		% bond
 \bond{1,0};		% bond
 \draw ++(2,0) -- +(0.5,0.865);	% bond
 \draw ++(2,0) -- +(0.5,-0.865);	% bond

 \StrBefore{#2}{,}[\labelone]
 \StrBetween[1,2]{#2}{,}{,}[\labeltwo]
 \StrBetween[2,3]{#2}{,}{,}[\labelthree]
 \StrBetween[3,4]{#2}{,}{,}[\labelfour]
 \StrBehind[4]{#2}{,}[\labelfive]

 \StrChar{#1}{1}[\nodetype];
 \DDnode{\nodetype}{0,0}{\labelone};
 \StrChar{#1}{2}[\nodetype];
 \DDnode{\nodetype}{1,0}{\labeltwo};
 \StrChar{#1}{3}[\nodetype];
 \DDnode{\nodetype}{2,0}{\labelthree};
 \StrChar{#1}{4}[\nodetype];
 \DDnode{\nodetype}{2.5,0.865}{\labelfour};
 \StrChar{#1}{5}[\nodetype];
 \DDnode{\nodetype}{2.5,-0.865}{\labelfive};
 \end{tikzpicture}
 \end{tiny}
 }

\newcommand\Dsix[2]{
 \begin{tiny}
 \begin{tikzpicture}[scale=\myscale,baseline=-3pt]
 \bond{0,0};		% bond
 \bond{1,0};		% bond
 \bond{2,0};		% bond
 \draw ++(3,0) -- +(0.5,0.865);	% bond
 \draw ++(3,0) -- +(0.5,-0.865);	% bond

 \StrBefore{#2}{,}[\labelone]
 \StrBetween[1,2]{#2}{,}{,}[\labeltwo]
 \StrBetween[2,3]{#2}{,}{,}[\labelthree]
 \StrBetween[3,4]{#2}{,}{,}[\labelfour]
 \StrBetween[4,5]{#2}{,}{,}[\labelfive]
 \StrBehind[5]{#2}{,}[\labelsix]

 \StrChar{#1}{1}[\nodetype];
 \DDnode{\nodetype}{0,0}{\labelone};
 \StrChar{#1}{2}[\nodetype];
 \DDnode{\nodetype}{1,0}{\labeltwo};
 \StrChar{#1}{3}[\nodetype];
 \DDnode{\nodetype}{2,0}{\labelthree};
 \StrChar{#1}{4}[\nodetype];
 \DDnode{\nodetype}{3,0}{\labelfour};
 \StrChar{#1}{5}[\nodetype];
 \DDnode{\nodetype}{3.5,0.865}{\labelfive};
 \StrChar{#1}{6}[\nodetype];
 \DDnode{\nodetype}{3.5,-0.865}{\labelsix};
 \end{tikzpicture}
 \end{tiny}
 }

\newcommand\Edd[2]{
 \begin{tiny}
 \begin{tikzpicture}[scale=\myscale,baseline=-3pt]
 \foreach \x in {0,1,2,3} {
	\bond{\x,0};
 }
 \vbond{2,0};
 
 \StrLen{#1}[\Ernk]
 
 \StrChar{#1}{1}[\nodetype];
 \DDnode{\nodetype}{0,0}{\StrBefore{#2}{,}};
 \StrChar{#1}{2}[\nodetype];
 \DDnode{\nodetype}{2,-1}{\StrBetween[1,2]{#2}{,}{,}};
 \StrChar{#1}{3}[\nodetype];
 \DDnode{\nodetype}{1,0}{\StrBetween[2,3]{#2}{,}{,}};
 \StrChar{#1}{4}[\nodetype];
 \DDnode{\nodetype}{2,0}{\StrBetween[3,4]{#2}{,}{,}};
 \StrChar{#1}{5}[\nodetype];
 \DDnode{\nodetype}{3,0}{\StrBetween[4,5]{#2}{,}{,}};
 \StrChar{#1}{6}[\nodetype];

 \ifthenelse{\equal{\Ernk}{6}}{
 		\DDnode{\nodetype}{4,0}{\StrBehind[5]{#2}{,}};
 		\useasboundingbox (-.4,-1.2) rectangle (4.4,0.55);
	}{}%
 
 \ifthenelse{\equal{\Ernk}{7}}{
 		\bond{4,0};
 		\DDnode{\nodetype}{4,0}{\StrBetween[5,6]{#2}{,}{,}};
		\StrChar{#1}{7}[\nodetype];
		\DDnode{\nodetype}{5,0}{\StrBehind[6]{#2}{,}};
 		\useasboundingbox (-.4,-1.2) rectangle (5.4,0.55);
	}{}%

 \ifthenelse{\equal{\Ernk}{8}}{
 		\bond{4,0};
 		\bond{5,0};
 		\DDnode{\nodetype}{4,0}{\StrBetween[5,6]{#2}{,}{,}};
		\StrChar{#1}{7}[\nodetype];
		\DDnode{\nodetype}{5,0}{\StrBetween[6,7]{#2}{,}{,}};
		\StrChar{#1}{8}[\nodetype];
		\DDnode{\nodetype}{6,0}{\StrBehind[7]{#2}{,}};
		\useasboundingbox (-.4,-1.2) rectangle (6.4,0.55);
	}{}%

 \end{tikzpicture}
 \end{tiny}
 }

 \newtheorem{thm}{Theorem}[section]
 \newtheorem{lemma}[thm]{Lemma}
 \newtheorem{cor}[thm]{Corollary}
 \newtheorem{prop}[thm]{Proposition}
 \theoremstyle{defn}
 \newtheorem{defn}[thm]{Definition}
 \newtheorem{example}[thm]{Example}

 \theoremstyle{remark}
 \newtheorem{remark}[thm]{Remark}
 \numberwithin{equation}{section}

 \newcommand\finf{\mathfrak{inf}}
 \newcommand\prn{\operatorname{pr}}
 \newcommand\gr{\operatorname{gr}}
 \newcommand\Seg{\operatorname{Seg}}
 \newcommand\Sec{\operatorname{Sec}}

 \newcommand\deq{\,\,\dot{=}\,\,}
 
 \newcommand\din{\,\dot\in\,}

 \newcommand\mydiamond{\scalebox{1.5}{{\color{ForestGreen} $\blackdiamond$}}}

\newcommand\Add[2]{
 \begin{tiny}
 \begin{tikzpicture}[scale=\myscale,baseline=-3pt]

 \bond{0,0};		% bond
 \tdots{1,0};
 \bond{2,0};		% bond

 \StrBefore{#2}{,}[\labelone]
 \StrBetween[1,2]{#2}{,}{,}[\labeltwo]
 \StrBetween[2,3]{#2}{,}{,}[\labelthree]
 \StrBehind[3]{#2}{,}[\labelfour]

 \StrChar{#1}{1}[\nodetype];
 \DDnode{\nodetype}{0,0}{\labelone};
 \StrChar{#1}{2}[\nodetype];
 \DDnode{\nodetype}{1,0}{\labeltwo};
 \StrChar{#1}{3}[\nodetype];
 \DDnode{\nodetype}{2,0}{\labelthree};
 \StrChar{#1}{4}[\nodetype];
 \DDnode{\nodetype}{3,0}{\labelfour};

 \useasboundingbox (-.4,-.2) rectangle (3.4,0.55); % make bounding box bigger
 \end{tikzpicture}
 \end{tiny}
 }



\begin{document}

 \title[Jet-determination of symmetries of parabolic geometries]{Jet-determination of symmetries of parabolic geometries}
 \author[B. Kruglikov]{Boris Kruglikov}
 \author[D. The]{Dennis The}
 % address & email appear at the end of the article

 \address{Department of Mathematics and Statistics, University of Troms$\o$, 90-37, Norway}
 \email{boris.kruglikov@uit.no}
 
 \address{
 Department of Mathematics and Statistics, University of Troms$\o$, 90-37, Norway; 
 Fakult\"at f\"ur Mathematik, Universit\"at Wien, Oskar-Morgenstern-Platz 1, 1090 Wien, \"Osterreich;
 Mathematical Sciences Institute, Australian National University, ACT 0200, Australia} 
 \email{dennis.the@uit.no}

 % date, subject class, keywords, thanks appear in footer at the bottom of first page
 \date{\today}
 \begin{abstract}

 We establish 2-jet determinacy for the symmetry algebra of the underlying structure of any (complex or real) parabolic geometry. At non-flat points, we prove that the symmetry algebra is in fact 1-jet determined.  Moreover, we prove 1-jet determinacy at any point for a variety of non-flat parabolic geometries -- in particular torsion-free, parabolic contact, and several other classes.
  \end{abstract}
 \maketitle

 %%%%%%%%%%%%%%%%%%%%%%%%%%%%%%%%%%%%%%%%%%%%%%%%%%

 \section{Introduction}

A classical problem in geometry is to determine when the local properties of an automorphism (or infinitesimally, a symmetry) constrain the geometry globally.

 \begin{defn}
A vector field $\bX$ on a manifold $M$ is $k$-jet determined at $x \in M$ if $j^k_x(\bX) \neq 0$, i.e.\ in any local coordinate system centred at $x$, the $k$-th order Taylor polynomials at x of the coefficients of $\bX$ (in the coordinate basis) are not all identically zero.  If $\cS \subset \fX(M)$ is a linear subspace, then $\cS$ is $k$-jet determined if $\{ \bX \in \cS \st j^k_x(\bX) = 0 \}$ is trivial for any $x\in M$.
 \end{defn}

For many geometric structures, their automorphisms are uniquely determined by a finite jet at a given point.  For example, an affine transformation, i.e.\ a transformation $\phi : M \to M$ that preserves geodesics of a given linear connection $\nabla$ on a (connected) manifold $M$ together with the affine parametrization, is completely determined by the value $\phi(x)$ and the differential $d_x\phi$ at any given point $x\in M$. Consequently, the same is true for isometries of Riemannian, pseudo-Riemannian and sub-Riemannian structures (in the latter case the claim is based on the Pontryagin maximum principle).  In other words, automorphisms of these geometries are determined by the 1-jet $j^1_x\phi$.

In complex analysis, an analogous statement is H.\ Cartan's uniqueness theorem \cite{H-Car1935} claiming that biholomorphic automorphisms of a bounded domain in $\bbC^n$ are uniquely determined by their 1-jets at any point $x$ inside the domain. On the boundary, the situation is more complicated.  Indeed, local automorphisms of Levi-flat real hypersurfaces are not determined by any finite jet.  On the contrary, by the classical results of \'E.\ Cartan, Tanaka, Chern and Moser \cite{Car1932,Tan1962,CM1974} a germ of a local biholomorphism in $\mathbb{C}^n$ sending a germ of a real analytic Levi-nondegenerate hypersurface to itself is uniquely determined by its 2-jet at a point.  (The corresponding problem for Levi-degenerate hypersurfaces is much more complicated \cite{ELZ}.)  For a non-degenerate quadric hypersurface, there exist automorphisms that are strictly 2-jet and not 1-jet determined.  However, in the general Levi-nondegenerate case the 2-jet is redundant.  Namely, if a strictly pseudoconvex partially integrable real-analytic hypersurface is not spherical, then its automorphisms are everywhere 1-jet determined by the Webster--Schoen theorem \cite{Web,Sch} (see also \cite{KL}). In the case the Levi form of the hypersurface has arbitrary signature, this statement is the context of Beloshapka--Loboda theorem \cite{B1980,L1982}.

In projective differential geometry, a projective structure $(M,[\nabla])$ is an equivalence class of torsion-free affine connections that share the same set of unparametrized geodesics.  Any projective symmetry is determined by its 2-jet.  If $(M,[\nabla])$ is not isomorphic near $x\in M$ to $\RP^n$ with its standard projective structure, then any projective symmetry is 1-jet determined at $x$.  This was proved by Nagano--Ochiai \cite{NO} in the case where $M$ is compact and $[\nabla]$ is Riemannian metrizable.  The proof in the general case is due to \v{C}ap--Melnick \cite{CM2013b}.

In conformal differential geometry, analogous results have been proven by Frances \cite{Frances}, Frances--Melnick \cite{FM}, and \v{C}ap--Melnick \cite{CM2013b}.  Namely, if a conformal Killing vector field vanishes to higher-order at a point, then the conformal structure is flat on an open set containing the point in its closure.  In this way, they established a local version of the Obata--Ferrand theorem.

 We note that Levi-nondegenerate hypersurface-type CR-geometry, projective differential geometry and conformal geometry are particular instances (of underlying structures) of the so-called {\em parabolic geometries} \cite{CS2009}.  For all such geometries, there is a homogeneous model $(G \to G/P, \omega_G)$ called the {\em flat model}, where $G$ is a semisimple Lie group, $P \subsetneq G$ is a parabolic subgroup, and $\omega_G$ is the left-invariant Maurer--Cartan form on $G$.  General (regular, normal) parabolic geometries $(\cG \stackrel{\pi}{\to} M, \omega)$ of type $(G,P)$ (or ``$G/P$ geometries'') are curvature deformations of the homogeneous model, and there is an equivalence of categories with underlying geometric structures (see \cite{CS2009} for more details).  In particular, there is a 1-1 correspondence between symmetries $\bX \in \fX(M)$ of the underlying geometric structure and symmetries $\xi \in \finf(\cG,\omega) \subset \fX(\cG)^P$ of the Cartan geometry $(\cG \to M, \omega)$.  We will assume throughout that $M$ is a {\em connected} manifold.

A parabolic subgroup $P \subsetneq G$ induces a ($P$-invariant) filtration $\fg = \fg^{-\nu} \supset ... \supset \fg^\nu$ (with $\nu \geq 1$), where $\fp = \fg^0$ and $\fp_+ = \fg^1$ is its nilradical.  Fix $u \in \cG$ and $x = \pi(u) \in M$.  If the symmetry $\bX$ has a fixed point at $x$, then $\omega_u(\xi) \in \fp$ is called its {\em isotropy}.\footnote{In \cite{CM2013a}, the isotropy is identified with an element of $T_x^*M$, but we will not use this identification.}  If $\omega_u(\xi) \in \fp_+$, then the fixed point $x$ is of {\em higher-order}.  In \cite{CM2013a}, \v{C}ap and Melnick initiated the program of studying higher-order fixed points of automorphisms of parabolic geometries, and they created a new technique through which they and subsequently Melnick--Neusser \cite{MN2015} obtained several new results in the field. Many of these results concern $|1|$-graded geometries (also known as generalized conformal or AHS-structures) and parabolic contact geometries, but the results for more general parabolic geometries have not been easily accessible.  Indeed, in \cite{KT2014} we observed that some parabolic geometries can admit symmetries with higher order fixed points without being flat anywhere.  More precisely, we introduced a (non-flat) submaximally symmetric path geometry \cite[eqn (5.6a)]{KT2014} with a symmetry \cite[see $\bS$ in eqn (5.6b)]{KT2014} having isotropy in $\fp_+$.  This example eliminates the hope for a bold general claim that the existence of a higher order fixed point could imply flatness of the geometry.

 In this paper we attack the problem of jet-determination for symmetries of parabolic geometries.  We show that this is related to non-existence of symmetries with isotropy in the top-slot $\fg^\nu$.  First of all, we apply the prolongation-rigidity results of \cite{KT2014} and exploit the fundamental derivative \cite{CS2009} to obtain the following result\footnote{General 2-jet determinacy of symmetry algebras of parabolic geometries was known earlier -- see Remark \ref{R:2-jet-det}.} (that, in an equivalent formulation for $|1|$-graded parabolic geometries, was already obtained in \cite{MN2015} via a different approach).
 
 \begin{thm}\label{Th-A}
The symmetry algebra $\cS$ of the underlying structure of any (real or complex) regular, normal parabolic geometry of type $(G,P)$ is everywhere 2-jet determined. Moreover, if $G$ is simple, then $\cS$ is 1-jet determined at any non-flat point.
 \end{thm}

 Here, $x \in M$ is a {\em non-flat point} if $\kappa_H(x) \neq 0$, where $\kappa_H$ is the harmonic curvature, which is the fundamental obstruction to flatness of the geometry.  (We have $\kappa_H \equiv 0$ if and only if the full curvature satisfies $\kappa \equiv 0$.) The question whether this can be extended to a point $x$ where $\kappa_H$ vanishes, but some finite jet of it does not, constitutes a more delicate problem that we study next.

We advance in this problem by combining the techniques of \cite{CM2013a} with the observation (explained in the next sections) that the only symmetries with isotropy in the top-slot $\fg^\nu$ are those that are 2-jet and not 1-jet determined.
For such symmetries the criterion of \cite{CM2013a} (see Section \ref{Ss:hovs} below) becomes much more tractable.  (This criterion is sufficient but not necessary to conclude local flatness of the geometry in the presence of higher order fixed points. One of the purposes of \cite{MN2015} was to develop stronger techniques to conclude flatness.)  This allows us to reduce the question of 1-jet determination to a purely algebraic problem in Lie algebras and representation theory.  We then develop an analogue of Kostant's orthogonal cascade of roots \cite{Kos2012} to reduce the problem to a purely combinatorial one.
This leads to the following results formulated in the analytic setting for simplicity.

 \begin{thm} Let $G$ be simple (with $\fg_\bbC$ simple if $\fg$ is real).  Consider an {\bf analytic} (real or complex) regular, normal parabolic geometry $(\cG \to M, \omega)$ of type $(G,P)$ with $M$ connected.  Suppose it is not flat, i.e.\ $\kappa_H$ is nonzero in at least one point.  Then the symmetry algebra $\cS \subset \fX(M)$ of the underlying geometric structure on $M$ is everywhere 1-jet determined in either of the following cases:\footnote{For complex Yamaguchi-nonrigid geometries: (i) torsion-free geometries are classified in Appendix \ref{S:nr-torfree}, (ii) $(G,P)$ with $\bbP(\fg^\nu)$ a single $P$-orbit are classified in Proposition \ref{P:NYR}.}
 \begin{enumerate}
 \item[(i)] The geometry is torsion-free.
 \item[(ii)] The top-slot $\fg^\nu$ contains only one nonzero $P$-orbit.
 \end{enumerate}
 \end{thm}
 
 We note that since $\fp_+$ acts trivially on $\fg^\nu$, then the $P$-orbits on $\fg^\nu$ are determined by the orbits for the reductive part $G_0 \subset P$.
 
 \begin{thm} Let $G$ be simple (with $\fg_\bbC$ simple if $\fg$ is real).  Suppose that $(G,P)$ is {\bf not} $(A_\ell, P_{s,s+1})$, $2\leq s<\frac{\ell}{2}$ or $(B_\ell,P_\ell)$ with $\ell \geq 5$ odd.  If an {\bf analytic} regular, normal parabolic geometry $(\cG \stackrel{\pi}{\to} M, \omega)$ of type $(G,P)$ (with $M$ connected) admits $\xi \in \finf(\cG,\omega)$ with isotropy $\omega_u(\xi)$ lying in the open $P$-orbit in the top-slot $\fg^\nu$.  (In particular, $\bX = \pi_*(\xi)$ has vanishing 1-jet at $x = \pi(u) \in M$.)  Then the geometry is flat.
 \end{thm}

The analyticity assumption can be relaxed to smooth with changing flatness to local flatness, see the formulation in Section \ref{S:notorsion} and \ref{S:gen}. The simplicity assumption on $G$ is crucial: a direct product of a flat $(G,P)$ geometry and a general $(G',P')$ geometry yields an example of non-flat $(G\times G', P\times P')$ geometry having a symmetry that is 2-jet and not 1-jet determined.  However, when $\fg$ is real, the assumption that $\fg_\bbC$ is simple is unnecessary and taken here only for simplification of the arguments. (For instance our proof works for c-projective structures already treated in \cite{MN2015}.)

The structure of the paper is as follows.  In Section \ref{S:flat}, we discuss the setup of parabolic geometries, and prove 2-jet determination of symmetries for all flat models.  In Section \ref{S:tractors}, we use the fundamental derivative and adjoint tractors to prove a fundamental property (Theorem \ref{T:strong-incl}) of Tanaka prolongation that strengthens our earlier key result \cite[Thm.\ 2.4.6]{KT2014}.  This is then applied in Section \ref{S:symmetries} to prove our first main result on 1-jet determination.

In Section \ref{S:subcominuscle}, we show that $\fg^\nu$ is a {\em sub-cominuscule} representation and introduce the {\em top-slot orthogonal cascade}, which yields natural representatives for the $G_0$-orbits in $\bbP(\fg^\nu)$.  These tools allow us to effectively control the ingredients of a rigidity criterion from \cite{CM2013a} and to prove 1-jet determinacy of symmetries for a wide range of non-flat parabolic geometries. This constitutes a smooth version of our second main result in Section \ref{S:rigidity} (the analytic version follows).

In Section \ref{S:conclusion}, we briefly discuss relations of our results to other investigations of symmetries for geometric structures, and outline further applications. In Appendix \ref{S:DD}, we summarize the crucial ingredients from representation theory associated to parabolic geometries that facilitates the reading of our paper.  In Appendix \ref{S:nr-torfree}, we classify all Yamaguchi-nonrigid torsion-free parabolic geometries, a result that is used in Theorem \ref{thm:notorsion}, but is of independent interest.

\medskip

{\bf Conventions:} We write $A_\ell, B_\ell, C_\ell, D_\ell, G_2, F_4, E_6, E_7, E_8$ for the complex simple Lie algebras, {\bf or} any complex Lie groups having these as their Lie algebras.  We use the Bourbaki ordering of simple roots.  If $\fh \subset \fg$ is the Cartan subalgebra, the symmetric pairing $\langle \cdot, \cdot \rangle$ on $\fh^*$ will be normalized so that all {\em short} simple roots $\alpha_i$ satisfy $\langle \alpha_i, \alpha_i \rangle = 2$.\\

{\bf Acknowledgements}:  We thank J.M.\ Landsberg for discussions on minuscule varieties and sub-cominuscule representations. We are grateful to A.\ Isaev and I.\ Kossovskiy for information about the results on stability and linearization in CR-geometry. B.K. was supported by the University of Troms\o{} while visiting the Australian National University (where this work was initiated) and the University of Vienna.  D.T. was supported by a Lise Meitner Fellowship (project M1884-N84) of the Austrian Science Fund (FWF).
 
 \section{Parabolic geometries: the flat model}
 \label{S:flat}

 \subsection{Generalities on parabolic geometries}
 \label{S:gen-parabolic}

 Let $G$ be a real or complex semisimple Lie group, $P \subsetneq G$ a parabolic subgroup, and $\fp \subsetneq \fg$ the corresponding Lie algebras.  There is a canonical $P$-invariant filtration $\fg = \fg^{-\nu} \supset ... \supset \fg^\nu$ ($\nu \geq 1$) with $\fp = \fg^0$ and $\fp_+ = \fg^1$ its nilradical.  The filtration is preserved by the Lie bracket, i.e.\ $[\fg^i,\fg^j] \subset \fg^{i+j}$, and iterated brackets of $\fg^{-1}$ generate all of $\fg$.  For convenience, we will fix a Lie algebra $\bbZ$-grading (henceforth, a ``$|\nu|$-grading'') $\fg = \fg_{-\nu} \op ... \op \fg_\nu$ satisfying $\fg^i = \bop_{j \geq i} \fg_j$, so that $\fp = \fg_{\geq 0}$, and $\fg_{-1}$ generates all of $\fg_-$.  (In particular, note that $\fg^\nu = \fg_\nu$.)  Such a grading always exists and arises from a grading element $\sfZ \in \fz(\fg_0)$.  (The specific choice of grading will not affect any of the arguments to follow.)  Non-degeneracy of the Killing form $B$ on $\fg$ induces isomorphisms $(\fg/\fp)^* \cong \fp_+$ as $P$-modules and $(\fg_i)^* \cong \fg_{-i}$ (for $i \neq 0$) as $G_0$-modules, where $G_0 \subset P$ is the reductive part.  (For more details, see \cite{CS2009}.)
 
  A parabolic geometry $(\cG\stackrel{\pi}{\to} M,\omega)$ of type $(G,P)$ consists of a $P$-principal bundle $\cG \to M$ endowed with a Cartan connection $\omega : T\cG \to \fg$, i.e.\
  \begin{enumerate}
  \item $\omega_u : T_u \cG \to \fg$ is a linear isomorphism for each $u \in \cG$;
  \item $\omega$ is $P$-equivariant, so $R_p^* \omega = \Ad_{p^{-1}} \circ \omega$ for $p \in P$, where $R_p$ denotes the right action on $\cG$;
  \item $\omega$ reproduces the fundamental vertical vector fields $\zeta_A$ for $A \in \fp$, so $\omega(\zeta_A) = A$.
  \end{enumerate}
 This has symmetries $\finf(\cG,\omega) = \{ \xi \in \fX(\cG)^P \st \cL_\xi \omega = 0 \}$.
 The curvature 2-form $K = d\omega + \frac{1}{2} [\omega,\omega]$ is horizontal, with curvature function $\kappa : \cG \to \bigwedge^2 (\fg / \fp)^* \otimes \fg \cong \bigwedge^2 \fp_+ \otimes \fg$.  The geometry is
 \begin{itemize}
 \item {\em regular} if $\kappa(\fg^i,\fg^j) \subset \fg^{i+j+1}$ for $i,j < 0$, i.e.\ $\kappa$ has positive homogeneity.
 \item {\em normal} if $\partial^* \kappa = 0$, where $\partial^*$ is the ($P$-equivariant) Kostant codifferential, which is (negative) the Lie algebra homology differential acting on chains on $\fp_+$ with values in $\fg$.
 \end{itemize}
 There is an equivalence of categories between regular, normal parabolic geometries and underlying geometric structures \cite{CS2009}.

 The fundamental invariant for regular, normal parabolic geometries is harmonic curvature $\kappa_H$, which is the $P$-equivariant function obtained by quotienting $\kappa$ by $\im(\partial^*)$.  Equivalently, $\kappa_H$ is a section of the natural vector bundle $\cV = \cG \times_P \bbV$, where $\bbV = \frac{\ker(\partial^*)}{\im(\partial^*)}$ is a completely reducible $P$-representation, i.e.\ $\fp_+$ acts trivially. 
 
 Since $\fg$ is a module for $\fg_-$, we have the Lie algebra cohomology differential $\partial$ on the space of cochains on $\fg_-$ with values in $\fg$, and the Kostant Laplacian $\Box := \partial \partial^* + \partial^* \partial$.  Both are $\fg_0$-equivariant.  There is an algebraic Hodge decomposition \cite{Kos1963} and $\fg_0$-module identification of $\bbV$ given by
 \[
 \bigwedge{\!}^2 \fg_-^* \ot \fg = \lefteqn{\overbrace{\phantom{ \im(\partial^*) \op \ker(\Box)}}^{\ker(\partial^*)}} \im(\partial^*) \op \underbrace{\ker(\Box) \op \im(\partial)}_{\ker(\partial)}, \quad
 \bbV = \frac{\ker(\partial^*)}{\im(\partial^*)} \cong \ker(\Box) \cong\frac{\ker(\partial)}{\im(\partial)} =: H^2(\fg_-,\fg).
 \]
 Since the geometry is regular, then $\kappa_H$ is valued in the subspace $H^2_+(\fg_-,\fg) \subseteq H^2(\fg_-,\fg)$ on which the grading element $\sfZ$ acts with positive eigenvalues.

 \subsection{The flat model}

 Let us first consider the flat case $\kappa \equiv 0$.  This is the homogeneous model $(G \to G/P, \omega_G)$, where $\omega_G$ denotes the left-invariant Maurer--Cartan form on $G$.

 \begin{lemma} \label{L:ss-bracket}
If $0<i<\nu$ and $0 \neq X\in\fg_i$, then $0 \neq [X,\fg_{-i-1}] \subset \fg_{-1}$.  If no simple ideal of $\fg$ is contained in $\fg_0$, i.e.\ the action of $G$ on $G/P$ is almost effective, this also holds for $i=0$.
\end{lemma}

 \begin{proof}
We may assume that $\fg$ is simple.  The Killing form $B$ is $\ad$-invariant, so $B(\ad_X(v), w) = -B(v,\ad_X(w))$ for $v \in \fg_{-i-1}$ and $w \in \fg_1$.  If $0 \neq \ad_X(w) \in \fg_{i+1}$, then choose $v \not\in \ad_X(w)^\perp$ (by non-degeneracy of $B$).  So it suffices to show that $0 \neq [X,\fg_1] \subset \fg_{i+1}$.  For $i=0$, see \cite[Prop.3.1.2(5)]{CS2009}.

Fixing $0 < i < \nu$, assume that $\cZ_i := \{ Y \in \fg_i \st [Y,\fg_1] = 0 \}$ is nontrivial.  Since $\fg_1$ is bracket-generating in $\fg_+$, then $[\cZ_i,\fg_+] = 0$.  Thus, $\cZ_i \subset \fz(\fg_+)$ is abelian.  Since $\fg_0$ consists of grading-preserving derivations of $\fg_+$, it preserves $\cZ_i$.  Hence, $\tad_\fg^\infty(\cZ_i):=\sum_{t=0}^\infty\tad_\fg^t(\cZ_i) = \sum_{t=0}^\infty\tad_{\fg_{-}}^t(\cZ_i)$ is a proper ideal in $\fg$, which contradicts $\fg$ being simple.
 \end{proof}

 Given $X \in \fg$, let $\bX = \iota(X) \in \fX(G/P)$ be the projection of the corresponding {\em right}-invariant vector field on $G$.  On $G/P$, the symmetry algebra of the differential geometric structure underlying the flat model is $\cS = \{ \bX = \iota(X) \st X \in \fg \}$.

 \begin{prop} \label{P:flat-jet} Let $G$ be a real or complex semisimple Lie group and $P \subsetneq G$ a parabolic subgroup with no simple ideal of $\fg$ contained in $\fp$.  Let $\cS \subset \fX(G/P)$ be the projection of the right-invariant vector fields on $G$.  Then $\cS$ is everywhere 2-jet (and not 1-jet) determined.   At the origin $o \in G/P$,
 \begin{align*}
 \fp = \{ X \in \fg \st j^0_o(\bX) = 0 \}, \qquad
 \fg^\nu = \{ X \in \fg \st j^1_o(\bX) = 0 \}.
 \end{align*}
 \end{prop}

 \begin{proof}  It suffices to consider $o$.  Since $P$ is the stabilizer of $o$ in $G$, then $X \in \fp$ if and only if $\bX|_o = 0$.
 
 Now let $0 \neq X \in \fg^i \subset \fp$ ($i \geq 0$ maximal) with $j^1_o(\bX) = 0$.  Write $X = \sum_{j=i}^\nu X_j$, with $X_j \in \fg_j$.  Assuming $0 \leq i < \nu$, pick $Y \in \fg_{-i-1}$ as in Lemma \ref{L:ss-bracket} so that $0 \neq [X_i,Y] \in \fg_{-1}$, so $[X - X_i,Y] \in \fp$.  Since $\bY|_o \neq 0$ and $j^1_o(\bX) = 0$, then $0 = [\bX,\bY]|_o = [\bX_i,\bY]|_o \neq 0$.  Thus, $\{ X \in \fg \st j^1_o(\bX) = 0 \} \subset \fg_\nu$.
 
 If $X \in \fg_\nu$, then $[X, \fg_-] \subset \fp$, so $[\bX, \bY]|_o = 0$ for any $Y \in \fg_-$. 
 \begin{enumerate}
 \item[(i)] In coordinates $(x^i)$ centred at $o$, take a basis $\{ Y_i \}$ of $\fg_-$ with $\bY_i = \partial_{x^i} + \bY_i'$, where $\bY_i'|_o = 0$.  (In general, $\bY_i' \not\in \cS$.)  Since $\bX|_o = 0$, then  $0 = [\bY_i,\bX]|_o = [\partial_{x^i},\bX]|_o$, so $j^1_o(\bX) = 0$.
 
 \item[(ii)] By \cite[Prop.3.1.2]{CS2009}, $\exists Y \in \fg_{-1}$ with $0 \neq [X,Y] \in \fg_{\nu-1} \subset \fp$.  Thus, $j^1_o([\bX,\bY]) \neq 0$ (see above).  Since $\bY|_o \neq 0$, then $\bY = \partial_{x^1}$ in some coordinates $(x^i)$ centered at $o$.  If $\bX = X^i \partial_{x^i}$, then $[\bY,\bX] = \frac{\partial X^i}{\partial x^1} \partial_{x^i}$.  Now, $j^1_o([\bX,\bY]) \neq 0$ implies $\frac{\partial^2 X^i}{\partial x^j \partial x^1} (o) \neq 0$ for some $i,j$, i.e.\ $j^2_o(\bX) \neq 0$.
 \end{enumerate}
 \end{proof}

 \begin{example} For the $n$-dimensional flat conformal structure $[g]$, where $g = (dx^1)^2 + ... + (dx^n)^2$, any conformal Killing vector field $\bX = X^i \partial_{x^i}$ has components
 \[
 X^i = s^i + m^i{}_j x^j + \lambda x^i + r^j x_j x^i - \frac{1}{2} r^i x^j x_j.
 \]
 Those with trivial $s^i,m^i{}_j,\lambda$ are 2-jet and not 1-jet determined at the origin.  These are the inversion symmetries.
 \end{example}

 \begin{example} The Hilbert--Cartan equation $z' = (y'')^2$ corresponds to the well-known flat model for $G_2 / P_1$ geometries.  More precisely, on a 5-manifold with coordinates $(x,y,p,q,z)$, consider the Pfaffian system $\cI = \tspan\{ du - p dx, \, dp - q dx, \, dz - q^2 dx \}$,
 or equivalently the distribution $\cD = \tspan\{
 \partial_q, \, \partial_x + p \partial_y + q\partial_p + q^2 \partial_z \}$ whose (weak) derived flag has growth vector $(2,3,5)$.  We have $\fg = \fg_{-3} \op ... \op \fg_3$ and $\dim(\fg_3) = 2$.  We can confirm Proposition \ref{P:flat-jet} in {\tt Maple} using:

 \begin{verbatim}
 > restart: with(DifferentialGeometry): with(GroupActions):
   DGsetup([x,y,p,q,z],M);
   dist:=evalDG([D_q,D_x+p*D_y+q*D_p+q^2*D_z]):
   sym:=InfinitesimalSymmetriesOfGeometricObjectFields([dist],output="list"):
   iso:=IsotropyFiltration(sym,[x=0,y=0,p=0,q=1,z=0]):
   map(nops,[sym,iso[1],iso[2]]);
 \end{verbatim}
 The final command gives output {\tt [14, 9, 2]},
which confirms that $\dim(\{ X \in \fg \st j^1_o(\bX) = 0\}) = 2$.
 \end{example}

 \begin{remark}
 By \cite{Kru2011} (see formulae (9), (10) and those in between) if $X\in\fg_i$ ($i\geq0$) in the graded Tanaka algebra
$\fg_{-\nu}\oplus\ldots\oplus\fg_i\oplus\ldots$ gives $\bX$ that is $k$-jet determined, where $k \geq 0$ is the minimal order with $j^k_o(\bX) \neq 0$, then
$\frac{i}{\nu}+1\leq k\leq i+1$.  This gives another perspective on why $\bX$ is not 1-jet determined when $X \in \fg_\nu$.
 \end{remark}

 %%%%%%%%%%%%%%%%%%%%%%%%%%%%%%%%%%%%%%%%%%%%%%%%%%
 \section{Adjoint tractors and Tanaka prolongation}
 \label{S:tractors}

 \subsection{Formulation of the main result}
 \label{S:main}

Let $\widetilde\cS = \finf(\cG,\omega) = \{ \xi \in \fX(\cG)^P \st \cL_\xi \omega = 0 \}$ be the infinitesimal symmetry algebra.  The filtration on $\fg$ corresponding to $\fp$ induces a ($P$-invariant) filtration on $T\cG$ via $T^i \cG = \omega^{-1}(\fg^i)$.  A choice of $x \in M$ induces a filtration on $\widetilde\cS$ via
 \begin{align} \label{E:S-filtration}
 \widetilde\cS(x)^i = \{ \xi \in \widetilde\cS \st \xi_u \in T^i_u \cG, \forall u \in \pi^{-1}(x) \}.
 \end{align}
 It was shown in  \cite{CN2009} that:
 \begin{enumerate}
 \item $\omega_u : T_u \cG \to \fg$ is a linear injection when restricted to $\{ \xi_u \st \xi \in \widetilde\cS \}$.
 \item the Lie bracket on $\widetilde\cS$ is mapped to the bracket on $\ff(u) := \omega_u(\widetilde\cS)$ given by:
 \[
 [X,Y]_{\ff(u)} := [X,Y]_\fg - \kappa_u(X,Y).
 \]
 \item by regularity, $\ff(u)$ is a filtered Lie algebra\footnote{While $\ff(u) \subset \fg$ is a linear subspace, it is in general {\em not} a Lie subalgebra.} upon restriction of the canonical filtration on $\fg$, while its associated graded algebra $\fs(u) = \gr(\ff(u))$ is a graded subalgebra of $\fg$.
 \item $\fs_0(u) \subseteq \fann_{\fg_0}(\kappa_H(u))$.
 \end{enumerate}

 Let us recall the {\em Tanaka prolongation algebra} \cite{KT2014}:

  \begin{defn} \label{D:g-prolong} Let $\fg$ be a $\bbZ$-graded semisimple Lie algebra, and $\fa_0 \subset \fg_0$ a subalgebra.  Define the graded subalgebra $\fa \subset \fg$ by: (i) $\fa_{\leq 0} := \fg_{\leq 0}$; (ii) $\fa_k = \{ X \in \fg_k \st [X,\fg_{-1}] \subset \fa_{k-1} \}$ for $k > 0$.  We will denote $\fa = \bop_k \fa_k$ by $\prn^\fg(\fg_-,\fa_0)$.  (In particular, $\prn^\fg(\fg_-,\fg_0) = \fg$.)  When $\phi$ is an element of some $\fg_0$-representation, we will also use the notation $\fa^\phi := \prn^\fg(\fg_-,\fann(\phi))$.
 \end{defn}

 \begin{remark} \label{R:gm1-gen}
 If $\fg_-$ is generated by $\fg_{-1}$, then $\fa_k = \{ X \in \fg_k \st \ad^k_{\fg_{-1}}(X) \subset \fa_0 \}$ (see \cite[Lemma 2.1.4]{KT2014}).
 \end{remark}

 We say that $x = \pi(u)$ is a {\em regular point} if $\dim(\fs_i(u))$ is locally constant for each $i$.  (By $P$-equivariancy of the function $\omega(\xi) : \cG \to \fg$, $\dim(\fs_i(u))$ is constant along fibres of $\cG \stackrel{\pi}{\to} M$.)  At such points, we proved in \cite[Prop.\ 2.4.3]{KT2014} a fundamental relationship between $\fs_i(u)$ (for $i > 0$) and $\fs_0(u)$.  Namely, $[\fs_{i+1}(u),\fg_{-1}] \subseteq \fs_i(u)$, which yields the inclusions
 \begin{align} \label{E:gap-incl}
 \fs(u) \subseteq \prn^\fg(\fg_-,\fs_0(u)) \subseteq \fa^{\kappa_H(u)}.
 \end{align}
  The set of regular points is open and dense in $M$, so for non-flat geometries there exists $u \in \cG$ such that $x = \pi(u)$ is a regular point and $\kappa_H(u) \neq 0$.  Since $\kappa_H(u) \in H^2_+(\fg_-,\fg)$, then \eqref{E:gap-incl} implies that $\fU := \max\{ \dim(\fa^\phi) \st 0 \neq \phi \in H^2_+(\fg_-,\fg) \}$ is a universal upper bound for $\dim(\widetilde\cS) = \dim(\fs(u))$ among all non-flat (regular, normal) geometries of type $(G,P)$.

 Our goal now is to remove the regular point assumption, thereby strengthening \cite[Thm.\ 2.4.6]{KT2014}:

 \begin{thm} \label{T:strong-incl} Let $(\cG \to M, \omega)$ be a regular, normal parabolic geometry of type $(G,P)$ with $\widetilde\cS = \finf(\cG, \omega)$ and harmonic curvature $\kappa_H$.  Given any $u \in \cG$, $\fs(u) = \gr(\omega_u(\widetilde\cS)) \subseteq \fg$ satisfies:
 \begin{align} \label{E:strong-incl}
 \fs(u) \subseteq \fa^{\kappa_H(u)}.
 \end{align}
 \end{thm}
 
 Hence, $\dim(\widetilde\cS) \leq \fU$ for non-flat (regular, normal) geometries follows.  The proof of Theorem \ref{T:strong-incl} is given in Section \ref{S:adj-tr-Kh}, but we first discuss some key properties of the fundamental derivative.

 %%%%%%%%%%%%%%%%%%%%%%%%%%%%%%%%%%%%%%%%%%%%%%%%%%

 \subsection{The fundamental derivative}  The adjoint tractor bundle is the associated bundle $\cA M = \cG \times_P \fg$, where $P$ acts on $\fg$ via restriction of the adjoint action.  Since $TM = \cG \times_P (\fg / \fp)$, there is a natural surjective bundle map $\Pi : \cA M \to TM$.  Curvature manifests as an $\cA M$-valued 2-form $\kappa \in \Omega^2(M;\cA M)$. There are two bracket operations associated with $\cA M$:
 \begin{enumerate}
 \item an {\em algebraic bracket} $\{ \cdot, \cdot \}: \cA M \times \cA M \to \cA M$, which is a bundle map making each fibre $\cA_x M$ into a Lie algebra isomorphic to $\fg$.
 \item a Lie bracket $[\cdot,\cdot]$ on $\Gamma(\cA M)$ induced via the isomorphism $\Gamma(\cA M) \cong \fX(\cG)^P$.  (See \cite[Prop.1.5.7]{CS2009}; in particular, any symmetry $\xi \in \finf(\cG,\omega)$ corresponds to some $s \in \Gamma(\cA M)$.)
 \end{enumerate}

 Given a representation $\rho : P \to \tGL(\bbE)$, consider the natural vector bundle $EM = \cG \times_P \bbE$.  The fundamental derivative $D : \Gamma(\cA M) \times \Gamma(EM) \to \Gamma(EM)$, $(r,\sigma) \mapsto D_r \sigma$ is defined as follows: Given that $r \in \Gamma(\cA M)$ corresponds to a $P$-invariant vector field $\xi \in \fX(\cG)$ and $\sigma \in \Gamma(EM)$ corresponds to a $P$-equivariant function $\phi : \cG \to \bbE$, then $D_r\sigma$ corresponds to $\xi \cdot \phi$.  In particular, the canonical $P$-invariant filtration on $\fg$ induces a filtration of $\{ \cA^i M \}$ of $\cA M$, and for any $r\in\Gamma(\cA M)$, we have
 \begin{equation}\label{D_tA}
D_r:\Gamma(\cA^i M)\to\Gamma(\cA^i M).
 \end{equation}

 \begin{framed}
Notation: Let $a,b \in \Gamma(EM)$, which we often identify with $P$-equivariant functions $a,b : \cG \to \bbE$.  Fix $u \in \cG$.  Writing $a \deq b$ will mean $a(u) = b(u)$.  (This property depends only on $x = \pi(u)$.)
 \end{framed}

 The following facts will be useful for the proof of Theorem \ref{T:strong-incl}.

 \begin{lemma}[Properties of the fundamental derivative] \label{L:fund-der} Fix $u \in \cG$ and $r,t \in \Gamma(\cA M)$.  Then
 \begin{enumerate}
 \item[(D.1)] $D_r$ is tensorial in $r$, i.e. $D_{fr} = f D_r$ for $f \in C^\infty(M)$.
 \item[(D.2)] $D_{[r,t]} = [D_r,D_t]$.
 \item[(D.3)] $D_r$ is a derivation of the algebraic bracket $\{ \cdot, \cdot \}$.
 \item[(D.4)] $[r,t] = D_r t - D_t r - \kappa(\Pi(r),\Pi(t)) + \{ r,t \}$.
 \item[(D.5)] If $r(u) \in \fp$, then $D_r t \deq \{ t, r \}$.
 \item[(D.6)] If $r(u) \in \fp$, then $D_t r \deq [t,r]$.
 \item[(D.7)] Let $\cV = \cG \times_P \bbV$, where $\bbV = \frac{\ker(\partial^*)}{\im(\partial^*)}$.  (Note $\kappa_H \in \Gamma(\cV)$.)  If $r(u) \in \fp$ and $\psi \in \Gamma(\cV)$, then $(D_r \psi)(u) = -r(u) \cdot \psi(u)$.  Thus, if $r(u) \in \fg^1$, then $D_r \psi \deq 0$ by complete reducibility of $\bbV$.
 \item[(D.8)] Suppose that $s \in \Gamma(\cA M)$ is a symmetry. Then $D_s t = [s,t]$, $D_s \kappa = 0$, $D_s \kappa_H = 0$.
 \end{enumerate}
 \end{lemma}

 \begin{proof}
 (D.1)-(D.5), (D.7) are proved in \cite[Sec.\ 1.5.8]{CS2009}, while (D.6) follows from (D.4) and (D.5).  Finally, (D.8) follows from \cite[Lemma 1.5.12]{CS2009}.
 \end{proof}

  We use the following notation: Given the ordered $n$-tuple $J = (t_n,...,t_1)$, with $t_j \in \Gamma(\cA M)$, define
 \begin{align*}
 D_J := D_{t_n} \cdots D_{t_1}, \qquad
 [J] := [t_n,[t_{n-1},\cdots [t_2,t_1]\cdots]], \qquad
 \{ J \} := \{ t_n,\{ t_{n-1},\cdots \{ t_2,t_1 \}\cdots\} \}.
 \end{align*}
 Write $s \in J$ if and only if $s = t_j$ for some $j$, and let $K \subset J$ if and only if $K = (t_{r_m},...,t_{r_1})$, where $1 \leq r_1 < ... < r_m \leq n$, i.e. $K$ is an {\em ordered} subset of $J$. Similarly, we abuse notation so that $J \backslash K$ makes sense as an ordered set in the obvious way.
 
  \begin{lemma} \label{L:DI} 
Let $t_j, r \in \Gamma(\cA M)$, and $I = (t_n,...,t_1,r)$.  
Then $D_I = \displaystyle\sum_{r \in A \subset I} D_{[A]} D_{I \backslash A}$.
 \end{lemma}

 \begin{proof} The $n=1$ case follows from (D.2).
Suppose it holds for $n-1$.  Let $J = I \backslash \{ t_n \}$.  Then
 \begin{align*}
 D_I = D_{t_n} D_J = \sum_{r \in B \subset J} D_{t_n} D_{[B]} D_{J \backslash B} = \sum_{r \in B \subset J} D_{[t_n, [B]]} D_{J \backslash B} + \sum_{r \in B \subset J} D_{[B]} D_{t_n} D_{J \backslash B} =  \sum_{r \in A \subset I} D_{[A]} D_{I \backslash A}.
 \end{align*}
 \end{proof}

 %%%%%%%%%%%%%%%%%%%%%%%%%%%%%%%%%%%%%%%%%%%%%%%%%%

 \subsection{Adjoint tractors and harmonic curvature}
 \label{S:adj-tr-Kh}

% By Remark \ref{R:gm1-gen}, the following fundamental property establishes Theorem \ref{T:strong-incl} in the case $i_j = 1$ for all $j$.  

 \begin{thm} \label{T:sip} 
Let $(\cG \to M, \omega)$ be a regular, normal parabolic geometry of type $(G,P)$.  Fix $u \in \cG$. Let $s \in \Gamma(\cA M)$ be a symmetry with $s(u) \in \fg^i \subset \fp$.  Let $i_j \in \bbZ_+$ and $t_j\in \Gamma(\cA^{-i_j} M)$ such that $i - i_1 - ... - i_n \geq 0$ and $I = (t_n,...,t_1,s)$.  Then:
 \begin{align} \label{E:sip}
 \{ I \}(u) \cdot \kappa_H(u) = 0.
 \end{align}
 \end{thm}
 
 This fundamental property implies Theorem \ref{T:strong-incl} as a corollary: If $X \in \fs_k(u)$ with $k > 0$ and $Y_1,...,Y_k \in \fg_{-1}$, take a symmetry $s \in \Gamma(\cA M)$ and $t_1,...,t_k \in \Gamma(\cA^{-1} M)$ with $s(u) = X$ and $t_j(u) = Y_j$.  Then \eqref{E:sip} implies that $\ad_{Y_k} \circ ... \circ \ad_{Y_1}(X)$ annihilates $\kappa_H(u)$.  By Remark \ref{R:gm1-gen}, $X \in \fa^{\kappa_H(u)}_k$.

 The $n=0$ case of Theorem \ref{T:sip} is true since $\fs_0(u) \subset \fann_{\fg_0}(\kappa_H(u))$ and $\fp_+ \cdot \kappa_H(u) = 0$.
 For $n=1$, $i \geq i_1 \geq 1$.  Since $D_s \kappa_H = 0$, then (D.2) implies
 \[
 0 = D_{t_1} D_s \kappa_H = D_s D_{t_1} \kappa_H + D_{[t_1,s]} \kappa_H.
 \]
 By (D.7), $D_s (D_{t_1} \kappa_H) \deq 0$ since $s(u) \in \fg^{i} \subseteq \fp_+$.
 By (D.8) and (D.5), $[t_1,s] = -D_s t_1 \deq -\{ t_1, s \}$.  At $u$, this lies in $\fp$.  By (D.7), $(D_{[t_1,s]} \kappa_H)(u) = \{ t_1, s \}(u) \cdot \kappa_H(u)$, so \eqref{E:sip} holds for $n=1$.  (For $|1|$-graded geometries, the proof finishes here.)  The general case will proceed by induction on $n$. However, we first introduce some notation and prove a technical lemma.

%
%\begin{framed}
%For $n=2$, $i \geq i_1 + i_2 \geq 2$.  We have
%\[
% 0 = D_{t_2} D_{t_1} D_s \kappa_H = D_{[t_2,[t_1,s]]} \kappa_H + D_{[t_1,s]} D_{t_2} \kappa_H + D_{[t_2,s]} D_{t_1} \kappa_H + D_s D_{t_2} D_{t_1} \kappa_H
%\]
% For $j=1,2$, we have $i \geq i_j \geq 1$, so from the $n=1$ case, $[t_j,s] \deq -\{ t_j, s \}$.  At $u$, this lies in $\fp_+$ and $s(u) \in \fp_+$, so $0 \deq D_{[t_2,[t_1,s]]} \kappa_H$ by (D.7).  By (D.4) and (D.5),
% \begin{align*}
% [t_2,[t_1,s]] &= D_{t_2} [t_1,s] - D_{[t_1,s]} t_2 - \kappa(\Pi(t_2), \Pi([t_1,s])) + \{ t_2, [t_1,s] \} \deq D_{t_2} [t_1,s] = -D_{t_2} D_s t_1\\
% &\deq -D_s D_{t_2} t_1 - D_{[t_2,s]} t_1 = -D_s D_{t_2} t_1 + \{ \{ s, t_2 \}, t_1 \} \quad\mbox{(by (D.5))}
% \end{align*}
% We have $(D_s D_{t_2} t_1)(u) \deq -s(u) \cdot (D_{t_2} t_1)(u) \in \fg^{i - i_1} \subset \fp_+$.  Also, $\{ s, \{ t_1, t_2 \} \}(u) \cdot \kappa_H(u) = 0$ from the $n=1$ case applied to $\{ t_1, t_2 \}$.  Thus, $[t_2,[t_1,s]](u) - \{ t_2, \{ t_1, s \} \}(u)$ consists of terms in $\fp$ that annihilate $\kappa_H(u)$.  From $0 \deq D_{[t_2,[t_1,s]]} \kappa_H$, \eqref{E:sip} holds in the $n=2$ case.
% \end{framed}
% 

 \begin{framed}
 Notation: Fix $q \geq 0$, $n \geq 1$, and $s(u)\in\fg^i$, $i>0$.  Denote 
 \[
 \cP^s_{n,q} = \tspan\left\{ \ad_{X_k} \circ ... \circ \ad_{X_1}(s(u)) \in \fg^q \st X_1,...,X_k\in \fg, \, k < n \right\} + \fg^{q+1}.
 \]
 Write $a\,\stackrel{n,q}\equiv\,b$ if $a(u)-b(u)\in\cP^s_{n,q}$.
 Note that $\cP^s_{n',q'}\subset\cP^s_{n,q}$ if $n'\leq n$, $q'\geq q$, so
 \begin{equation}\label{anqb}
a\,\stackrel{n',q'}\equiv\,b\ \Rightarrow\ a\,\stackrel{n,q}\equiv\,b\quad
\text{ for }n'\leq n,\ q'\geq q.
 \end{equation}
 \end{framed}

 \begin{lemma} \label{L:[]} With notations as in Theorem \ref{T:sip}, let $q = i - i_1 - ... - i_n \geq 0$.  Then $[I]\,\stackrel{n,q}\equiv\,(-1)^n \{I\}$.
 \end{lemma}

 \begin{proof} For $k\le n$ and $J_k = (t_k,...,t_1,s)$, let $v_k = [J_k]$, $w_k = \{ J_k \}$.  We use induction on $n$ to prove:\\
 
{\bf Claim $1_n$.} For $1<k\leq n$, $[I] \stackrel{n,q}\equiv D_{t_n} \cdots D_{t_k} v_{k-1}$.

{\bf Claim $2_n$.} $v_n \stackrel{n,q}\equiv (-1)^n w_n$.\\

$n=1$: Claim $1_1$ is vacuous, while for Claim $2_1$, $[t_1,s] = -D_s t_1 \deq -\{ t_1,s \}$ was shown above.

Induction hypothesis: Claims $1_{n'}$, $2_{n'}$ hold for all $n'<n$.

 \medskip\noindent
 {\bf Proof of Claim $1_n$:} Use decreasing induction on $k$.
For the base case $k=n$, use (D.4):
 \[
v_n = [t_n,v_{n-1}] = D_{t_n} v_{n-1} - D_{v_{n-1}} t_n - \kappa(\Pi(t_n),\Pi(v_{n-1})) + \{ t_n, v_{n-1} \} \deq D_{t_n} v_{n-1}.
 \]
The last equality follows from Claim $2_{n-1}$: Since $\iota_n=q+i_n>0$, then $v_{n-1}\stackrel{n-1,\iota_n}{\equiv}(-1)^{n-1} w_{n-1} \din \fg^{\iota_n}$.  Thus, $v_{n-1}\din \fp_+$, $\Pi(v_{n-1}) \deq 0$ and by (D.5), $D_{v_{n-1}} t_n \deq \{ t_n, v_{n-1} \}$.

Now suppose that the result holds for $k+1\leq n$.  We will show that it holds for $k$.
 \begin{align}
D_{t_n}\cdots D_{t_{k+1}}v_k &=D_{t_n}\cdots D_{t_{k+1}}[t_k,v_{k-1}] \nonumber\\
&=D_{t_n}\cdots D_{t_{k+1}}(D_{t_k}v_{k-1}-D_{v_{k-1}}t_k
-\kappa(\Pi(t_k),\Pi(v_{k-1}))+\{t_k,v_{k-1}\}) \nonumber\\
 &=: {\rm I} - {\rm II} - {\rm III} + {\rm IV} \label{E:I-IV}
 \end{align}

 Given an ordered subset $B = (t_{r_m},...,t_{r_1}) \subset (t_n,\dots,t_{k+1})$, consider $D_B v_{k-1}$.  Let $n':=m+k-1 < n$ and $q':= q+\sum_{j=k}^n i_j-\sum_{a=1}^m i_{r_a} \geq q + i_k > q$.  By Claim $1_{n'}$ and $2_{n'}$, we have
 \begin{align} \label{E:DB}
 D_B v_{k-1} \stackrel{n',q'}{\equiv} [t_{r_m},\cdots [t_{r_1},v_{k-1}],\cdots] 
  \stackrel{n',q'}{\equiv} (-1)^{n'} \{ B \sqcup w_{k-1} \} \din \fg^{q'} \subset \fp_+.
 \end{align}

For the second term of \eqref{E:I-IV}, we have by Lemma \ref{L:DI}:
 \begin{align*}
 {\rm II}
 &= D_{t_n}\cdots D_{t_{k+1}} D_{v_{k-1}}t_k = \sum_{A \sqcup B = (t_n,...,t_{k+1})} D_{[B \sqcup v_{k-1}]} D_A t_k.
 \end{align*}
 At $u$, given $B \subset (t_n,\dots,t_{k+1})$ and $q'$ as above, we have
 $[B \sqcup v_{k-1}] \stackrel{n',q'}{\equiv} (-1)^{n'} \{ B \sqcup w_{k-1} \}$.  Now use (D.5) on $D_{[B \sqcup v_{k-1}]} D_A t_k$.  At $u$, both $\{ \{ B \sqcup w_{k-1} \}, t_k \}$ and $\{ \cP^s_{n',q'}, D_A t_k \}$ lie in $\fg^{q'-i_k}$ by \eqref{D_tA}. If $B \neq (t_n,\dots,t_{k+1})$, these terms lie in $\fg^{q+1} \subset \cP^s_{n,q}$ since $q' - i_k > q$.  Hence, we obtain:
 \begin{align*}
 {\rm II} &\stackrel{n,q}{\equiv} \{t_k,(-1)^{n-1}\{t_n,\{t_{n-1},\dots\{t_{k+1},w_{k-1}\}\}...\}\}.
 \end{align*}

 For the third term of \eqref{E:I-IV}, we use the Leibniz rule $D_r(\kappa(t_1,t_2)) = (D_r\kappa)(t_1,t_2) + \kappa(D_r t_1, t_2) + \kappa(t_1, D_r t_2)$.  Given $B$ as above, \eqref{E:DB} implies $\Pi(D_{t_{r_m}}\cdots D_{t_{r_1}}v_{k-1})\deq0$.  Since $\kappa$ is horizontal, then $D_t \kappa$ and its further iterates are horizontal.  Hence, ${\rm III}=D_{t_n}\cdots D_{t_{k+1}}(\kappa(\Pi(t_k),\Pi(v_{k-1}))) \deq 0$.
 
 For the fourth term of \eqref{E:I-IV}, use (D.3), \eqref{E:DB}, and \eqref{anqb} to obtain
 \begin{align*}
{\rm IV}&=\sum_{A\sqcup B=(t_n,\dots,t_{k+1})}\{D_At_k,D_Bv_{k-1}\}
 \stackrel{n,q}{\equiv} \sum(-1)^{|B|+k-1}\{D_At_k,\{B\sqcup w_{k-1}\}\}\\
& \stackrel{n,q}{\equiv} (-1)^{n-1}\{t_k,\{t_n,\dots\{t_{k+1},w_{k-1}\}...\}\} 
\qbox{(if $A \neq\emptyset$, then $\{ B \sqcup w_{k-1} \} \din \fg^{q+1}$).}
 \end{align*}
Thus, ${\rm IV} \stackrel{n,q}{\equiv} {\rm II}$, and we conclude $D_{t_n}\cdots D_{t_{k+1}}v_k\,\stackrel{n,q}{\equiv}\,D_{t_n}\cdots D_{t_{k+1}}D_{t_k}v_{k-1}$, proving the claim.

 \medskip\noindent
 {\bf Proof of Claim $2_n$:} We have:
 \begin{align*}
[I] &\stackrel{n,q}{\equiv} D_{t_n}\cdots D_{t_2}[t_1,s] \qbox{(by Claim $1_n$)}\\
&=-D_{t_n}\cdots D_{t_2}D_st_1 = -D_{[t_n,\dots[t_2,s]...]}t_1+\dots \qbox{(by Lemma \ref{L:DI})}\\
 &\stackrel{n,q}{\equiv} (-1)^n\{t_1,\{t_n,\dots\{t_2,s\}...\}\}\qbox{(by Claim $2_{n-1}$ and (D.5))}\\
&=(-1)^n(\{\{t_1,t_n\},\{t_{n-1},\dots\{t_2,s\}...\}\}+\{t_n,\{t_1,\{t_{n-1},\dots\{t_2,s\}\}...\}\})\\
 &\stackrel{n,q}{\equiv} (-1)^n\{t_n,\{t_1,\{t_{n-1},\dots\{t_2,s\}\}...\}\} \qbox{(by the definition of $\cP^s_{n,q}$)}\\
 &\stackrel{n,q}{\equiv} ... \stackrel{n,q}{\equiv} (-1)^n\{t_n,\dots\{t_2,\{t_1,s\}...\}\}=(-1)^n\{I\}.
 \end{align*}

 \end{proof}

 We now return to proving Theorem \ref{T:sip}.  Assuming the induction hypothesis, we will prove it for general $n$.  By Lemma \ref{L:DI}, we have $0=D_I\kappa_H=\sum_{s\in A\subset I}D_{[A]}D_{I\setminus A}\kappa_H$.  Fixing any such $A = (t_{r_m},...,t_{r_1},s)$, let $n' = m$ and $q' = i - i_{r_1} - ... - i_{r_m}$.  Lemma \ref{L:[]} implies $[A] \stackrel{n',q'}{\equiv} (-1)^{n'} \{ A \}$. 
 \begin{itemize}
 \item If $A \neq I$, then $n' < n$ and $q' > q \geq 0$.  By the induction hypothesis, $0 = \{A\}\cdot\kappa_H$.  Moreover, $\cP^s_{n',q'} \subset \fp_+$, so these terms also annihilate $\kappa_H(u)$.  (So by (D.7), we ignore such terms.)
 \item If $A = I$, then $n'=n$ and $q' = q \geq 0$.  We have $\cP^s_{n,q} \subset \fp$ and by the induction hypothesis, $\cP^s_{n,q} \cdot \kappa_H(u) = 0$.  (In the definition of $\cP^s_{n,q}$, we had the strict inequality $k < n$.  For each $X_j \in \fg^{-i_j}$, take the ($P$-equivariant) $\omega$-constant vector field $\omega^{-1}(X_j)$ which induces a section $t_j \in \Gamma(\cA^{-i_j}M)$.)  Again by (D.7), the $\cP^s_{n,q}$ terms can be ignored.
 \end{itemize}
 Finally, we have
 \[
 0 = D_I \kappa_H \deq D_{[I]} \kappa_H \deq (-1)^n D_{\{I\}}\kappa_H\deq(-1)^{n-1}\{I\}\cdot\kappa_H.
 \]
 This finishes the proof of Theorem \ref{T:sip} and hence of Theorem \ref{T:strong-incl}.

 \subsection{The symmetry algebra is 2-jet determined}
 
 Recall that $\fX(\cG)^P \cong \Gamma(\cA M)$ via sending $\eta \in \fX(\cG)^P$ to the $P$-equivariant function $\omega(\eta)$ from $\cG$ to $\fg$ and this corresponds to a section of $\cA M$.  The following statement generalizes Proposition \ref{P:flat-jet}.
 
 \begin{thm} \label{T:M-bracket}  Let $G$ be a semisimple Lie group, $P$ a parabolic subgroup, $\fg = \fg^{-\nu} \supset ... \supset \fg^\nu$ the $P$-invariant filtration with $\fp = \fg^0$.  Assume that no simple ideal of $\fg$ is contained in $\fp$.  Let $(\cG \stackrel{\pi}{\to} M, \omega)$ be a regular, normal geometry of type $(G,P)$.  Let $0 \neq \xi \in \finf(\cG,\omega)$ and $\bX = \pi_*(\xi)$.  Fix $u \in \cG$, $x = \pi(u)$, and $\omega_u(\xi) \in \fg^i \backslash \fg^{i+1}$.  Then
 \begin{enumerate}
 \item[(a)] $i < 0$: $j^0_x(\bX) \neq 0$
 \item[(b)] $0 \leq i < \nu$: $j^0_x(\bX) = 0$ and $j^1_x(\bX) \neq 0$.
 \item[(c)] $i = \nu$: $j^1_x(\bX) = 0$ and $j^2_x(\bX) \neq 0$.
 \end{enumerate}
 Consequently, the symmetry algebra $\cS = \pi_*(\finf(\cG,\omega))$ is 2-jet determined.
 \end{thm}
 
 \begin{proof}
 Part (a) is clear, so suppose $i \geq 0$.  Then $j^0_x(\bX) = 0$ is clear.  
 Let $s \in \Gamma(\cA M)$  correspond to the symmetry $\xi$.
 By (D.8) in Lemma \ref{L:fund-der}, $D_s t = [s,t]$ for $t \in \Gamma(\cA M)$, so $[t,s](u) = \{ s, t \}(u)$.  Let $A = \omega_u(\xi)$.  As in \S \ref{S:gen-parabolic}, we introduce a grading $\fg = \fg_{-\nu} \op ... \op \fg_\nu$ with $\fg^i = \bop_{j \geq i} \fg_j$.
 
 Suppose $0 \leq i < \nu$. By Lemma \ref{L:ss-bracket}, $\exists B \in \fg_{-i-1}$ with $0 \neq [A,B] \in \fg^{-1} \backslash \fp$.  Pick $t \in \Gamma(\cA^{-i-1} M)$ with $t(u) = B \neq 0$.  Then $[t,s](u) = \{ s, t \}(u) = [A,B] \neq 0$.  Hence, the projection $\bY \in \fX(M)$ of the vector field corresponding to $t$ satisfies $[\bX,\bY]|_x \neq 0$ and $\bY|_x \neq 0$, so $j^1_x(\bX) \neq 0$.  Thus, (b) is proved.

Suppose $i = \nu$.  Given any $t \in \Gamma(\cA M)$, let $B := t(u)$.  Then $[t,s](u) = \{ s,t \}(u) = [A,B] \in \fp$, and so $[\bX,\bY]|_x = 0$.  Since $t$ (hence $\bY$) was arbitrary, then $j^1_x(\bX) = 0$.  We will now prove $j^2_x(\bX) \neq 0$.  By \cite[Prop.3.1.2]{CS2009}, $\exists B \in \fg_{-1}$ with $0 \neq [A,B] \in \fg^{\nu-1} \backslash \fg^\nu$ and choose any $t_1 \in \Gamma(\cA^{-1} M)$ with $t_1(u) = B$.  By Lemma \ref{L:ss-bracket},  $\exists C \in \fg_{-\nu}$ with $0 \neq [C,[B,A]] \in \fg^{-1} \backslash \fp$.  Choose any $t_2 \in \Gamma(\cA M)$ with $t_2(u) = C$.  Note that $[t_i,s](u) = \{ s, t_i \}(u) \in \fp$.  Using (D.4), (D.5), (D.2) from Lemma \ref{L:fund-der}, we have at $u$:
 \begin{align*}
 [t_2,[t_1,s]] &= D_{t_2} [t_1,s] - D_{[t_1,s]} t_2 - \kappa(\Pi(t_2), \Pi([t_1,s])) + \{ t_2, [t_1,s] \} \deq D_{t_2} [t_1,s] = -D_{t_2} D_s t_1\\
 &\deq -D_s D_{t_2} t_1 - D_{[t_2,s]} t_1 \deq -D_s D_{t_2} t_1 + \{ \{ s, t_2 \}, t_1 \}
 \end{align*}
 By \eqref{D_tA}, $D_{t_2} t_1 \in \Gamma(\cA^{-1} M)$, so $(D_s D_{t_2} t_1)(u) = -\{ s, D_{t_2} t_1 \}(u) \in \fg^{\nu-1} \subset \fp$.  Also, $\{ \{ s, t_2 \}, t_1 \}(u) = [[A,C],B] = [[A,B],C]$ since $[C,B] = 0$.  Thus, $0 \neq \{ \{ s, t_2 \}, t_1 \}(u) \in \fg^{-1} \backslash \fp$.  Hence, $\bY_i \in \fX(M)$ corresponding to $t_i$ satisfy $[\bY_2,[\bY_1,\bX]]|_x \neq 0$, with $[\bY_1,\bX]|_x = 0$ and $\bY_2|_x \neq 0$, so $j^2_x(\bX) \neq 0$.
 \end{proof}

 \begin{remark} \label{R:2-jet-det}
The result that symmetries are 2-jet determined in the $|1|$-graded case is classical -- see for instance \cite[Sec.\ 1.8]{Kob1972}; it also follows from \cite[Cor.\ 2.3]{BCEG2006} since the infinitesimal symmetry equation is an instance of a first BGG operator associated to the adjoint tractor bundle.  In the higher-graded case, the result apparently has not been explicitly stated anywhere, but K.\ Neusser \cite{Neusser2012} related the prolongations of BGG operators to weighted jets of solutions, and using these tools one can similarly derive general 2-jet determinacy for symmetries. Our proof of Theorem \ref{T:M-bracket} is completely independent of these developments.
 \end{remark}

 %%%%%%%%%%%%%%%%%%%%%%%%%%%%%%%%%%%%%%%%%%%%%%%%%%

 \section{Symmetries are 1-jet determined at non-flat points}
 \label{S:symmetries}

 \subsection{Prolongation height}
 \label{S:height}

 Given a parabolic subalgebra $\fp \subset \fg$, we have an associated $|\nu|$-grading $\fg = \fg_{-\nu} \op ... \op \fg_\nu$ corresponding to the eigenspaces of the grading element $\sfZ$.  We refer to $\nu \geq 1$ as the {\em depth} of the gradation.  Given an element $\phi$ in a $\fg_0$-representation, the {\em height} of the gradation on $\fa^\phi = \prn^\fg(\fg_-,\fann(\phi)) \subset \fg$ is the maximal $\tilde\nu = \tilde\nu(\phi) \geq 0$ such that $0 \neq \fa^\phi_{\tilde\nu} \subset \fg_{\tilde\nu}$.

 \begin{thm} \label{T:height} Let $\fg$ be a real or complex {\bf simple} Lie algebra, and $\fp \subsetneq \fg$ a parabolic subalgebra.  Let $0 \neq \phi \in H^2_+(\fg_-,\fg)$, and $\fa^\phi = \prn^\fg(\fg_-,\fann(\phi))$.  Then $0 \leq \tilde\nu(\phi) < \nu$, i.e. $\fa^\phi_\nu = 0$.
 \end{thm}

 Given $\phi \in H^2_+(\fg_-,\fg)$, decompose $\phi = \sum_i \phi_i$, where each $\phi_i$ lies in some $\fg_0$-irreducible component of $H^2_+(\fg_-,\fg)$.  Then $\fann(\phi) = \bigcap_i \fann(\phi_i)$ and $\fa^\phi \subset \bigcap_i \fa^{\phi_i}$.  Thus, it suffices to consider $0 \neq \phi \in \bbU$, where $\bbU \subset H^2_+(\fg_-,\fg)$ is $\fg_0$-irrep.  If $\fg$ is real, then under complexification, the height can only increase (since the number of irreps may increase).  Thus, it suffices to consider the complex simple case, except when the complexification $\fg_\bbC$ is not simple.  This occurs if and only if $\fg$ is the underlying real Lie algebra of a complex simple Lie algebra.  This case will be treated separately.

 First consider the complex simple case.  This follows immediately from \cite[Theorem 3.6.1]{KT2014}, but since this relied on the detailed classifications of \cite[Section 3.4]{KT2014}, we instead give a simpler and uniform proof here. Let $\lambda$ be the highest weight of $\fg$.  (See Appendix \ref{S:DD} for various notations below.)  By Kostant's theorem, each irreducible $\fg_0$-submodule $\bbU_\mu \subset H^2(\fg_-,\fg)$ has lowest weight $\mu = -w \cdot \lambda$, where $w \in W^\fp(2)$.  We have $\bbU_\mu \subset H^2_+(\fg_-,\fg)$ if and only if $\sfZ(\mu) > 0$.  By \cite[Lemma 3.1.1]{KT2014}, if $0 \neq \phi \in \bbU_\mu$, then $\dim(\fa^{\phi}_r) \leq \dim(\fa^{\phi_0}_r)$ for any $0 \leq r \leq \nu$, where $\phi_0$ is a lowest weight vector in $\bbU_\mu$.  Thus, it suffices to study the height of $\fa^{\phi_0}$, and structural properties thereof were established in \cite[Section 3.2]{KT2014}.  Given $\ell = \rnk(\fg)$, let $I_\fp = \{ i \st \fg_{-\alpha_i} \not\subset \fp \} \subset \{ 1,..., \ell \}$ and define
 \[
 I_w = \{ j \in I_\fp \st \langle  \mu, \alpha_j^\vee \rangle = 0 \}, \qquad
 J_w = \{ j \not\in I_\fp \st \langle \mu, \alpha_j^\vee \rangle \neq 0 \}.
 \]
 In terms of the dual basis $\{ \sfZ_i \}$ to the simple roots $\{ \alpha_j \}$, let $\sfZ_S = \sum_{a \in S} \sfZ_a$ for any subset $S \subset \{ 1, ..., \ell \}$.  The grading element is $\sfZ = \sfZ_{I_\fp}$. Then \cite[Thm.~3.2.6]{KT2014} describes each $\fa_r^{\phi_0}$ $(r > 0$) as a sum of root spaces corresponding to the roots
 \begin{align} \label{E:phi0-rt-dec}
 \Delta(\fa_r^{\phi_0}) = \{ \alpha \in \Delta(\fg_r) \st \sfZ_{I_w}(\alpha) = r, \,\, \sfZ_{J_w}(\alpha) = 0 \}, \qquad r > 0.
 \end{align}

With the above simplifications, let $\alpha \in \Delta(\fa^{\phi_0}_{\tilde\nu})$.  By \eqref{E:phi0-rt-dec}, $\tilde\nu = \sfZ_{I_w}(\alpha)$ and $\sfZ_{I_\fp\backslash I_w}(\alpha) = 0$.  Since $\alpha \leq \lambda$ and $I_w \subset I_\fp$, then $\tilde\nu = \sfZ_{I_w}(\alpha) \leq \sfZ_{I_w}(\lambda) \leq \sfZ_{I_\fp}(\lambda) = \nu$.  Thus, $\tilde\nu = \nu$ only if $I_w = I_\fp$.  We will show that this is impossible. Write $w = \sigma_j \sigma_k \in W^\fp(2)$, so $j \in I_\fp$ and $k \neq j$ (by properties (P.1), (P.2) in Appendix \ref{S:DD}).  Let $\lambda = \sum_i r_i \lambda_i$, so $r_i = \langle \lambda, \alpha_i^\vee\rangle \geq 0$.  By \eqref{E:mu}, we have
 \[
  \langle \mu, \alpha_j^\vee \rangle = -r_j + 2(r_j + 1) + (r_k+1)(c_{kj} - 2 c_{kj} ) = r_j + 2 - (r_k+1) c_{kj} \geq 2,
 \]
 since $c_{kj} = \langle \alpha_k, \alpha_j^\vee \rangle \leq 0$.  Thus, $j \not\in I_w$ and consequently, $I_w \subsetneq I_\fp$ and $\tilde\nu < \nu$.

Now consider real $\fg$ such that $\fg_\bbC$ is semisimple.  In \cite[Sec.2]{KMT2016}, we developed analogous tools to study this case.  Given any (real) $\fg_0$-irrep $\bbU \subset H^2_+(\fg_-,\fg)$, there exists a $\fg_\bbC$-weight $\mu$ such that the complexification of $\bbU$ is either $\bbW_\mu \cong \overline{\bbW_\mu}$ or $\bbW_\mu \op \overline{\bbW_\mu}$ (if $\bbW_\mu \not\cong \overline{\bbW_\mu}$).  For the real Lie algebra $\fa^\phi$, we proved in \cite[Prop.1]{KMT2016} that if $\phi_0 \in \bbW_\mu$ is extremal, then for $k \geq 0$ and any $0 \neq \phi \in \bbW_\mu$, we have:
 \begin{enumerate}
 \item[(i)] if $\bbW_\mu \cong \overline{ \bbW_\mu}$: $\dim(\fa_k^{\phi}) \leq \dim(\fa_k^{\phi_0})$;
 \item[(ii)] if $\bbW_\mu \not\cong \overline{ \bbW_\mu}$: $\dim(\fa_k^{\phi + \bar\phi}) \leq \dim(\fa_k^{\phi_0 + \overline{\phi_0}})$.
 \end{enumerate}
 For (i), $\fa^{\phi_0}$ will be a real Lie algebra underlying a complex one and we can use the complex case above to obtain the claimed result.  For (ii), we consider the complexification $\fg_\bbC \cong \fg \times \fg$ (where we regard $\fg$ as complex on the right-hand side).  Then $\dim_\bbR(\fa_k^{\phi_0 + \overline\phi_0})$ is bounded above by $2\,\dim_\bbC(\fa_k^{\phi_0})$, where $\fa^{\phi_0} \subset \fg$.  Again, use the complex case above to obtain the claimed result.
This completes the proof of Theorem \ref{T:height}.

 %%%%%%%%%%%%%%%%%%%%%%%%%%%%%%%%%%%%%%%%%%%%%%%%%%

 \subsection{General 1-jet determinacy}
 
 \begin{thm} \label{T:nf-jet} Let $G$ be a real or complex {\bf simple} Lie group, and $P \subsetneq G$ a parabolic subgroup.  Then the symmetry algebra $\cS = \pi_*(\finf(\cG,\omega))$ of the underlying structure of any regular, normal geometry $(\cG \stackrel{\pi}{\to} M, \omega)$ of type $(G,P)$ is 1-jet determined at any non-flat point $x \in M$.
 \end{thm}

 \begin{proof}  Fix $x \in M$, $u \in \pi^{-1}(x)$ with $\kappa_H(u) \neq 0$.  Since $\fs(u) \subset \fa^{\kappa_H(u)}$ (Theorem \ref{T:strong-incl}) and $\fa^{\kappa_H(u)}_\nu = 0$ (Theorem \ref{T:height}), then $\fs_\nu(u) = 0$.  Let $0 \neq \xi \in \finf(\cG,\omega)$ with $\omega_u(\xi) \in \fg^i \backslash \fg^{i+1}$ projecting to $\fs_i(u)$ (with $i < \nu$).  Let $\bX = \pi_*(\xi)$.  By Theorem \ref{T:M-bracket}, $j^0_x(\bX) \neq 0$ if $i < 0$, while $j^1_x(\bX) \neq 0$ if $0 \leq i < \nu$.
 \end{proof}

 \begin{example}\label{Ex:path-geom}
In \cite[Prop.\ 5.3.2]{KT2014}, we gave the following (homogeneous) submaximally symmetric model for systems of 2nd order ODE 
in $m \geq 2$ dependent variables (or $\SL(m+2,\bbR) / P_{1,2}$ geometry):
 \[
 \ddot{x}_1 = ... = \ddot{x}_{m-1} = 0, \qquad \ddot{x}_m = \dot{x}_1^3.
 \]
Given any point symmetry $\bX$ on the space $J^0(\bbR,\bbR^m)=\bbR^{m+1}(t,x_1,\dots,x_m)$, 
we prolong it to a symmetry $\bX^{(1)}$ on the space $J^1(\bbR,\bbR^m)=\bbR^{2m+1}(t,x_1,\dots,x_m,p_1,\dots,p_m)$
(which is identified with $M$). All symmetries $\bX^{(1)}$ are 1-jet determined, including the symmetry 
 $$
\bS=t^2\partial_t+t\sum_{i=1}^mx_i\partial_{x_i}+\tfrac12x_1^3\partial_{x_m}\ \stackrel{\text{prolong}}\dashrightarrow
\bS^{(1)} = \bS + \sum_{i=1}^m (x_i - tp_i)\partial_{p_i} + \tfrac{3}{2} x_1^2 p_1 \partial_{p_m}.
 $$
This latter symmetry has isotropy in $\fp_+$ at $0\in M$, even though the geometry is nowhere flat.
 \end{example}

 %%%%%%%%%%%%%%%%%%%%%%%%%%%%%%%%%%%%%%%%%

 \section{Sub-cominuscule representations}
 \label{S:subcominuscle}

 Let $G$ be a complex {\bf simple} Lie group and $\fg = \fg_{-\nu} \op ... \op \fg_\nu$ the grading associated to a parabolic subgroup $P \subsetneq G$.  In this section, we study the $G_0$-module structure of the top slot $\fg_\nu$.  We will show that $\fg_\nu$ is a {\em sub-cominuscule} representation and give an effective method for constructing explicit $G_0$-orbit representatives.

 \subsection{$\fg_\nu$ is sub-cominuscule}

 A fundamental weight $\lambda_j$ of $\fg$ is {\em cominuscule} if the highest root $\lambda = \sum_i m_i \alpha_i$ of $\fg$ has $m_j = 1$.  Any such weight corresponds to a {\em minuscule} variety $X \cong G/P \subset \bbP(\bbU)$, which is a Hermitian symmetric space (so $\fp = \fp_j$ induces a $|1|$-grading on $\fg$) in its minimal homogeneous embedding, i.e.\ $\bbU$ has highest weight $\lambda_j$ and $X$ is the $G$-orbit of the highest weight line in $\bbU$ (the unique closed orbit).  All {\em irreducible} minuscule varieties are given by
 \begin{align} \label{E:1-graded}
 A_\ell / P_k, \quad B_\ell / P_1, \quad C_\ell / P_\ell, \quad D_\ell / P_1, \quad D_\ell / P_\ell, \quad E_6 / P_6, \quad E_7 / P_7
 \end{align}
 in their minimal homogeneous embeddings, e.g.\ $A_\ell / P_k \cong \tGr(k,\ell+1) \inj \bbP(\bigwedge^k \bbC^{\ell+1})$.

 \begin{defn}
 Let $V$ be an irreducible $R$-module, where $R$ is a complex reductive Lie group, and let $\cV\subset \bbP(V)$ denote the unique closed $R$-orbit.  If the only $R$-orbit closures in $\bbP(V)$ are the secant varieties $\cV =: \Sec_1(\cV) \subsetneq \Sec_2(\cV) \subsetneq ... \subsetneq \Sec_m(\cV) = \bbP(V)$, then $V$ is {\em sub-cominuscule} and $\cV$ is a {\em sub-cominuscule variety}.  (The $R$-orbits are $\cV$ and $\Sec_i(\cV) \backslash \Sec_{i-1}(\cV)$ for $2 \leq i \leq m$.)
 \end{defn}

 Note that it suffices to consider only the effective part of the $R$-action on $\bbP(V)$.  In particular, by Schur's lemma, all central elements in $R$ act as scalars on $V$, so these can be ignored.

From the marked Dynkin diagram $\fD(\fg,\fp)$, a simple recipe (see Appendix \ref{S:DD}) yields the $\fg_0^{ss}$-module structure on $\fg_1$, where $\fg_0^{ss}$ denotes the semisimple part of $\fg_0$.  For (irreducible) minuscule varieties $G/P$, it was observed in \cite[Sec.3.3]{LM2003} (see Remark 3.5) that the only $G_0$-orbit closures in $\bbP(\fg_1)$ are the secant varieties (and that this is a characterizing property of minuscule varieties).  Table \ref{F:SC} gives the list of sub-cominuscule representations.  Not all sub-cominuscule varieties are themselves minuscule (since not all embeddings are minimal), but each has an associated $|1|$-grading.

 \begin{table}[h]
 \[
 \begin{array}{|c|c|c|c|c|} \hline
 G/P & G_0^{ss} & \mbox{Sub-cominuscule variety $\cV \subset \bbP(\fg_1)$} & \mbox{$\# G_0$-orbits in $\bbP(\fg_1)$}\\ \hline\hline
 A_\ell / P_k & A_{k-1} \times A_{\ell - k} &
  \Seg(\bbP^{k-1} \times \bbP^{\ell-k}) \inj \bbP(\bbC^k \boxtimes \bbC^{\ell+1-k}) & \min(k,\ell+1-k)\\ \hline
 \begin{array}{c} B_\ell / P_1 \\ D_\ell / P_1 \end{array} &
 \begin{array}{c} B_{\ell-1} \\ D_{\ell-1} \end{array} &
 \mbox{quadrics}
 \begin{array}{l}
 Q^{2\ell-3} \inj \bbP^{2\ell-2} \\
 Q^{2\ell-4} \inj \bbP^{2\ell-3}
 \end{array} & 2\\ \hline
 C_\ell / P_\ell & A_{\ell-1} & \bbP^{\ell-1} \inj \bbP(S^2 \bbC^\ell) & \ell\\ \hline
 D_\ell / P_\ell & A_{\ell-1} & \tGr(2,\ell) \inj \bbP(\bigwedge^2 \bbC^\ell) & \floor{\frac{\ell}{2}}\\ \hline
 E_6 / P_6 & D_5 & \bbS_5 = D_5 / P_5 \inj \bbP^{15} & 2\\ \hline
 E_7 / P_7 & E_6 & \bbO\bbP^2 = E_6 / P_6 \inj \bbP^{26} & 3\\ \hline
 \end{array}
 \]
 \caption{Complete list of sub-cominuscule varieties / modules}
 \label{F:SC}
 \end{table}

 Consider now the general case.  Let $\fD(\fg,\fp,\lambda)$ be $\fD(\fg,\fp)$ with coefficients $r_i$ of the highest weight $\lambda = \sum_i r_i \lambda_i$ of $\fg$ inscribed over the nodes.  There is a simple Dynkin diagram recipe (see Appendix \ref{S:DD}) to obtain the effective part of the $\fg_0^{ss}$-module structure on the top-slot $\fg_\nu$: {\em Remove crosses from $\fD(\fg,\fp,\lambda)$, then remove any connected components with only zero coefficients.}  This gives a diagram $\fT^0(\fg_0^{ss},\lambda)$.

  \begin{prop} \label{P:SC} The top-slot $\fg_\nu$ is a sub-cominuscule $G_0$-module.
 \end{prop}

 \begin{proof} Starting with the highest weight of $\fg$ (Table \ref{F:hw}), we consider any non-trivial collection of crosses on the Dynkin diagram $\fD(\fg)$.  Applying the recipe above, we easily obtain for $\fT^0(\fg_0^{ss},\lambda)$ only those possibilities in the third column of Table \ref{F:SC}.  For example, for $G = E_8$ and almost all choices of $\fp \subsetneq \fg$, we obtain $\fT^0(\fg_0^{ss},\lambda) \cong \Add{wwww}{1,0,0,0}$.  The only exception is $E_8 / P_1$, for which the (effective) top-slot structure agrees with that for $D_8 / P_1$.  Both are sub-cominuscule.
 \end{proof}

 %%%%%%%%%%%%%%%%%%%%%%%%%%%%%%%%%%%%%%%%%%%%%%%%%%

 \subsection{The top-slot orthogonal cascade}

  If $\lambda = \sum_i r_i \lambda_i$, we refer to $I_\sfc = \{ i \st r_i \neq 0 \}$ as the {\em contact nodes} and denote these by \mydiamond{} on the Dynkin diagram $\fD(\fg)$.  (This is the same as the set of nodes connected to $-\lambda$ in the extended Dynkin diagram of $\fg$.)  Removing all $\mydiamond{}$ yields a diagram $\fD_\lambda(\fg)$ with corresponding semisimple subalgebra $\fg(\lambda) \subset \fg$ and $\Delta(\fg(\lambda)) = \{ \alpha \in \Delta(\fg) \st \langle \lambda, \alpha \rangle = 0 \}$.  Note that $\fD_\lambda(\fg)$ is connected (i.e.\ $\fg(\lambda)$ is simple) if and only if $\fg \neq B_\ell, D_\ell$.  For $\fg = B_\ell$ or $D_\ell$, there are two connected components in $\fD_\lambda(\fg)$.

 \begin{remark} \label{R:contact}
 If $G \neq A_1$, then $\fg$ admits a contact gradation $\fg = \fc_{-2} \op \fc_{-1} \op \fc_0 \op \fc_1 \op \fc_2$ with respect to some grading element $\sfZ_\sfc$.  Here, $\fg(\lambda) = \fc_0^{ss}$, $\fc_-$ isomorphic to a Heisenberg Lie algebra, and $\fc^0 = \fc_{\geq 0}$ is the parabolic subalgebra of $\fg$, obtained from $\fD(\fg)$ by replacing all \mydiamond{} by crosses.
 \end{remark}

 \begin{prop} \label{P:hw-perp}
 Let $\beta \in \Delta(\fg(\lambda))$ be maximal.  Expressed in terms of the weights of $\fg$, we have the decomposition $\beta = \widetilde\lambda - \lambda$, where $\widetilde\lambda = \sum_{i \not\in I_\sfc} r_i \lambda_i$ is a dominant integral weight.
 \end{prop}

 \begin{proof} For $A_\ell$ ($\ell \geq 3$), $\lambda = \alpha_1 + ... + \alpha_\ell = \lambda_1 + \lambda_\ell$ and $\beta = \lambda - \alpha_1 - \alpha_\ell = -\lambda_1 + \lambda_2 + \lambda_{\ell-1} - \lambda_\ell = \widetilde\lambda - \lambda$.  For $C_\ell$ ($\ell \geq 2$), $\lambda = 2\alpha_1 + ... + 2\alpha_{\ell-1} + \alpha_\ell = 2\lambda_1$ and $\beta = \lambda - 2\alpha_1 = -2\lambda_1 + 2\lambda_2 = \widetilde\lambda - \lambda$.

 Suppose $\fg$ is not type $A$ or $C$.  Let $j$ be the unique contact node, so $\lambda = \lambda_j$. We must have $\beta = \widetilde\lambda + b\lambda_j$, where $b \in \bbZ$.  In particular, $\langle \beta,\alpha_j^\vee \rangle = \langle \widetilde\lambda + b\lambda,\alpha_j^\vee \rangle = b \,\langle \lambda,\alpha_j^\vee \rangle = b$.

 Since $\langle \beta, \lambda_j \rangle = 0$, then $\sfZ_j(\beta) = 0$, so $\beta - \alpha_j \not\in \Delta$ (since $\beta \in \Delta^+$). While $\beta$ is maximal in $\Delta(\fg(\lambda))$, it is not maximal in $\Delta(\fg)$, so $\beta + \alpha_j \in \Delta$. The $\alpha_j$-string through $\beta$ is  $\beta, \beta + \alpha_j, ..., \beta - b\alpha_j$, so $b=-1$ or $-2$.  If $b = -2$, then $\lambda = \beta + 2\alpha_j$ (since $\lambda$ is the unique root with $\sfZ_j$-grading $+2$).  Thus, $\langle \lambda, \alpha_j^\vee \rangle = \langle \beta + 2\alpha_j, \alpha_j^\vee \rangle = b+4 = 2$, so $\fg = C_\ell$, a contradiction.  Hence, $b=-1$.
 \end{proof}

  Given $\alpha \in \Delta$, define $h_\alpha \in \fh$ by $\alpha = B(h_\alpha, \cdot)$, where $B$ is the Killing form.  For $\fg(\lambda)$, take:
 \begin{itemize}
 \item $\widetilde\fh = \fh \cap \fg(\lambda) = \tspan\{ h_{\alpha_i} \st \langle \lambda, \alpha_i \rangle = 0 \} = \tspan\{ h_{\alpha_i} \st i \not\in I_\sfc \}$ to be the Cartan subalgebra.
 \item $\widetilde\alpha_i = \alpha_i|_{\widetilde\fh}$, $i \not\in I_\sfc$ to be the simple roots.
 \item $\widetilde\lambda_i = \lambda_i|_{\widetilde\fh}$, $i \not\in I_\sfc$ to be the fundamental weights.
 \end{itemize}
 Using Proposition \ref{P:hw-perp} and the known highest weights for simple Lie algebras (Table \ref{F:hw}), we can immediately write any maximal $\beta \in \Delta(\fg(\lambda))$ in terms of the fundamental weights of $\fg$.  The following is similar to Kostant's cascade of orthogonal roots \cite{Kos2012}, but is restricted to the top part of the grading.

 \begin{defn} \label{D:TSOC} The {\em top-slot orthogonal cascade} (TSOC) is an ordered sequence $\{ \beta_1, \beta_2,... \} \subset \Delta(\fg_\nu)$, where $\beta_1 = \lambda$ is the highest root of $\fg$, and
 \begin{align} \label{E:max-root}
 \beta_j = \max\{ \alpha \in \Delta(\fg_\nu) \st \alpha \in \{ \beta_1,...,\beta_{j-1} \}^\perp \}, \quad j \geq 2.
 \end{align}
 \end{defn}

 \begin{example}[$C_\ell / P_{\ell-1}$] \label{E:CP1} Here $\nu=2$.  Using Proposition \ref{P:hw-perp} and Table \ref{F:hw}, we obtain the TSOC:
 \begin{align} \label{E:C-ex}
  \beta_1 = 2\lambda_1, \quad \beta_j = -2\lambda_{j-1} + 2\lambda_j \quad (2 \leq j \leq \ell-1).
 \end{align}
 In root notation, $\beta_i = \underbrace{0\cdots0}_{i-1}\underbrace{2\cdots2}_{\ell-i}1$, $1 \leq i \leq \ell-1$.
 Note that $\beta_\ell = -2\lambda_{\ell-1} + 2\lambda_\ell = 0\cdots01 \not\in \Delta(\fg_\nu)$.
 \end{example}

 By Lemma \ref{L:max}(ii) below, there is a unique maximal root in \eqref{E:max-root} at each step.  Now our main result in this section is:

 \begin{thm} \label{T:TSOC} Let $G$ be a complex simple Lie group, $P \subsetneq G$ a parabolic subgroup, and $\fg = \fg_{-\nu} \op ... \op \fg_\nu$ the associated grading.  Representative elements of the $G_0$-orbits in $\bbP(\fg_\nu)$ are given in terms of the TSOC $\{ \beta_1,\beta_2,\beta_3,... \} \subset \Delta(\fg_\nu)$ as the projectivization of the elements
 \begin{align} \label{E:G0-orbit-reps}
  e_{\beta_1}, \quad e_{\beta_1} + e_{\beta_2}, \quad e_{\beta_1} + e_{\beta_2} + e_{\beta_3}, \quad ...,
 \end{align}
 for any fixed choice of corresponding root vectors.
 \end{thm}

 The proof will follow from two lemmas.  Given a parabolic subalgebra $\fp \subsetneq \fg$, apply the following recipe: {\em Remove all contact nodes of $\fg$ from $\fD(\fg,\fp)$, then remove all cross-free connected components.}  This gives a diagram $\fD_\lambda^0(\fg,\fp)$ corresponding to an ideal $\fl_\fp(\lambda) \subset \fg(\lambda)$.  (All other ideals of $\fg(\lambda)$ lie in $\fg_0$.  We will be interested in those that intersect $\fg_\nu$ non-trivially.)  Let $\fp(\lambda) = \fl_\fp(\lambda) \cap \fp$ and denote by $L_\fp(\lambda)$, $P(\lambda)$ the connected subgroups of $G$ corresponding to $\fl_\fp(\lambda)$, $\fp(\lambda)$.

 \begin{lemma} \label{L:max}
 $\Delta(\fg(\lambda)) \cap \Delta(\fg_\nu) \neq \emptyset$ if and only if $\emptyset \neq \fT^0(\fg_0^{ss},\lambda) \not\cong \Add{wwww}{1,0,0,0}$.  In this case,
 \begin{enumerate}
 \item[(i)] $I_\fp$ does not contain any contact node(s) of $\fg$.
 \item[(ii)] $\fD^0_\lambda(\fg,\fp) \neq \emptyset$ and contains a single connected component.  (Hence, $\fl_\fp(\lambda)$ is simple and $\max\Delta(\fg(\lambda)) = \max\Delta(\fl_\fp(\lambda)) \in \Delta(\fg_\nu)$ is well-defined and unique.)
 \end{enumerate}
 \end{lemma}

 \begin{proof} Note that $\fT^0(\fg_0^{ss},\lambda) = \emptyset$ if and only if $I_\fp$ contains {\bf all} contact nodes of $\fg$.  In this case, $\fg_\nu$ is 1-dimensional and so $\Delta(\fg(\lambda)) \cap \Delta(\fg_\nu) = \emptyset$.  So let us suppose $\fT^0(\fg_0^{ss},\lambda) \neq \emptyset$ below.
 
 For $\beta \in \Delta(\fg_\nu)$, we have $\lambda + \beta \not\in \Delta$ always, so the orthogonality condition $\langle \lambda, \beta\rangle = 0$ is equivalent to $\lambda - \beta \not\in \Delta(\fg_0)$, i.e.\ the $\beta$-string through $\lambda$ is trivial.  This statement depends only on (the effective part of) the $\fg_0^{ss}$-module structure on $\fg_\nu$, which is sub-cominuscule.  Similarly, $\fT^0(\fg_0^{ss},\lambda)$ only depends on the same restricted structure.  Thus, in proving the first statement, we may without loss of generality suppose that $G/P$ is irreducible minuscule, so $\nu=1$.
 
  Suppose $\fT^0(\fg_0^{ss},\lambda) \cong \Add{wwww}{1,0,0,0}$.  Given that $\nu=1$, we must have $G/P \cong A_\ell / P_1$ and $\lambda = \alpha_1 + ... + \alpha_\ell$.   Any $\beta \in \Delta(\fg_1)$ is of the form $\beta = \alpha_1 + ... + \alpha_k$ for some $1 \leq k \leq \ell$, so $\lambda - \beta \in \Delta$ for $\lambda \neq \beta$, i.e.\ $\langle \lambda, \beta \rangle \neq 0$.  Thus, $\Delta(\fg(\lambda)) \cap \Delta(\fg_1) = \emptyset$.  Now suppose $\emptyset \neq\fT^0(\fg_0^{ss},\lambda) \not\cong \Add{wwww}{1,0,0,0}$.  Table \ref{F:orth-minuscule} gives $\lambda$ and $\beta \in \Delta(\fg_1)$ such that $\lambda - \beta \not\in\Delta$ and hence $\langle \lambda,\beta \rangle = 0$.  Thus, $\Delta(\fg(\lambda)) \cap \Delta(\fg_1) \neq\emptyset$ and hence the first statement is proven.

For the proof of (i) and (ii), consider again the case of general $\nu$.  Suppose (i) is false.  When $G \neq A_\ell$, the contact node is unique, so $\fT^0(\fg_0^{ss},\lambda) = \emptyset$.  When $G = A_\ell$, $\fT^0(\fg_0^{ss},\lambda) = \Add{wwww}{1,0,0,0}$ is forced.  For (ii), take $\beta \in \Delta(\fg(\lambda)) \cap \Delta(\fg_\nu)$.  Then $\beta$ must be an integer linear combination of simple roots from a {\em single} connected component of $\fD_\lambda(\fg,\fp)$.  If this component is cross-free, then $\beta \in \Delta(\fg_0)$, a contradiction, so $\fD^0_\lambda(\fg,\fp) \neq \emptyset$.  If $\fD^0_\lambda(\fg,\fp)$ has two connected components (each with crosses), then $\sfZ(\beta) < \sfZ(\lambda)$.  This proves (ii).
 \end{proof}

  \begin{table}[h]
  \[
 \begin{array}{|c|c|c|c|c|} \hline
 G/P & \mbox{Range} & \begin{tabular}{c} $\lambda$ in root\\ notation \end{tabular} & \beta = \max\Delta(\fl_\fp(\lambda)) & L_\fp(\lambda) / P(\lambda)\\ \hline\hline
 A_\ell / P_k & 1 < k < \ell & 11\cdots11 & 011\cdots110 & A_{\ell-2} / P_{k-1}\\
 B_\ell / P_1 & \ell \geq 2 & 122\cdots222 & 100\cdots000 & A_1 / P_1 \\
 D_\ell / P_1 & \ell \geq 4 & 122\cdots211 & 100\cdots000 & A_1 / P_1\\
 C_\ell / P_\ell & \ell \geq 3 & 222\cdots221 & 022\cdots221 & C_{\ell-1} / P_{\ell-1}\\
 D_\ell / P_\ell & \ell \geq 4 & 1\overbrace{22\cdots22}^{\ell-3}11 & \left\{ \begin{array}{ll} 001\overbrace{22\cdots22}^{\ell-5}11, & \ell \geq 5;\\ 0001, & \ell = 4
 \end{array} \right. & D_{\ell-2} / P_{\ell-2}\\
 E_6 / P_6 & - & 122321 & 101111 & A_5 / P_1\\
 E_7 / P_7 & - & 2234321 & 0112221 & D_6 / P_1\\ \hline
 \end{array}
 \]
 \caption{Maximal root orthogonal to the highest root: $|1|$-graded case}
  \label{F:orth-minuscule}
 \end{table}

 \begin{lemma} \label{L:minuscule}
 We have $\Delta(\fg(\lambda)) \cap \Delta(\fg_\nu) \neq \emptyset$ if and only if $\bbP(\fg_\nu)$ contains at least two $G_0$-orbits.

 In this case, if $\beta \in \Delta(\fg(\lambda))$ is the highest root, then
 \begin{enumerate}
 \item[(i)] $\beta$ lies in the $W_\fp$-orbit of $\lambda$, where $W_\fp$ is the Weyl group of $G_0$.  Thus, $[e_\beta] \in \cV := G_0  [e_\lambda]$.
 \item[(ii)] $[e_\lambda]$ and $[e_\lambda + e_\beta]$ lie in different $G_0$-orbits, namely $\cV$ and $\Sec_2(\cV)\backslash \cV$ respectively, and these have different dimensions.
 \end{enumerate}
 \end{lemma}

 \begin{proof} As in Lemma \ref{L:max}, it suffices to verify the claims in the irreducible minuscule case.

 By Lemma \ref{L:max}, if $\Delta(\fg(\lambda)) \cap \Delta(\fg_\nu) = \emptyset$, then $\fT^0(\fg_0^{ss},\lambda) = \emptyset$ or $\fT^0(\fg_0^{ss},\lambda) \cong \Add{wwww}{1,0,0,0}$.  In both cases, there is a single $G_0$-orbit in $\bbP(\fg_1)$.  Now suppose $\Delta(\fg(\lambda)) \cap \Delta(\fg_1) \neq \emptyset$.  We will prove (i) and (ii), from which the first claim follows.  For $B_\ell / P_1$ and $D_\ell / P_1$, (i) is immediate since $\alpha_1 = \min \Delta(\fg_1)$ is $\fg_0$-extremal.  For other cases, we give $w \in W_\fp$ such that $\beta = w(\lambda)$:
 \[
 A_\ell / P_k : \,\, (1\ell); \quad C_\ell / P_\ell : \,\, (1); \quad
 D_\ell / P_\ell : \,\, (2132); \quad E_6 / P_6:\,\, (243542); \quad E_7 / P_7: \,\,(13425431).
 \]
 Now consider (ii).   We have $\cV = G_0[e_\lambda] \cong G_0^{ss} / Q \subset \bbP(\fg_1)$, where $Q$ is a parabolic subgroup of $G_0^{ss}$.  Letting $\fl = \fg_0^{ss}$, $Q$ induces a $|1|$-grading $\fl = \fl_{-1} \op \fl_0 \op \fl_1$ with $\fq = \fl_{\geq 0}$.  Since $\fl_{-1}$ is abelian, then
 \[
 \exp(x) \cdot e_\lambda = (\id + \ad_x)(e_\lambda) = e_\lambda + [x,e_\lambda], \qquad \forall x \in \fl_{-1}.
 \]
 Assuming $[e_\lambda + e_\beta] \in G_0[e_\lambda]$, we must have $[x,e_\lambda] \in \fg_\beta$ for some $x \in \fl_{-1}$.  Hence, $\exists \gamma \in \Delta(\fl_{-1})$ such that $\lambda + \gamma = \beta$, so $\lambda - \beta \in \Delta(\fl_1)$.  Clearly $\lambda + \beta \not\in \Delta$, and $\langle \lambda, \beta \rangle = 0$ by assumption, so the $\beta$-string through $\lambda$ is trivial.  Hence, $\lambda - \beta \not\in\Delta$, a contradiction.  This proves (ii).
 \end{proof}

 Given $(\fg,\fp,\lambda)$, we apply Lemma \ref{L:max} to construct the TSOC iteratively.
 The parabolic in $G$ stabilizing $[e_\lambda]$ yields a contact grading (see Remark \ref{R:contact}).  Its intersection with $G_0^{ss}$ is a parabolic subgroup that yields a secondary grading,
 \[
 \fg_\nu = \fg_{\nu,0} \op \fg_{\nu,1} \op \fg_{\nu,2}, \quad \lambda \in \Delta(\fg_{\nu,2}), \quad
 \fg_0 = \fg_{0,-1} \op \fg_{0,0} \op \fg_{0,1}.
 \]
 (Note that $\fg_{0,2} = 0$ since $\lambda$ is the unique root with contact grading +2.)  Let $\beta = \max \Delta(\fl_\fp(\lambda)) \subset \Delta(\fg_{\nu,0})$.  We now apply Lemma \ref{L:max} to $(\fl_\fp(\lambda),\fp(\lambda),\beta)$.  (Note that $\fl_\fp(\lambda) = \fg_{*,0}^{ss}$ and $\fp(\lambda) = (\fg_{*,0}^{ss})_{\geq 0}$.)  By Lemma \ref{L:minuscule}(ii), $[e_\lambda + e_\beta]$ lies in $\Sec_2(\cV) \backslash \cV$.

 Similarly, the parabolic in $L_\fp(\lambda)$ yields stabilizing $[e_\beta]$ yields a contact grading.  Its intersection with $G_{0,0}^{ss}$ is a parabolic subgroup that yields a tertiary grading,
 \[
 \fg_{\nu,0} = \fg_{\nu,0,0} \op \fg_{\nu,0,1} \op \fg_{\nu,0,2}, \quad
 \beta \in \Delta(\fg_{\nu,0,2}), \quad
 \fg_{0,0} = \fg_{0,0,-1} \op \fg_{0,0,0} \op \fg_{0,0,1}.
 \]

 The third element of the TSOC will be $\gamma \in \Delta(\fg_{\nu,0,0})$ and
 by Lemma \ref{L:minuscule}(i), $[e_\gamma] \in \cV$.  Note $[\fg_{0,-1},e_\lambda] \subset \fg_{\nu,1}$, $[\fg_{0,-1},e_\beta] = 0$, and $[\fg_{0,0,-1},e_\beta] \subset \fg_{\nu,0,1}$.  Similar to the proof of Lemma \ref{L:minuscule}(ii),
 \[
 e_\lambda + e_\beta + e_\gamma \not\in G_0 \cdot (e_\lambda + e_\beta) + G_0\cdot e_\lambda.
 \]
 Thus, $[e_\lambda + e_\beta + e_\gamma]$ lies in $\Sec_3(\cV) \backslash \Sec_2(\cV)$.  Continuing in a similar fashion, the sequence \eqref{E:G0-orbit-reps} yields representative elements in successive secant varieties of $\cV$ (and by definition of sub-cominuscule, these are all the orbits).  This finishes the proof of Theorem \ref{T:TSOC}.

 Let us summarize: Provided that $\emptyset \neq \fT^0(\fg_0^{ss},\lambda) \not\cong \Add{wwww}{1,0,0,0}$ holds, we proceed to construct the TSOC iteratively (by removing $\mydiamond{}$ from $\fD(\fg)$ and using Proposition \ref{P:hw-perp}).

 \begin{example}[$E_7 / P_7$] \label{E:E7P7}
 \[
 \begin{array}{|c|ccccc|} \hline
 \raisebox{-0.05in}{\begin{tabular}{c}
 Dynkin diagram\\
 sequence
 \end{tabular}} & \Edd{dwwwwwx}{} & \leadsto&  \Dsix{xdwwww}{} & \leadsto&  \Aone{x}{}\\ \hline
 \fT^0(\fg_0^{ss},\lambda) &\Edd{wwwwww}{1,\!\!\!\!\!0,0,0,0,0} & & \Dfive{wwwww}{1,0,0,0,0} && \emptyset\\ \hline
% \mbox{Highest roots} & 2234321 && 122211 && 1\\ \hline
% \begin{tabular}{c} TSOC in\\ root notation
% \end{tabular} & \beta_1 = 2234321 && \beta_2 = 0112221 && \beta_3 = 0000001 \\ \hline
 \begin{tabular}{c} TSOC in\\ weight notation \end{tabular} &
  \beta_1 = \lambda_1 && \beta_2 = -\lambda_1 + \lambda_6 && \beta_3 = -\lambda_6 + 2\lambda_7 \\ \hline
 \end{array}
 \]
 \end{example}

 %%%%%%%%%%%%%%%%%%%%%%%%%%%%%%%%%%%%%%%%%%%%%%%%%%

 \subsection{The $H$-sequence for the TSOC}
 \label{S:H-seq}

 By the Jacobson--Morozov theorem, any nilpotent element $E$ in $\fg$ can be included into a standard $\fsl_2$-triple $\{ E, H, F \}$ with commutation relations
 \begin{align} \label{E:sl2}
 [H,E] = 2E, \quad [H,F] = -2F, \quad [E,F] = H.
 \end{align}
 For any $\alpha \in \Delta$, define $H_\alpha \in \fh$ by $\alpha(H) = B(H,H_\alpha)$, where $B$ is the Killing form of $\fg$.  We have the standard relation $\langle \alpha, \beta \rangle = B(H_\alpha,H_\beta) = \alpha(H_\beta) = \beta(H_\alpha)$.  Letting $\alpha \in \Delta^+$ and defining $h_\alpha = \frac{2}{\langle \alpha, \alpha \rangle} H_\alpha$, we can choose root vectors $e_{-\alpha} \in \fg_{-\alpha}$ and $e_\alpha \in \fg_\alpha$ such that $\{ e_{\alpha}, h_\alpha, e_{-\alpha} \}$ is a standard $\fsl_2$-triple \cite[p.143]{Kna2002}.  We have $\alpha(h_\beta) = \langle \alpha, \beta^\vee \rangle$.
 
   Now let $\fp \subsetneq \fg$ be a parabolic subalgebra.
  We will identify the semisimple element $H$ for each element $E$ of the sequence \eqref{E:G0-orbit-reps} associated to the TSOC.

 \begin{lemma} \label{L:TSOC-sl2}
 Suppose that $\{ \beta_1,..., \beta_m \} \subset \Delta(\fg_\nu)$ is the TSOC.  For any $1 \leq j \leq m$, we have the standard $\fsl_2$-triple $\{ \sfE_j, \sfH_j, \sfF_j \}$ given by
 \begin{align*}
 \sfE_j &= e_{\beta_1} + ... + e_{\beta_j}\\
 \sfH_j &= h_{\beta_1} + ... + h_{\beta_j}\\
 \sfF_j &= e_{-\beta_1} + ... + e_{-\beta_j}
 \end{align*}
 \end{lemma}

 \begin{proof}
 For $1 \leq a \neq b \leq j$, $\beta_a + \beta_b \not\in \Delta$ since both are in the top-slot.  Since $\langle \beta_a, \beta_b \rangle = 0$, then $\beta_a - \beta_b \not\in \Delta$, so $[\sfE_j,\sfF_j] = \sfH_j$.
 Also, $[h_{\beta_a},e_{\beta_b}] =\beta_b(h_{\beta_a}) e_{\beta_b} = \langle \beta_b, \beta_a^\vee \rangle e_{\beta_b} = 0$, so $[\sfH_j,\sfE_j] = 2\sfE_j$ and similarly, $[\sfH_j,\sfF_j] = -2\sfF_j$.
 \end{proof}

 \begin{defn}
 We refer to $\{ \sfH_j \}_{j=1}^m$ as the $H$-sequence for the TSOC $\{ \beta_1, ..., \beta_m \} \subset \Delta(\fg_\nu)$.
 \end{defn}

 Let $\{ \sfZ_i \} \subset \fh$ be the dual basis to the simple roots $\{ \alpha_i \} \subset \fh^*$.
 Given $\alpha = \sum_i r_i \lambda_i$, we have $h_\alpha = \sum_i t_i \sfZ_i$, so $t_i = \alpha_i(h_\alpha) = \langle \alpha_i, \alpha^\vee \rangle = \sum_j r_j \langle \lambda_j, \alpha_i \rangle \frac{2}{\langle \alpha, \alpha \rangle} = \sum_j r_j \langle \lambda_j, \alpha_i^\vee \rangle \frac{\langle \alpha_i, \alpha_i \rangle}{\langle \alpha, \alpha \rangle}$.  Thus,
 \begin{align} \label{E:alpha-Z}
 \alpha = \sum_i r_i \lambda_i \qRa h_\alpha = \sum_i r_i \frac{\langle \alpha_i, \alpha_i \rangle}{\langle \alpha, \alpha \rangle} \sfZ_i.
 \end{align}

By Lemma \ref{L:minuscule}, the TSOC is contained in a single $W_\fp$-orbit, so $\langle \beta_j, \beta_j \rangle = \langle \lambda, \lambda \rangle$ for $1 \leq j \leq m$, where $\lambda$ is the highest weight of $\fg$.
 
  \begin{framed} Normalize $\langle \cdot, \cdot \rangle$ so that all {\em short} simple roots $\alpha_i$ satisfy $\langle \alpha_i, \alpha_i \rangle = 2$.  Hence, from the extended Dynkin diagram, we see that $\langle \lambda, \lambda \rangle$ is 2 in type ADE, 4 in type BCF, and 6 in type G.
  
 \end{framed}
 
It is straightforward to verify that $\sfH_1 = \sum_{i \in I_\sfc} \sfZ_i$, where $\lambda = \sum_i r_i \lambda_i$ and $I_\sfc = \{ i \st r_i \neq 0 \}$.

 \begin{example}[$C_\ell / P_{\ell-1}$] \label{E:CP2} Recall from \eqref{E:C-ex} that $\beta_1 = 2\lambda_1$, $ \beta_j = -2\lambda_{j-1} + 2\lambda_j \quad (2 \leq j \leq \ell-1)$.
 Since $\alpha_1,...,\alpha_{\ell-1}$ are short, then \eqref{E:alpha-Z} implies $\sfH_j = \sum_{i=1}^j h_{\beta_i} = \sfZ_j$ for $1 \leq j \leq \ell-1$.
 \end{example}

 \begin{example}[$B_\ell / P_\ell$] $ \beta_1 = \lambda_2, \quad
 \beta_j = -\lambda_{2j-2} + \lambda_{2j}\,\, (1 < j < \floor{\frac{\ell}{2}}), \quad
 \beta_{\floor{\frac{\ell}{2}}} = \left\{ \begin{array}{ll} \lambda_{\ell-1}, & \ell \mbox{ odd};\\ 2\lambda_\ell, & \ell \mbox{ even} \end{array} \right.
$ gives the TSOC.  Since $\alpha_\ell$ is the only short root, the $H$-sequence is $ \sfH_j = \sfZ_{2j}$ for $1 \leq j \leq \floor{\frac{\ell}{2}}$.
 \end{example}

 In fact, the construction of the $H$-sequence is even simpler:

 \begin{cor} \label{C:Hj-Z}
 Let $\{ \sfH_j \}_{j=1}^m$ be the $H$-sequence for the TSOC $\{ \beta_1, ..., \beta_m \} \subset \Delta(\fg_\nu)$.  Then $\sfH_j$ is obtained from the $j$-th step of the Dynkin diagram sequence for the TSOC by forming its highest weight, regarding it as a weight of $\fg$ (see discussion preceding Definition \ref{D:TSOC}), and using \eqref{E:alpha-Z}.  In particular, $\sfH_j = \sum_i t_{ji} \sfZ_i$ has $t_{ji} \in \bbZ_{\geq 0}$.
 \end{cor}

 \begin{proof} By Proposition \ref{P:hw-perp}, a cascade of cancellations occurs when forming the sum $\beta_1 + ... + \beta_j$ and this sum corresponds to the highest weight at the $j$-th step of the Dynkin diagram sequence.
 \end{proof}

 %%%%%%%%%%%%%%%%%%%%%%%%%%%%%%%%%%%%%%%%%%%%%%%%%%

 \subsection{Top-slot orbits for Yamaguchi-nonrigid geometries}

 A regular, normal parabolic geometry of type $(G,P)$ is {\em Yamaguchi-rigid} if $H^2_+(\fg_-,\fg) = 0$.  In this case, $\kappa_H = 0$ and the geometry is flat.  Yamaguchi \cite{Yam1993, Yam1999} identified all non-rigid geometries (which we refer to as Yamaguchi-nonrigid) when $G$ is complex simple.  The further requirement that the top-slot $\bbP(\fg_\nu)$ contain more than one $G_0$-orbit is restrictive:

 \begin{prop} \label{P:NYR}
 Let $G$ be a complex simple Lie group, $P \subsetneq G$ a parabolic subgroup, and $\fg = \fg_{-\nu} \op ... \op \fg_\nu$ the associated grading.  Suppose that:
 \begin{enumerate}
 \item regular, normal parabolic geometries of type $(G,P)$ are Yamaguchi-nonrigid, and
 \item there are at least two $G_0$-orbits in the top-slot $\bbP(\fg_\nu)$.
 \end{enumerate}
 Then $G/P \not\cong A_\ell / P_1$ is either $|1|$-graded or in the following list:
 \[
 \begin{array}{|c|c|c|c|}\hline
 G/P & \mbox{Range} & \# \mbox{orbits} & \mbox{$H$-sequence for TSOC}\\ \hline\hline
 A_\ell / P_{2,s} & 3\leq s \leq \ell-1 & 2 & \sfH_1 = \sfZ_1 + \sfZ_\ell, \quad \sfH_2 = \sfZ_2 + \sfZ_{\ell-1}\\
 A_\ell / P_{s,s+1} & 3 \leq s \leq \floor{\frac{\ell}{2}} & s & \sfH_j = \sfZ_j + \sfZ_{\ell+1-j} \quad (1 \leq j \leq s)\\
 B_\ell / P_\ell & \ell \geq 4 & \floor{\frac{\ell}{2}} & \sfH_j = \sfZ_{2j} \quad (1 \leq j \leq \floor{\frac{\ell}{2}})\\
 C_\ell / P_2 & \ell \geq 3 & 2 & \sfH_1 = \sfZ_1, \,\, \sfH_2 = \sfZ_2\\
 C_\ell / P_{2,\ell} & \ell \geq 3 & 2 & \sfH_1 = \sfZ_1, \,\, \sfH_2 = \sfZ_2\\
 C_\ell / P_{\ell-1} & \ell \geq 3 & \ell-1 & \sfH_j = \sfZ_j \quad (1 \leq j \leq \ell-1)\\
 C_\ell / P_{\ell-1,\ell} & \ell \geq 3 & \ell-1 & \sfH_j = \sfZ_j \quad (1 \leq j \leq \ell-1)\\ \hline
 \end{array}
 \]
 \end{prop}

 \begin{proof}
 Aside from parabolic contact geometries (whose top slot is trivial) and $|1|$-graded geometries, Yamaguchi's non-rigid list contains only those in \cite[Table 8]{KT2014}.  Of the latter, we can automatically exclude any $(G,P)$ where $I_\fp$ contains a contact node of $G$.  (The top-slot contains just a single orbit.)  This yields the above list together with $(B_\ell,P_3), (B_3, P_{1,3}), (D_\ell,P_3), (D_\ell,P_{1,\ell}), (G_2,P_1)$, which are all extraneous by Lemma \ref{L:max} since $\fT^0(\fg_0^{ss},\lambda) \cong \Add{wwww}{1,0,0,0}$.
 \end{proof}

 %%%%%%%%%%%%%%%%%%%%%%%%%%%%%%%%%%%%%%%%%%%%%%%%%%

 \section{Application: Rigidity results for parabolic geometries}
 \label{S:rigidity}

 \subsection{Higher-order vanishing of symmetries}
 \label{Ss:hovs}

 Given any regular, normal parabolic geometry $(\cG \stackrel{\pi}{\to} M, \omega)$ of type $(G,P)$, let $0 \neq \xi \in \finf(\cG,\omega)$.  Fix $x \in M$ and any $u \in \pi^{-1}(x)$.  If $\bX = \pi_*(\xi)$ vanishes at $x$, then $0 \neq E := \omega_u(\xi) \in \fp$.  If moreover $E \in \fp_+$, then the symmetry $\bX$ vanishes to higher order.  The existence of such symmetries is often a strong restriction on the given geometry.

 By Jacobson--Morozov, we can complete $E$ to a standard $\fsl_2$-triple $\{ E, H, F \}$ satisfying \eqref{E:sl2}. Consider the following criteria due to \v{C}ap and Melnick \cite{CM2013a}:
 \begin{itemize}
\item[(CM.1)] the element $H$ belongs to $\fg_0$;
\item[(CM.2)] $H$ has non-positive spectrum as an endomorphism of $\fg_{-}$,
the generalized eigenspace for eigenvalues with real part zero is $C_\fg^{-}(E)=\{X\in\fg_{-} \st [X,E]=0\}$, and $\ad_H|_{C_\fg^{-}(E)}=0$;
\item[(CM.3)] $H$ acts semisimply on $H^2_+(\fg_-,\fg)$ and
all its eigenvalues are non-negative.
 \end{itemize}

 We refer to \cite[Cor.\ 2.14]{CM2013a} (see also \cite[Prop.2.2]{CM2013b}) for the following:

 \begin{thm} \label{T:CM-curve}
Suppose that a symmetry $\bX \neq 0$ vanishes at $x\in M$ to higher order and conditions (CM.1-3) hold.
With notations as above, denote $v=\pi_*(\omega_u^{-1}(F))\in T_xM$. Suppose that $\kappa_H(x) = 0$. Then there exists a distinguished curve $\gamma:(-\epsilon,+\epsilon)\to M$, $\gamma(0)=x$, $\gamma'(0)=v$, which is preserved by the flow of ${\rm\bf X}$ and on which it acts by projective transformations. Denoting $\gamma_+=\gamma((0,+\epsilon))\subset M$ there exists a neighborhood $U$ of $\gamma_+$, $\bar{U}\ni x$, on which the geometry is flat.
 \end{thm}
 
 The curve $\gamma$ is the projection of the flow of the vector field $\omega^{-1}(F)$ through the point $u$.
 
 By Theorem \ref{Th-A}, the hypothesis $\kappa_H(x)=0$ is fulfilled if $G$ is simple, $\bX \neq 0$ and $j^1_x(\bX) = 0$.

 %%%%%%%%%%%%%%%%%%%%%%%%%%%%%%%%%%%%%%%%%%%%%%%%%%

 \subsection{2-jet determined symmetries}
 \label{Ss:62}

 Fix $x \in M$ and any $u \in \pi^{-1}(x)$.  Let $\xi \in \finf(\cG,\omega)$ be such that $0 \neq E := \omega_u(\xi) \in \fg_\nu$ (top-slot).  Since the choice of $u \in \pi^{-1}(x)$ is irrelevant, it suffices to consider $G_0$-orbit representatives $0 \neq E \in \fg_\nu$.  In terms of the TSOC $\{ \beta_1,..., \beta_m \} \subset \Delta(\fg_\nu)$, these representatives are given by $\sfE_1,...,\sfE_m$ (see Theorem \ref{T:TSOC} and Lemma \ref{L:TSOC-sl2}).
 
 Take the corresponding $H$-sequence $\sfH_j = h_{\beta_1} + ... + h_{\beta_j}$, $1 \leq j \leq m$.  Each lies in $\fh \subset \fg_0$, so (CM.1) is satisfied.  If $\alpha \in \Delta(\fg_+)$, then $[\sfH_j,e_{-\alpha}] = \sum_{i=1}^j -\alpha(h_{\beta_i}) e_{-\alpha} = -\left( \sum_{i=1}^j \langle \alpha, \beta_i^\vee \rangle \right)e_{-\alpha}$.  But $\beta_i + \alpha \not\in \Delta$ (since $\beta_i$ are in the top-slot) and the $\alpha$-string through $\beta_i$, namely $\beta_i + n\alpha$, $-p \leq n \leq 0$, satisfies $p = \langle \beta_i, \alpha^\vee \rangle \geq 0$ (hence, $\langle \alpha, \beta_i^\vee \rangle \geq 0$).  Thus, $\sfH_j$ has non-positive spectrum on $\fg_-$.  Since $\langle \alpha, \beta_i \rangle \geq 0$, then $e_{-\alpha}$ lies in the zero eigenspace for $\sfH_j$ if and only if $\alpha \in \{ \beta_1,..., \beta_j \}^\perp$.  This gives precisely $C^-_\fg(\sfH_j)$ and $\ad(\sfH_j)$ restricted here is zero.  Thus, (CM.2) are satisfied for $\sfH_j$.

 Since the $H$-sequence is contained in $\fh$, then $\sfH_j$ acts semisimply on $H^2_+(\fg_-,\fg)$.  All $\fg_0$-weights of $H^2_+(\fg_-,\fg)$ are obtained from the lowest weight $\mu$ by adding roots in $\Delta^+(\fg_0)$. By Corollary \ref{C:Hj-Z}, $\sfH_j$ has non-negative integral coefficients when expressed as a linear combination of the $\sfZ_i$, so $\alpha(\sfH_j) \geq 0$ for any $\alpha \in \Delta^+(\fg_0)$.  Thus, (CM.3) reduces to verifying that:
 \begin{framed}
 (CM.3'): \qquad For any lowest $\fg_0$-weight $\mu$ of $H^2_+(\fg_-,\fg)$, we have $\mu(\sfH_j) \geq 0$.
 \end{framed}

 \begin{cor}
  Given any regular, normal parabolic geometry $(\cG \stackrel{\pi}{\to} M, \omega)$ of type $(G,P)$, let $\xi \in \finf(\cG,\omega)$.  Fix $x \in M$, $u \in \pi^{-1}(x)$ and suppose that $0 \neq E := \omega(\xi_u) \in \fg_\nu$, so $E$ is in the $G_0$-orbit of some element $\sfE_j$ associated to the $H$-sequence for the TSOC (see Lemma \ref{L:TSOC-sl2}).  If (CM.3') holds, then the geometry is flat on an open set $U \subset M$ with $x \in \overline{U}$.
 \end{cor}

 \begin{example}[$C_\ell / P_{\ell-1}$ and $B_\ell / P_\ell$]
  \begin{footnotesize}
 \[
 \begin{array}{|c|c|c|c|c|c|c|c} \hline
 G/P & \mbox{Range} & \#\mbox{orb} & \mbox{$H$-sequence for the TSOC} & w \in W^\fp & \begin{tabular}{c} $\fg_0$-lowest weight \\
 $-w\cdot\lambda$  in root notation \end{tabular} \\ \hline\hline
 C_\ell / P_{\ell-1} & \ell \geq 3 & \ell-1 & \sfH_j = \sfZ_j \,\, (1 \leq j \leq \ell-1) & (\ell-1,\ell) & [-2,...-2,1,0] \\
 & \ell = 3 & & & (21) & [1,2,-1] \\ \hline
 B_\ell / P_\ell  & \ell \geq 3 & \floor{\frac{\ell}{2}} &\sfH_j = \sfZ_{2j}\quad (1 \leq j \leq \floor{\frac{\ell}{2}}) & (\ell,\ell-1) &
 \left\{ \begin{array}{ll} [-1,-2,...,-2,-1,1], & \ell \geq 4;\\{} [-1,0,3] & \ell=3 \end{array} \right. \\ \hline
 \end{array}
 \]
 \end{footnotesize}

 The listed $w \in W^\fp$ parametrize irreducible components of $H^2_+(\fg_-,\fg)$.  For $C_\ell / P_{\ell-1}$, (CM.3') holds for $\sfH_{\ell-1}$.  When $\ell=3$, for geometries with $\kappa_H$ concentrated in the $(21)$-component, then $\sfH_1$ is also valid.  For $B_\ell / P_\ell$, $\sfH_{\frac{\ell}{2}}$ satisfies (CM.3') when $\ell \geq 4$ is even.  When $\ell=3$, $\sfH_1$ satisfies (CM.3').
 \end{example}

 %%%%%%%%%%%%%%%%%%%%%%%%%%%%%%%%%%%%%%%%%%%%%%%%%%

 \subsection{$|1|$-graded geometries}

 \begin{thm} \label{T:1-graded} Let $(\mathcal{G} \stackrel{\pi}{\to} M,\omega)$ be a (real or complex) normal $|1|$-graded parabolic geometry of type $(G,P)$ with $G$ simple.  If $\fg$ is real, we further assume that $\fg_\bbC$ is simple.  Suppose there exists a symmetry $0 \neq \xi \in \finf(\cG, \omega)$ such that at some $u \in \cG$, $\omega_u(\xi)$ belongs to the {\bf open} $G_0$-orbit in $\fg_1$. (In particular, $\bX = \pi_*(\xi)$ has vanishing 1-jet at $x = \pi(u) \in M$.)  Then the geometry is flat on a neighbourhood $U \subset M$ with $\bar{U} \ni x$.
 \end{thm}

 \begin{proof}
 In the complex case, the result follows from Table \ref{F:1-graded}: If $m$ is the number of $G_0$-orbits in $\bbP(\fg_1)$, then $\sfH_m(\mu) \geq 0$, where $\mu = -w\cdot \lambda$.  For example, for $C_\ell / P_\ell$, $\mu = [-2,...,-2,-1,1]$ (in root notation), $m=\ell$ and $\sfH_\ell = 2\sfZ_\ell$ satisfies $\sfH_\ell(\mu) = +2$.  Thus, (CM.3') is satisfied.

 In the real case, if $E \in \fg_1$ lies in the open orbit, then its complexification also lies in the open orbit.  Hence, the validity of (CM.3') follows from that of the complex case.
 \end{proof}

 \begin{footnotesize}
 \begin{table}[h]
 \[
 \begin{array}{|c|c|c|c|c|c|c|} \hline
 G/P & \mbox{Range} & \#\mbox{orb} & \mbox{$H$-sequence for the TSOC} & w \in W^\fp & \begin{tabular}{c} $\fg_0$-lowest weight\\
 $-w\cdot\lambda$  in root notation \end{tabular}\\ \hline\hline
 A_\ell / P_1 & \ell \geq 2 & 1 & \sfH_1 = \sfZ_1 + \sfZ_\ell & (12) &
 \left\{\begin{array}{ll}
 [3,1], & \ell = 2;\\ {}
 [2,0,-1,...,-1], & \ell \geq 3
 \end{array} \right.\\ \hline
 A_\ell / P_2 & \ell \geq 3 & 2 & \sfH_1 = \sfZ_1 + \sfZ_\ell, \,\, \sfH_2 = \sfZ_2 + \sfZ_{\ell-1} &
 (21) & [1,2,-1,...,-1]\\
 &&&& (23) &
 \left\{ \begin{array}{ll}
 [-1,2,1], & \ell = 3;\\ {}
  [-1,1,0,-1,...,-1], & \ell \geq 4
  \end{array} \right.
 \\ \hline
 A_\ell / P_k & 3 \leq k \leq \ceil{\frac{\ell}{2}} & k & \sfH_j = \sfZ_j + \sfZ_{\ell+1-j}\,\, (1 \leq j \leq k) &
 (k,k-1) & -(\alpha_1 + ...  + \alpha_\ell) + 2\alpha_k + \alpha_{k-1}\\
 &&&&(k,k+1) & -(\alpha_1 + ...  + \alpha_\ell) + 2\alpha_k + \alpha_{k+1}
 \\\hline
 B_\ell / P_1 & \ell \geq 2 & 2 & \sfH_1 = \sfZ_2, \,\, \sfH_2 = 2\sfZ_1 & (12) &
 \left\{ \begin{array}{ll}
 [3,1], & \ell = 2;\\ {}
 [2,0,-2,...,-2], & \ell \geq 3
 \end{array} \right.\\ \hline
 D_\ell / P_1 & \ell \geq 4 & 2 & \sfH_1 = \sfZ_2, \,\, \sfH_2 = 2\sfZ_1 & (12) & [2,0,\overbrace{-2,...,-2}^{\ell-4},-1,-1]\\ \hline
 C_\ell / P_\ell & \ell \geq 3 & \ell &
 \begin{array}{l}
  \sfH_j = \sfZ_j \, (1 \leq j \leq \ell-1),\\
  \sfH_\ell = 2\sfZ_\ell
  \end{array} & (\ell,\ell-1) & [-2,...,-2,-1,1] \\ \hline
 D_\ell / P_\ell & \ell \geq 5 & \floor{\frac{\ell}{2}} &
 \begin{array}{c} \sfH_j = \sfZ_{2j}\, (1 \leq j \leq \floor{\frac{\ell}{2}} -1),\\
 \sfH_{\floor{\frac{\ell}{2}}} = \left\{ \begin{array}{ll}
 \sfZ_{\ell-1} + \sfZ_\ell, & \ell \mbox{ odd};\\
 2\sfZ_\ell, & \ell \mbox{ even}
 \end{array} \right.
 \end{array} & (\ell,\ell-2) & [-1,-2,...,-2,-1,-1,1]
  \\ \hline
 E_6 / P_6 & - & 2 & \sfH_1 = \sfZ_2, \,\, \sfH_2 = \sfZ_1 + \sfZ_6 & (65) & [-1,-2,-2,-3,-1,1]\\ \hline
 E_7 / P_7 & - & 3 & \sfH_1 = \sfZ_1,\,\, \sfH_2 = \sfZ_6, \,\, \sfH_3 = 2\sfZ_7 & (76) & [-2,-2,-3,-4,-3,-1,1]\\ \hline
 \end{array}
 \]
 \caption{Data associated with $|1|$-graded geometries (with $G$ simple)}
  \label{F:1-graded}
 \end{table}
 \end{footnotesize}

  We note that these results were established (using a different approach) in \cite[Cor.\ 4.2]{MN2015} (see also further discussion in \S \ref{S:isolation}) with the exception of $E_6 / P_6$ and $E_7 / P_7$ cases.  (In the real case, these admit two real forms each corresponding to $E$I, $E$IV for $E_6$ and $E$V, $E$VII for $E_7$ in terms of Table B.4 of \cite{CS2009}.  Note that the Bourbaki simple root ordering was not followed there.)

 For projective geometry ($A_\ell / P_1$), there is only a single orbit.
For conformal geometry ($B_\ell / P_1$ and $D_\ell / P_1$) in any signature, the openness condition can be omitted since $\sfH_1(\mu) \geq 0$.

%Quaternionic geometry has type $\operatorname{PGL}(\ell+1,\mathbb{H})/P$ with $P$ being a stabilizer of a quaternionic line. It has complexification $A_{2\ell+1}/P_2$, but now TSOC has length 2 (and $H_1$ evaluates negatively on  $-w\cdot\lambda$). However for quaternionic real version there is only one $G_0$-orbit in $\mathbb{P}\fg_1=\mathbb{H}P^{\ell-1}$. Thus only $H_2$ shall be evaluated and this gives a positive number.

C-projective geometry is a real geometry with type $\SL(\ell+1,\bbC)_\bbR/P$, where $P$ is the  stabilizer of a complex line.  (Here, $\fg_\bbC$ is {\em not} simple.)  The complexification is of type $A_\ell / P_1 \times A_{\ell'} / P_{1'}$.  (Use primes to denote the second factor.)  There is only a single $G_0$-orbit in $\bbP(\fg_1) = \bbC\bbP^\ell$.  In $(\fg_1)_\bbC$, we can use the representative element $\sfH_1 = \sfZ_1 + \sfZ_\ell + \sfZ_{1'} + \sfZ_{\ell'}$.  There are three irreducible components of $H^2_+(\fg_-,\fg)$.  (See \cite{CGMN2015} and see \cite[Table 1]{KMT2016} for the (minus) lowest weights expressed in weight notation.)  Evaluating $\sfH_1$ on all lowest weights, we verify (CM.3').

 %%%%%%%%%%%%%%%%%%%%%%%%%%%%%%%%%%%%%%%%%%%%%%%%%%

 \subsection{Non-Yamaguchi-rigid, torsion-free parabolic geometries}
 \label{S:notorsion}

 \begin{thm}\label{thm:notorsion}
  Let $(\cG \stackrel{\pi}{\to} M, \omega)$ be any Yamaguchi-nonrigid, torsion-free (regular, normal) parabolic geometry of type $(G,P)$, where $G$ is simple. Let $0 \neq \xi \in \finf(\cG,\omega)$ with $\bX = \pi_*(\xi)$ having vanishing 1-jet at some $x \in M$.  Then the geometry is flat on an open set $U \subset M$ with $x \in \overline{U}$.
 \end{thm}

 \begin{proof}
 The classification of Yamaguchi-nonrigid, torsion-free geometries is given in Appendix \ref{S:nr-torfree}.  We give the corresponding data in Table \ref{F:tor-free}.  The following $|1|$-graded geometries were excluded since they are given in Table \ref{F:1-graded}: $A_\ell / P_1$ ($\ell \geq 2$), $A_\ell / P_2$ ($\ell \geq 3$): $(21)$-branch, $B_\ell / P_1$ ($\ell \geq 2$), $D_\ell / P_1$ ($\ell \geq 4$).  In all cases, all elements of the $H$-sequence for the TSOC act with a non-negative eigenvalue on the lowest weight of $H^2_+(\fg_-,\fg)$.
 \end{proof}

 \begin{footnotesize}
 \begin{table}[h]
 \[
 \begin{array}{|c|c|c|c|c|c|c|c|c|} \hline
 G & \mbox{Range} & P & \nu & \#\mbox{orb} & \multicolumn{2}{c|}{\mbox{$H$-sequence for the TSOC}} & w \in W^\fp & \begin{tabular}{c} $\fg_0$-lowest weight\\
 $-w\cdot\lambda$  in root notation \end{tabular}\\ \hline\hline
 A_\ell & \ell=2 & P_{1,2} & 2 & 1 &
 \sfH_1 = \sfZ_1 + \sfZ_\ell & - & (12) & [3,1]\\
  & \ell \geq 3 & P_{1,2} & 2 & 1 & & - & (21) & [1,2,-1,-1,...,-1,-1]\\
  & \ell \geq 3 & P_{1,\ell} & 2 & 1 & & - & (1\ell) & [1,-1,-1,...,-1,-1,1]\\ \hline
 %%%%%
 B_\ell
  & \ell=2 & P_{1,2} & 3 & 1 & \sfH_1 = \sfZ_2 & - & (12) & [3,1]\\
  & \ell=3 & P_3 & 2 & 1 & & - & (32) & [-1,0,3]\\ \hline
 %%%%%
 C_\ell & \ell=2 & P_1 & 2 & 1 & \sfH_1 = \sfZ_1 & - & (12) & [3,0]\\
 & \ell \geq 3 & P_1 & 2 & 1 & & - & (12) & [2,-1,-2,-2,...,-2,-2,-1]\\
 & \ell \geq 3 & P_2 & 2 & 2 & & \sfH_2 = \sfZ_2 & (21) & [1,2,-2,...,-2,-2,-1]\\
 & \ell \geq 3 & P_{1,2} & 4 & 1 & & - & (21) & [1,2,-2,...,-2,-2,-1]\\ \hline
 G_2 & -& P_1 & 3 & 1 & \sfH_1 = \sfZ_2 & - & (12) & [4,0]\\ \hline
 \end{array}
 \]
 \caption{Data associated with Yamaguchi-nonrigid, torsion-free, non-$|1|$-graded parabolic geometries}
 \label{F:tor-free}
 \end{table}
 \end{footnotesize}

 %%%%%%%%%%%%%%%%%%%%%%%%%%%%%%%%%%%%%%%%%%%%%%%%%%

 \subsection{General geometries}
 \label{S:gen}

  \begin{thm} \label{T:gen} Let $G$ be a real or complex simple Lie group, $P \subsetneq G$ a parabolic subgroup, and $\fg = \fg_{-\nu} \op ... \op \fg_\nu$ the associated grading.  In the real case, require that $\fg_\bbC$ be simple. Suppose that:
  \begin{enumerate}
  \item[(i)] $(\mathcal{G}\stackrel{\pi}{\to} M,\omega)$ is a regular, normal parabolic geometry of type $(G,P)$, except for:
  \[
  A_\ell/P_{s,s+1},\ 2 \leq s < \frac{\ell}{2},\ w=(s+1,s); \qquad B_\ell / P_\ell, \quad \ell \geq 5 \mbox{ odd},\ w=(\ell,\ell-1),
  \]
  i.e.\ require the component of $\kappa_H$ corresponding to the indicated $w \in W^\fp(2)$ to be trivial.

 \item[(ii)] there exists $0 \neq \xi \in \finf(\cG,\omega)$ such that at some $u \in \cG$, $\omega_u(\xi)$ lies in the {\bf open} $G_0$-orbit of $\fg_\nu$.  (In particular, the 1-jet of ${\rm\bf X}=\pi_*(\xi)$ vanishes at $x = \pi(u) \in M$.)
 \end{enumerate}
  Then the geometry is flat on an open neighbourhood $U \subset M$ with $\bar{U} \ni x$.
 \end{thm}

 \begin{proof} It suffices to consider Yamaguchi-nonrigid geometries.  Let $m$ be the number of $G_0$-orbits in $\bbP(\fg_\nu)$.  Since $\omega_u(\xi)$ lies in the open orbit, we need to verify that $\sfH_m(\mu) \geq 0$ for $\mu = -w\cdot \lambda$ when $w = (jk) \in W^\fp(2)$ and $\sfZ(\mu) > 0$.  All such $w$ are in \cite[Tables 9--13]{KT2014}.  From \eqref{E:mu}, we have
 \[
 \sfH_m(\mu) = -\sfH_m(\lambda) + (r_j + 1 - c_{kj}(r_k+1)) \sfH_m(\alpha_j) + (r_k+1)\sfH_m(\alpha_k).
 \]
 \begin{itemize}
 \item $m > 1$: The $|1|$-graded cases are complete, so it remains to consider the $(G,P)$ and $\sfH_m$ in the table in Proposition \ref{P:NYR}.  Aside from the listed exceptions, we verify $\sfH_m(\mu) \geq 0$ directly.
 \item $m=1$: From Section \ref{S:H-seq}, $\sfH_1 = \sum_{i \in I_\sfc} \sfZ_i$ and hence $\sfH_1(\lambda) = 2$.  If $j \in I_\sfc$ (contact node), then $r_j \geq 1$ and $\sfH_1(\mu) \geq 0$ (since $c_{kj} \leq 0$).  This is the case for all parabolic contact geometries.  Similarly, if $k$ is a contact node, then $\sfH_1(\mu) \geq 0$.

 Suppose that neither $j$ nor $k$ is a contact node, so $\sfH_1(\mu) < 0$.  However, a quick scan through \cite[Tables 12--13]{KT2014} yields no such $(G,P)$ satisfying $m=1$.
 \end{itemize}
 Thus, $\sfH_m(\mu) \geq 0$ has been verified, and the result follows.
 \end{proof}
 
 In particular, geometries with precisely one $G_0$-orbit in $\bbP(\fg_\nu)$ satisfy the hypotheses of Theorem \ref{T:gen}.  These include projective structures ($A_\ell / P_1$), 2nd order ODE (systems) ($A_\ell / P_{1,2}$), $(2,3,5)$-distributions $(G_2 / P_1)$, $(3,6)$-distributions $(B_3 / P_3)$, parabolic contact structures, and many torsion-free geometries from Theorem \ref{thm:notorsion} -- see Table \ref{F:tor-free}.  The case of parabolic contact structures essentially follows from \cite{CM2013a}.

%%%%%%%%%%%%%%%%%%%%%%%%%%%%%%%%%%%%%%%%%%%%%%%%%%

 \section{Discussion and outlook}
 \label{S:conclusion}

 Let us discuss some results and conjectures, related to the subject of this paper.

 \subsection{Non-flat points}
 
 For parabolic geometries, the local vanishing of the full curvature $\kappa$ is equivalent to the local vanishing of the harmonic curvature $\kappa_H$.  Fixing a point $x \in M$, the vanishing of $\kappa$ at $x$ implies the vanishing of $\kappa_H$ at $x$, but the converse is not true.  Consequently, it is natural to ask whether our notion of a non-flat point $x$, defined to mean $\kappa_H(x) \neq 0$, used in the hypothesis of Theorem \ref{Th-A} can be weakened to mean $\kappa(x) \neq 0$ instead.  Indeed, the statement
 \begin{align} \label{E:conjecture}
 0 \neq \bX \in \cS, \quad j^1_x \bX = 0 \qRa \kappa(x) = 0
 \end{align}
 follows from \cite[Thm.1.1]{MN2015} for irreducible ($|1|$-graded) parabolic geometries.  We conjecture that \eqref{E:conjecture} holds for any parabolic geometry of type $(G,P)$ with $G$ simple.  

 While $\kappa$ makes sense for any Cartan geometry, $\kappa_H$ may not having meaning.  The study of the validity of \eqref{E:conjecture} outside the parabolic setting would also be worthwhile.

 %%%%%%%%%%%%%%%%%%%%%%%%%%%%%%%%%%%%%%%%%%%%%%%%%%

 \subsection{Higher order fixed points}

 Example \ref{Ex:path-geom} (coming from \cite{KT2014}) exhibits an instance of a non-flat geometry with a higher-order fixed point.  Nevertheless, the prolongation-rigidity results of \cite{KT2014} imply that if $0 \neq \phi \in H^2_+(\fg_-,\fg)$ (in particular, if $\phi = \kappa_H(u) \neq 0$), then the Tanaka prolongation algebra $\fa^\phi$ (see Section \ref{S:main}) will often have trivial positive part.\footnote{The exceptions admitting non-trivial positive prolongation are classified in \cite[Table 4]{KT2014}, and $A_\ell / P_{1,2}$ geometry from Example \ref{Ex:path-geom} appears in this list.}  For such geometries, we anticipate that the existence of a higher-order fixed point (not necessarily only those whose isotropy is in the top-slot) would imply rigidity results, i.e.\ vanishing of $\kappa$ on $U$ with $\bar{U} \ni x$.  A similar combination of the techniques developed here together with those of \v{C}ap--Melnick and Melnick--Neusser should prove useful in this study.

 %%%%%%%%%%%%%%%%%%%%%%%%%%%%%%%%%%%%%%%%%%%%%%%%%%

 \subsection{Isolation of critical points}
 \label{S:isolation}

Another important aspect considered in \cite{CM2013a} is the isolation of higher order fixed points of an automorphism (critical points of a symmetry). While the rigidity criterion considered in \S\ref{Ss:hovs} was only sufficient (for 1-jet determinacy of symmetries at a critical point), their criterion for such a critical point being smoothly isolated is both necessary and sufficient. Our results allow us to efficiently assess this criterion for symmetries with vanishing 1-jet.

If $E=\omega_u(\xi)\in\fp$ is the isotropy of the symmetry ${\rm\bf X}=\pi_*(\xi)$ at $x\in M$, then by \cite[Proposition 2.5]{CM2013a} the dimension of the set of critical points of the same isotropy type is equal to the dimension of the centralizer $C_\fg^-(E) = \{ X \in \fg_- \st [X,E] = 0 \}$. In particular, the fixed point $x$ is smoothly isolated if and only if $C_\fg^-(E)=0$.  The space $C_\fg^-(E)$ for $G_0$-conjugacy classes of elements in $\fg_\nu$ was described in \S\ref{Ss:62}. In particular, if $\beta_1,...,\beta_m$ is the TSOC, consider $\sfE_j = e_{\beta_1} + ... + e_{\beta_j}$ as in Lemma \ref{L:TSOC-sl2}.  Then $X \in C_\fg^-(\sfE_j)$ if and only if $X$ is a sum of root vectors $e_{-\alpha}$ with  $\alpha \in \Delta(\fg_+) \cap \{ \beta_1,...,\beta_j \}^\perp$.  (Since $\beta_i \in \fg_\nu$, then $\beta_i + \alpha \not\in \Delta$.  Thus, from root strings, $\beta_i - \alpha \not\in \Delta$ if and only if $\langle \beta_i, \alpha \rangle = 0$.) Hence, $C_\fg^-(\sfE_j)\neq0$ for $j<m$, but $C_\fg^-(\sfE_m)$ can be either zero (as in $C_\ell/P_{\ell-1}$, see Example \ref{E:CP2}) or nonzero (as in $C_\ell/P_{\ell-1,\ell}$, for example).  When $j=m$, Lemma \ref{L:max} indicates that $\alpha \not\in \Delta(\fg_\nu)$.  

For $|1|$-graded geometries, this forces $\alpha \in \Delta(\fg_0)$ and $C_\fg^-(\sfE_m) = 0$. Hence, smooth isolation of the fixed point is equivalent to having isotropy in the open $G_0$-orbit of $\fg_1$.  In view of this, our Theorem \ref{T:1-graded} and \cite[Prop.\ 4.1]{MN2015} become equivalent.  In \cite[Prop.\ 4.5]{MN2015}, the authors also treat the case when the isotropy lies in the minimal $G_0$-orbit (using different techniques).

 %%%%%%%%%%%%%%%%%%%%%%%%%%%%%%%%%%%%%%%%%%%%%%%%%%

 \subsection{Linearization of symmetries at fixed points}
 \label{Ss:Lin}

For several geometries all symmetries are proved linearizable around fixed points, provided the geometry is not locally flat (clearly 1-jet determinacy is a necessary condition for linearization). This always happens by Bochner's technique in the case the action of the transformation group is proper, see \cite{Bochner1945} and the recent survey \cite{Isa}.

This also happens, due to Obata--Ferrand theorem and its local versions, for Riemannian, Lorent\-zian and
partially for general pseudo-Riemannian conformal structures, see \cite{Frances,FM}. Namely,
if the structure $[g]$ is not locally conformally flat near $x\in M$ and a conformal symmetry ${\rm\bf X}$ vanishes at $x$, then (in the general pseudo-Riemannian case, under the additional assumption that the generated local group acts semisimply) this vector field is linearizable.

Similarly, for an analytic Levi-nondegenerate hypersurface-type CR-structure that is non-spherical at $x$, a holomorphic field $\bX$ vanishing at $x$ is necessarily linearizable by Ezhov's theorem \cite{Ezh1986}. There are even stronger statements that the whole stability group is linearizable in the pseudoconvex \cite{KL}, Lorentzian \cite{Ezh1986} and some other cases (but this does not hold for all Levi-signatures \cite{Ezh1988}).

Supported by these and some other examples, one might expect that for a non-flat parabolic geometry of type $(G,P)$ (with $G$ simple), any symmetry ${\rm\bf X}$ could be proved linearizable around each fixed point $x$.  This is however not true -- our Example \ref{Ex:path-geom} yields a counterexample:

 \begin{prop}
The symmetry $\bS^{(1)}$ for the submaximally symmetric path geometry given in Example \ref{Ex:path-geom} is 1-jet determined, yet it is not linearizable around the fixed point 0.
 \end{prop}

 \begin{proof}
We consider the case $n=2$ for simplicity of exposition (no restriction of generality).
Then $M=\bbR^5({\rm\bf x})$, ${\rm\bf x}=(t,x_1,x_2,p_1,p_2)$, and
 $$
\bS^{(1)}=t^2\partial_t+tx_1\partial_{x_1}+(\tfrac12x_1^3+t x_2)\partial_{x_2}+(x_1-t p_1)\partial_{p_1}+(x_2+\tfrac32x_1^2p_1-t p_2)\partial_{p_2}.
 $$
At 0, the linear part is the vector field $\ell_{\bS^{(1)}}=x_1\partial_{p_1}+x_2\partial_{p_2}$. To conjugate $\bS^{(1)}$ to $\ell_{\bS^{(1)}}$ consider a coordinate transformation $\Psi({\bf\rm x})={\bf\rm x}+q_2({\bf\rm x})+\dots$, where $q_2$ is the quadratic part and we do not indicate the higher order terms. The inverse transformation has the form $\Psi^{-1}({\bf\rm x})={\bf\rm x}-q_2({\bf\rm x})+\dots$, and it is straightforward to check that whatever $q_2$ is, the term $t^2$ in the $\partial_t$-component of $\bS^{(1)}$ is not affected in the pushforward $\Psi_*\bS^{(1)}$. Consequently the field $\bS^{(1)}$ is not even formally linearizable.
 % Alternatively, it is easy to check that there is only one smooth integral of $\bS^{(1)}$ around 0, while there 
 % are four such for $\ell_\bS^{(1)}$ (three linear and one quadratic integrals: $t,x_1,x_2,x_1p_2-x_2p_1$).
 \end{proof}

 %%%%%%%%%%%%%%%%%%%%%%%%%%%%%%%%%%%%%%%%%%%%%%%%%%

 \subsection{Dimension of the stability group}

 Here we discuss applications of jet-determination to the description of admissible dimensions of stability groups (or infinitesimally, isotropy algebras), which is closely related to the problem of classifying geometric structures with large symmetry \cite{KT2014,Isa}.  Below we will consider only real parabolic geometries of type $(G,P)$ with $G^\bbC=\SL(3,\bbC)$.

 \'E.\ Cartan \cite{Car1932} proved that for a Levi-nondegenerate analytic hypersurface $M^3 \subset \bbC^2$, the stability group at a non-umbilic point is finite and has at most two elements. Beloshapka \cite{B1980} strengthened this by proving that at any non-spherical point, the stability group has dimension either 0 or 1.
 
 The split-form counterpart to CR-geometry in dimension 3 is the geometry of scalar 2nd order ODE $y'' = F(x,y,y')$ modulo point transformations, studied by S.\ Lie, R.\ Liouville, A.\ Tresse and \'E.\ Cartan (see the survey \cite{BK-ODE}).  This is the underlying structure for an $\SL(3,\bbR) / P_{1,2}$ geometry.  Flatness of this geometry is equivalent to point trivializability of the ODE.  Symmetries are vector fields on $M^3 = J^1(\bbR,\bbR)$ (with coordinates $\bx = (x,y,p)$) prolonged from $J^0(\bbR,\bbR) = \bbR^2$ (with coordinates $(x,y)$) that preserve the pair of line fields $\langle\partial_x+p\partial_y+F(x,y,p)\partial_p\rangle$ and $\langle\partial_p\rangle$.

 \begin{example}\label{Ex:ODE}
 For the examples below, $0$ is not a regular point in the sense of \cite{KT2014}.
 \begin{align*}
 \begin{array}{c|c|c}
 \mbox{ODE} & \mbox{Symmetry algebra} & \mbox{Isotropy dimension at $0$}\\ \hline\hline
 y''=(xy'-y)^3 & x\partial_y + \partial_p, \, x\partial_x-y\partial_y-2p\partial_p, \, y\partial_x-p^2\partial_p & 2 \mbox{ (not linearizable)}\\ \hline
 y'' = (y')^4 & \partial_x,\, \partial_y, \, 3x\partial_x+2y\partial_y-p\partial_p & 1 \mbox{ (linearizable)}
 \end{array}
 \end{align*}
  Notice $y\partial_x - p^2\partial_p$ is not linearizable.  This gives an even simpler example for the discussion of \S\ref{Ss:Lin}.
 \end{example}

 \begin{thm}
Suppose that the ODE $y''=F(x,y,y')$ is not point trivializable on any open domain. Then the
dimension of the isotropy algebra at any point ${\rm\bf x}$ can be either 0 or 1 or 2.
 \end{thm}

 \begin{proof} For $G/P = \SL(3,\bbR) / P_{1,2}$, we have a contact gradation on $\fg$, while $\fp = \fg_0 \op \fg_1 \op \fg_2$ has dimensions $(2,2,1)$.  Let $(\cG \to M, \omega)$ correspond to the ODE geometry on $M$.  From the discussion in Section \ref{S:main}, fixing $u \in \pi^{-1}(\bx)$, we can regard the isotropy algebra as a graded subalgebra $\fk = \fs_{\geq 0} \subset \fp$.
 
 By the jet-determination result Theorem \ref{T:gen}, we have $\fk \cap \fg_2 = 0$, so $\fk \subset \fg_0 \op \fg_1$. Since $[\fg_1,\fg_1] = \fg_2$, then $\fk \cap \fg_1$ is at most 1-dimensional.

 Assume $\dim(\fk \cap \fg_0) = 2$.  Then the grading element $\sfZ\in\fk$. We shall deduce that the geometry is locally flat around point ${\rm\bf x}$.  Since $\fs_0$ (in particular, $\sfZ$) annihilates $\kappa_H(\bx)$, we must have $\kappa_H(\bx) = 0$.  Let $\bZ\in\fX(M)$ be the symmetry corresponding to $\sfZ$. This vector field has a critical point at ${\rm\bf x}$ with the spectrum of its linear part $\{-1,-1,-2\}$. Thus, $\bx$ is an asymptotically stable sink for ${\rm\bf Z}$. Now $\kappa_H$ is constant along the flow of $\bZ$ that has the attractor at $\bx$.
Hence, by continuity $\kappa_H$ vanishes in a neighborhood of $\bx$, which contradicts the assumption.  Thus, $\dim(\fk \cap \fg_0) \leq 1$.

Thus, $\dim\,\fk\in\{0,1,2\}$. All values are realized as follows from Example \ref{Ex:ODE}.
 % (and the case of generic ODE)
 \end{proof}

 \begin{remark}
The same arguments allow one to reprove Beloshapka's result \cite{B1980} without appealing to the Chern--Moser normal form.
 \end{remark}
 
 Recall that $y'' = F(x,y,y')$ is the (unparametrized) geodesic equation of a projective connection $[\nabla]$, i.e.\ an $\SL(3,\bbR) / P_1$ geometry, if and only if $F(x,y,p)$ is cubic in $p$.  (See also \cite[(5.10)]{KT2014}.)
 
 \begin{example} Let $(x^0,x^1) = (x,y)$ and $\nabla_{\partial_{x^a}} \partial_{x^b} = \Gamma^c_{ab} \partial_{x^c}$ with $\Gamma^c_{ab} = \Gamma^c_{(ab)}$.
  \begin{align*}
 \begin{array}{c|c|c}
 \mbox{ODE} & \mbox{Non-vanishing $\Gamma^c_{ab}$ for $\nabla$} & \mbox{Projective symmetry algebra} \\ \hline\hline
 y'' = (xy' - y)^3 & 
 \begin{array}{l} 
 \Gamma^0_{11} = x^3, \quad
 \Gamma^1_{11} = 3x^2 y, \\
 \Gamma^0_{00} = 3xy^2, \quad
 \Gamma^1_{00} = y^3
 \end{array}
 & x\partial_y, \, x\partial_x-y\partial_y, \, y\partial_x \\ \hline
 y'' = (xy'-y)^2 &
 \begin{array}{ccc}
 \Gamma^1_{11} = -x^2, \,\, \Gamma^0_{00} = -2xy, \,\, \Gamma^1_{00} = -y^2
 \end{array} & x\partial_y, \, x\partial_x - 2y\partial_y \\ \hline
 y'' = xy y' & \Gamma^0_{00} = xy & x\partial_x - 2y\partial_y 
 \end{array}
 \end{align*}
 In each of these cases, the projective symmetry algebra agrees with the isotropy algebra at $0$.  (Again, $0$ is not a regular point.) Since 3 is the well-known submaximal symmetry dimension, we conclude that 3 is the maximal isotropy algebra dimension for non-flat 2-dimensional projective structures.  Moreover, isotropy algebras of dimension 0, 1, 2, or 3 are all realizable.
 \end{example}

Using the same ideas, further results on dimensions of the isotropy algebras for other geometries can be obtained, but this lies outside the scope of this paper.

 %%%%%%%%%%%%%%%%%%%%%%%%%%%%%%%%%%%%%%%%%%%%%%%%%%

 \appendix

 \section{Dynkin diagram recipes and Kostant's theorem}
 \label{S:DD}

 Let $\fg$ be a complex semisimple Lie algebra with Borel subalgebra $\fb$, Cartan subalgebra $\fh$, simple roots $\{ \alpha_i \} \subset \fh^*$, dual basis $\{ \sfZ_i \} \subset \fh$, coroots $\alpha_i^\vee = \frac{2\alpha_i}{\langle \alpha_i, \alpha_i \rangle}$, and fundamental weights $\{ \lambda_i \} \subset \fh^*$.  Let $\langle \cdot, \cdot \rangle$ be the symmetric bilinear form on $\fh^*$ induced from the Killing form $B$ on $\fg$.  Let:
 \begin{itemize}
 \item $\fD(\fg)$ be the corresponding Dynkin diagram.
 \item Given a parabolic subalgebra $\fb \subseteq \fp \subsetneq \fg$ with corresponding index set $I_\fp$, let $\fD(\fg,\fp)$ be the marked diagram obtained by putting crosses on $\fD(\fg)$ corresponding to $I_\fp$.
 \item Given a weight $\mu \in \fh^*$, inscribe the coefficient $r_i = \langle \mu, \alpha_i^\vee \rangle$ on the $i$-th node of $\fD(\fg)$ or $\fD(\fg,\fp)$.  Denote this by $\fD(\fg,\mu)$ or $\fD(\fg,\fp,\mu)$ respectively.  Its {\em support} consists of all connected components containing at least one nonzero coefficient over a node.  This will be denoted by using $0$ as a superscript on the previous diagram.
 \end{itemize}

  Let $\fg = \fg_{-\nu} \op ... \op \fg_\nu$ be the grading induced by $\sfZ = \sum_{i \in I_\fp} \sfZ_i$.

 \begin{enumerate}[label=\textbf{(R.\arabic*)}]

 \item {\bf Structure of $\fg_0$}: Removing crossed nodes from $\fD(\fg,\fp)$ yields $\fD(\fg_0^{ss})$ and $\dim(\fz(\fg_0)) = |I_\fp|$.
 \item {\bf $\fg_{-1}$ as a $\fg_0^{ss}$-module}: Fix a crossed node $i$.  For any adjacent uncrossed node $j$, inscribe the multiplicity of the bond between $i$ and $j$ if the bond is directed from $i$ to $j$.  Otherwise, inscribe a 1.  Do this for each crossed node $i$ to get the decomposition of $\fg_{-1}$ into irreducibles.
 \item {\bf $\fg_1$ as a $\fg_0^{ss}$-module}:  Since $(\fg_{-1})^* \cong \fg_1$, we apply the duality involution\footnote{The duality involution is trivial except in the following cases: $A_\ell$ for $\ell \geq 2$, $D_\ell$ for $\ell$ odd, or $E_6$.  In these cases, it is the unique non-trivial automorphism of the Dynkin diagram.} to the $\fg_0^{ss}$ irreps in $\fg_{-1}$ to obtain the $\fg_0^{ss}$-decomposition for $\fg_1$.
 \end{enumerate}

 Now suppose that $\fg$ is simple and $\lambda$ its highest weight, given below.
 \begin{table}[h]
 \[
 \begin{array}{|c||c|c|c|c|c|c|c|c|c|} \hline
 G &
 \begin{array}{c} A_\ell \\ (\ell \geq 1) \end{array} &
 \begin{array}{c} B_\ell \\ (\ell \geq 3) \end{array} &
 \begin{array}{c} C_\ell \\ (\ell \geq 2) \end{array} &
 \begin{array}{c} D_\ell \\ (\ell \geq 4) \end{array} & G_2 & F_4 & E_6 & E_7 & E_8\\ \hline
 \lambda & \lambda_1 + \lambda_\ell & \lambda_2 & 2\lambda_1 & \lambda_2 & \lambda_2 & \lambda_1 & \lambda_2 & \lambda_1 & \lambda_8\\ \hline
 \begin{tabular}{c} $\lambda$ in root\\ notation \end{tabular} & 11\cdots11 & 12\cdots2 & 2\cdots21 & 12\cdots211 & 32 & 2342 & 122321 & 2234321 & 23465432\\ \hline
 \end{array}
 \]
 \caption{Highest roots in terms of fundamental weights and simple roots}
  \label{F:hw}
 \end{table}

 \begin{enumerate}[label=\textbf{(R.\arabic*)}] \setcounter{enumi}{3}
 \item {\bf Top-slot $\fg_\nu$ as a $\fg_0^{ss}$-module}: Remove crosses from $\fD(\fg,\fp,\lambda)$ to get the diagram $\fT(\fg_0^{ss},\lambda)$.
 \item {\bf Effective part of the $\fg_0^{ss}$-action on $\fg_\nu$}:  Restrict to the support $\fT^0(\fg_0^{ss},\lambda)$ of $\fT(\fg_0^{ss},\lambda)$.
 \end{enumerate}

 \begin{example}[$D_6 / P_{1,4}$] $\Dsix{xwwxww}{0,1,0,0,0,0} \leadsto \Atwo{ww}{1,0}$.  Here, $\nu = 3$ and $\fg_0^{ss} = A_2 \times A_1 \times A_1$, but $A_1 \times A_1$ acts trivially on $\fg_\nu$.  The effective part is given by the $A_2$-action with the above weight.
 \end{example}

 The extended Dynkin diagram $\widetilde\fD(\fg)$ augments $\fD(\fg)$ by a single node corresponding to $-\lambda$ and corresponding bonds.  Refer to the node(s) adjacent to $-\lambda$ as the {\em contact node(s)}.  Equivalently, if $\lambda = \sum_i r_i \lambda_i$, these are nodes $j$ for which $r_j \neq 0$.  (See Table \ref{F:hw}.)  Mark these on $\fD(\fg,\lambda)$ with a \mydiamond.

 \begin{enumerate}[label=\textbf{(R.\arabic*)}] \setcounter{enumi}{5}
 \item {\bf Simple roots orthogonal to $\lambda$}:  All nodes in $\fD(\fg)$ other than those marked with \mydiamond.
 \item {\bf The subalgebra $\fg(\lambda) \subset \fg$}:  Remove all \mydiamond{} from $\fD(\fg,\lambda)$ to obtain a diagram $\fD_\lambda(\fg,\fp)$.  (Equivalently, if $\fq \subset \fg$ is the parabolic subalgebra corresponding to crossing all \mydiamond{}, then $\fg(\lambda)$ is the semisimple part of $\fq$.)
 \item {\bf The ideal $\fl_\fp(\lambda) \subset \fg(\lambda)$}: From $\fD_\lambda(\fg,\fp)$, remove all cross-free connected components to obtain $\fD_\lambda^0(\fg,\fp)$ corresponding to an ideal $\fl_\fp(\lambda) \subset \fg(\lambda)$.  (All other ideals of $\fg(\lambda)$ lie in $\fg_0$.) Let $\fp(\lambda) = \fl_\fp(\lambda) \cap \fp$.
 \end{enumerate}

 Using Lemma \ref{L:max}, we obtain the following recipe (see also Example \ref{E:E7P7}):

  \begin{enumerate}[label=\textbf{(R.\arabic*)}] \setcounter{enumi}{8}
 \item {\bf Top-slot orthogonal cascade}:
 \begin{enumerate}
 \item Start with $\fD(\fg,\fp,\lambda)$, where $\fg$ is simple, $\fp \subsetneq \fg$ is parabolic, and $\lambda = \max \Delta(\fg)$.
 \begin{itemize}
 \item Termination condition: $\fT^0(\fg_0^{ss},\lambda) = \emptyset$ or $\fT^0(\fg_0^{ss},\lambda) \cong \Add{wwww}{1,0,0,0}$.
 \end{itemize}
 Remove all $\mydiamond$.  Remove all connected components without crosses.  Result: $\fD(\fl_\fp(\lambda),\fp(\lambda))$.
 Now iterate using $\fD(\fl_\fp(\lambda),\fp(\lambda),\mu)$ for the new $\fD(\fg,\fp,\lambda)$, where $\mu = \max \Delta(\fl_\fp(\lambda))$.
 \item For each diagram of the sequence produced in (a), write the highest root (use Table \ref{F:hw}).  Write the corresponding root in the initial $\fg$ by putting a zero coefficient for all nodes that carried a $\mydiamond$ in the previous steps.
 \end{enumerate}
 \end{enumerate}

By Kostant's version of the Bott--Borel--Weil theorem \cite{CS2009,Kos1963}, the $\fg_0$-irreps $\bbU_\mu \subset H^2(\fg_-,\fg)$ are parametrized by the length two elements $w \in W^\fp(2)$ of the {\em Hasse diagram} $W^\fp$, which is a distinguished subset of the Weyl group $W$ of $\fg$. Let $\sigma_i$ denote the simple reflections in $W$, defined by $\sigma_i(\alpha) = \alpha - \langle \alpha, \alpha_i^\vee \rangle \alpha_i$.  Then all $w \in W^\fp(2)$ are of the form $w = (jk) := \sigma_j \circ \sigma_k$, where
 \begin{itemize}
 \item[(P.1)] $j \in I_\fp$ (i.e.\ a crossed node), and
 \item[(P.2)] $j \neq k$ with either $k \in I_\fp$ or $k$ is adjacent to $j$ in the Dynkin diagram of $\fg$.
 \end{itemize}
 The submodule $\bbU_\mu \subset H^2(\fg_-,\fg)$ has lowest $\fg_0$-weight $\mu = -w\cdot\lambda$, where $\cdot$ denotes the affine action of $W$.  By \cite[Prop. 3.2.14 (1)]{CS2009}, $w \cdot 0 = -\alpha_j - \sigma_j(\alpha_k)$.  If $\lambda = \sum_i r_i \lambda_i$ (so $r_i = \langle \lambda, \alpha_i^\vee \rangle$) then
 \begin{align}
  \mu &= -w\cdot \lambda = -w(\lambda) - w\cdot 0 = -\sigma_j(\lambda - r_k \alpha_k) + \alpha_j + \sigma_j(\alpha_k) \label{E:mu} \\
  &= -\lambda + (r_j + 1) \alpha_j + (r_k+1)(\alpha_k - c_{kj} \alpha_j), \nonumber
 \end{align}
 where $c_{kj} = \langle \alpha_k, \alpha_j^\vee \rangle$ are entries of the Cartan matrix.

 Using the natural $\fg_0$-module isomorphism between $H^2(\fg_-,\fg)$ and $\ker(\Box) \subset \bigwedge^2 \fg_-^* \ot \fg$, where $\Box$ is the Kostant Laplacian, Kostant identified the following lowest $\fg_0$-weight vector for $\bbU_{-w \cdot \lambda}$:
 \begin{align} \label{E:vw}
 e_{\alpha_j} \wedge e_{\sigma_j(\alpha_k)} \otimes e_{w(-\lambda)},
 \end{align}
 where $e_\gamma$ denotes a root vector corresponding to the root $\gamma \in \Delta$.  Since $\lambda \in \Delta$, then $w(-\lambda) \in \Delta$.

 %%%%%%%%%%%%%%%%%%%%%%%%%%%%%%%%%%%%%%%%%%%%%%%%%%

 \section{Classification of Yamaguchi-nonrigid, torsion-free parabolic geometries}
 \label{S:nr-torfree}

 For any regular, normal parabolic geometry of type $(G,P)$, the harmonic curvature $\kappa_H$ takes values in $H^2_+(\fg_-,\fg)$.  The geometry is {\em Yamaguchi-rigid} if $H^2_+(\fg_-,\fg) = 0$.  All such geometries are automatically flat.  In the non-rigid case, the geometry is {\em torsion-free} if the curvature $\kappa$ takes values in the $P$-submodule $\bigwedge^2 \fp_+ \ot \fp \subset \bigwedge^2 \fp_+ \ot \fg$.  We will prove the following classification result.

  \begin{thm} Let $G$ be a complex simple Lie group.  All Yamaguchi-nonrigid, torsion-free (regular, normal) parabolic geometries of type $(G,P)$ are given by:
 \[
 \begin{array}{ccccl}
 G & \mbox{Range} & P & w \in W^\fp & \mbox{Description of some real forms} \\ \hline
 A_\ell & \ell \geq 2 & P_1 & (12) & \mbox{projective structures}\\
 & \ell \geq 3 & P_2 & (21) & \mbox{$(2,\ell-1)$-Segr\'e structures}\\
 & \ell \geq 2 & P_{1,2} & (21) & \mbox{torsion-free second order ODE system in $(\ell-1)$-dep. vars}\\
 & \ell \geq 3 & P_{1,\ell} & (1\ell) & \mbox{integrable Legendrian contact structures}\\
 B_\ell & \ell \geq 2 & P_1 & (12) & \mbox{odd-dimensional conformal structures}\\
 B_2 & - & P_{1,2} & (12) & \mbox{scalar 3rd order ODE with vanishing W\"unschman invariant}\\
 B_3 & - & P_3 & (32) & \mbox{$(3,6)$-distributions}\\
 C_\ell & \ell \geq 2 & P_1 & (12) & \mbox{contact projective structure}\\
 & \ell \geq 3 & P_2 & (21) & \mbox{split quaternionic contact structure}\\
 & \ell \geq 3 & P_{1,2} & (21) & \mbox{contact path geometry}\\
 D_\ell & \ell \geq 4 & P_1 & (12) & \mbox{even-dimensional conformal structures}\\
 G_2 & -& P_1 & (12) & \mbox{$(2,3,5)$-distributions}
 \end{array}
 \]
 \end{thm}

 Let $\lambda$ be the highest weight of $\fg$.  From \eqref{E:vw}, if the geometry is torsion-free, then $\kappa_H$ takes values in the direct sum of those $\bbU_{-w\cdot \lambda} \subset H^2_+(\fg_-,\fg)$, where $w \in W^\fp(2)$ and the grading element $\sfZ$ satisfy
 \begin{align} \label{E:nn-hom}
 \sfZ(w(-\lambda)) \geq 0.
 \end{align}
 The component of $\kappa$ of lowest homogeneity is harmonic (see \cite[Thm.\ 3.1.12]{CS2009} for a precise statement), so if \eqref{E:nn-hom} holds, then the geometry is torsion-free. The following appeared in \cite[Lemma 4.3.2]{KT2014}.

 \begin{lemma}
 Let $\fg$ be complex simple.  If $w \in W^\fp(2)$, and $w(-\lambda) \in \Delta^+$, then $G/P$ is one of $A_2 / P_1$, $A_2 / P_{1,2}$, $B_2 / P_1$, or $B_2 / P_{1,2}$.
 \end{lemma}

Thus, it suffices to study the case $\sfZ(w(-\lambda)) = 0$.  The rank 2 cases are easily settled by hand: $B_2 / P_2 = C_2 / P_1$ and $G_2 / P_1$ are the only cases with $\sfZ(w(-\lambda)) = 0$.  So let $\ell := \rnk(G) \geq 3$.

 %Suppose $\fg$ has Killing form $\langle \cdot, \cdot \rangle$, Cartan matrix $c_{ab} = \langle \alpha_a, \alpha_b^\vee \rangle$ (where $\alpha^\vee = \frac{2\alpha}{\langle \alpha, \alpha \rangle}$), and fundamental weights $\{ \lambda_a \}$.
 Since $\lambda = \sum_a r_a \lambda_a$ is a dominant integral weight, then $r_a := \langle \lambda, \alpha_a^\vee \rangle \in \bbZ_{\geq 0}$.  (Indeed $r_a \in \{ 0,1,2 \}$ always.)  Hence, for $w = (jk) \in W^\fp(2)$,
 \begin{align*}
 w(\lambda) = \sigma_j\sigma_k(\lambda)
 &= \sigma_j(\lambda - r_k \alpha_k) = \lambda - r_j \alpha_j - r_k (\alpha_k - c_{kj} \alpha_j)
 = \lambda - (r_j - r_k c_{kj}) \alpha_j - r_k \alpha_k.
 \end{align*}
 Since $j \in I_\fp$, then $\sfZ(\alpha_j) = 1$, so $\sfZ(w(-\lambda)) = 0$ if and only if
 \begin{align} \label{E:zero-hom}
 \sfZ(\lambda) = r_j - r_k c_{kj} + r_k \sfZ(\alpha_k).
 \end{align}
 Since $I_\fp \neq \emptyset$, then $\sfZ(\lambda) \geq |I_\fp| \geq 1$, so $r_j = r_k = 0$ is impossible.  Thus, $r_j \geq 1$ or $r_k \geq 1$.  If $r_a \geq 1$, node $a$ will be referred to as a {\em contact node}.  An important property is:
 \begin{itemize}
 \item[(P.3)]   If $\fg$ is not of type A, there is a {\em unique} contact node $a$ and we have $\sfZ_a(\lambda) = 2$.  If $\fg$ is of type C, then $r_a = 2$, and otherwise $r_a = 1$.
 \end{itemize}
 Using (P.1), (P.2) (from Appendix \ref{S:DD}), and (P.3), we extract the implications of \eqref{E:zero-hom}.

%Also, let $\{ Z_a \}$ be dual to $\{ \alpha_b \}$, i.e. $Z_a(\alpha_b) = \delta_{ab}$, so $Z = \sum_{a \in I_\fp} Z_a$.  Lastly, we remark that the classifications below are up to Dynkin diagram automorphisms.

 \begin{enumerate}
 \item \framebox{$r_j \geq 1$}. ($j$ is a contact node)

 \begin{itemize}
 \item Type A: Since $\lambda = \lambda_1 + \lambda_\ell$, then we may assume $j=1$, so $r_j = 1$ and $\sfZ_j(\lambda) = 1$.
 \begin{itemize}
 \item If $r_k \geq 1$, then $k = \ell \geq 3$, $r_k = 1$, $c_{kj} = 0$, and $k \in I_\fp$.  Thus, \eqref{E:zero-hom} implies $\sfZ(\lambda) = 2$.  Since $\lambda = \alpha_1 + ... + \alpha_\ell$ with $j,k \in I_\fp$, then  $G/P = A_\ell / P_{1,\ell}$.
 \item Otherwise, if $r_k = 0$, then \eqref{E:zero-hom} implies $\sfZ(\lambda) = 1$, so $G/P = A_\ell / P_1$.
 \end{itemize}
 \item Other types: We must have $r_k = 0$, so \eqref{E:zero-hom} implies $\sfZ(\lambda) = r_j$.
 \begin{itemize}
 \item Type C: Since $\lambda = 2\lambda_1 = 2\alpha_1 + ... + 2\alpha_{\ell-1} + \alpha_\ell$, then $j=1$, $r_j = 2$, and $\sfZ(\lambda) = \sfZ_j(\lambda) = 2$.  Thus, $G/P = C_\ell / P_1$.
 \item Not type A,C: $r_j = 1$, so $\sfZ(\lambda) = 1$ forces $|I_\fp| = 1$.  However, since $j$ is the contact node, then $\sfZ_j(\lambda) = 2$, so $\sfZ(\lambda) = r_j$ is impossible.
 \end{itemize}

 \end{itemize}

 \item \framebox{$r_j = 0$ and $r_k \geq 1$}. ($k$ is a contact node)

 \begin{itemize}
 \item Type A: We may assume $k=1$, so $r_k = 1$ and \eqref{E:zero-hom} implies $\sfZ(\lambda) = \sfZ(\alpha_k) - c_{kj}$.
 \begin{itemize}
 \item $c_{kj} = 0$: Then $I_\fp \supset \{ j, k \}$, so $2 \leq \sfZ(\lambda) = \sfZ(\alpha_k) \leq 1$, a contradiction.
 \item $c_{kj} \neq 0$: $2 = j \in I_\fp$ and $c_{kj} = -1$.  If $k \not\in I_\fp$, then $\sfZ(\alpha_k) = 0$ and $G/P = A_\ell / P_2$.  If $k \in I_\fp$, then $\sfZ(\alpha_k) = 1$ and $G/P = A_\ell / P_{1,2}$.
 \end{itemize}

 \item Type C: $k=1$, $r_k = 2$, and \eqref{E:zero-hom} implies $\sfZ(\lambda) = 2 \sfZ(\alpha_1) - 2 c_{1j}$.
 \begin{itemize}
 \item $j\neq 2$: $c_{1j} = 0$, so $1 \leq \sfZ(\lambda) = 2\sfZ(\alpha_1)$ implies $1 \in I_\fp$ and $\sfZ(\lambda) = 2$. Since $\lambda = 2(\sum_{i=1}^{\ell-1} \alpha_i) + \alpha_\ell$, then $I_\fp = \{ 1 \}$.  But $j \in I_\fp$, so $j=k=1$, a contradiction.
 \item $j=2$: $c_{1j} = -1$, so $\sfZ(\lambda) = 2 + 2\sfZ(\alpha_1)$.  If $1 \not\in I_\fp$, then $G/P = C_\ell / P_2$.  If $1 \in I_\fp$, then $G/P = C_\ell / P_{1,2}$.
 \end{itemize}

 \item Not type A,C: By (P.3), $r_k = 1$ and $\sfZ_k(\lambda) = 2$, so \eqref{E:zero-hom} implies
 $\sfZ(\lambda) = - c_{kj} + \sfZ(\alpha_k)$.
 \begin{itemize}
 \item $k \in I_\fp$: Since $j \in I_\fp$, $\sfZ_j(\lambda) \geq 1$, so $ 3 \leq \sfZ_j(\lambda) + \sfZ_k(\lambda) \leq \sfZ(\lambda) = 1-c_{kj}$.  Since $\ell = \rnk(G) \geq 3$, then $G$ must be doubly-laced with $c_{kj} = -2$ and $\sfZ(\lambda) = 3$.
 \begin{itemize}
 \item Type B: $r_k=1$ and $c_{kj} = -2$ implies $k=2$ and $j=\ell=3$.  But then $\sfZ(\lambda) \geq \sfZ_2(\lambda) + \sfZ_3(\lambda) \geq 4$ (since $\lambda = \alpha_1 + 2\alpha_2 + 2\alpha_3$), a contradiction.
 \item Type F: $r_k=1$ so $k=1$, but then $c_{kj} \neq -2$, a contradiction.
 \end{itemize}

 \item $k \not\in I_\fp$: Here, \eqref{E:zero-hom} implies $\sfZ(\lambda) = - c_{kj}$ with $k$ the unique contact node.
 \begin{itemize}
 \item Type B: $k=2$ and $\lambda = \alpha_1 + 2(\sum_{i=2}^\ell \alpha_i)$.  If $c_{kj} = -1$, then $\sfZ(\lambda) = \sfZ_j(\lambda) = 1$, so $G/P$ is $B_\ell / P_1$.  Otherwise, if $c_{kj} = -2$, then $G/P = B_3 / P_3$.
 \item Type D: $k=2$ and $\lambda = \alpha_1 + 2(\sum_{i=2}^{\ell-2} \alpha_i) + \alpha_{\ell-1} + \alpha_\ell$, so $G/P = D_\ell / P_1$.
 \item Type E: $\sfZ(\lambda) = -c_{kj} = 1$, so $G \neq E_8$.  For $E_6$, $k=2$ and $j=1$, but then $c_{kj} = 0$.  For $E_7$, $k=1$ and $j=7$, but then $c_{kj} = 0$.  Both are contradictions.
 \item Type F: $k=1$, so $j=2$, and $3 = \sfZ_2(\lambda) \leq \sfZ(\lambda) = 1$, a contradiction.
 \end{itemize}
 \end{itemize}

 \end{itemize}
 \end{enumerate}

\end{document}